\documentclass[a4paper,reqno,11pt,twoside]{article}

\pdfoutput=1

\usepackage{amsthm,enumerate,amssymb,amsmath,nccmath,mathtools,pgfmath,tikz,genyoungtabtikz,needspace,scalerel}
\usepackage{hyperref}
\hypersetup{colorlinks=true,linkcolor=blue,citecolor=red,urlcolor=blue}
\usepackage{scalerel}
\usepackage{mathrsfs}
\usepackage{blkarray}
\usepackage{tikz,tikz-cd}
\usepackage[a4paper,margin=0.8in]{geometry}
\usepackage{cleveref}
\usepackage{comment}

\newtheorem{thm}{Theorem}[section]
\newtheorem{lem}[thm]{Lemma}
\newtheorem{prop}[thm]{Proposition}
\newtheorem{cor}[thm]{Corollary}
\newtheorem{ex}[thm]{Example}
\newtheorem{defn}[thm]{Definition}
\newtheorem{rmk}[thm]{Remark}

\numberwithin{equation}{section}

\crefname{thm}{Theorem}{Theorems}
\crefname{prop}{Proposition}{Propositions}
\crefname{lem}{Lemma}{Lemmas}
\crefname{cor}{Corollary}{Corollaries}
\crefname{conj}{Conjecture}{Conjectures}
\crefname{defn}{Definition}{Definitions}
\crefname{rmk}{Remark}{Remarks}
\crefname{section}{Section}{Sections}
\crefname{subsection}{Subsection}{Subsections}
\crefname{ex}{Example}{Examples}

\Crefname{thm}{Theorem}{Theorems}
\Crefname{prop}{Proposition}{Propositions}
\Crefname{lem}{Lemma}{Lemmas}
\Crefname{cor}{Corollary}{Corollaries}
\Crefname{conj}{Conjecture}{Conjectures}
\Crefname{defn}{Definition}{Definitions}
\Crefname{rmk}{Remark}{Remarks}
\Crefname{section}{Section}{Sections}
\Crefname{subsection}{Subsection}{Subsections}
\Crefname{ex}{Example}{Examples}

\DeclareMathOperator{\im}{im}
\DeclareMathOperator{\spn}{span}
\DeclareMathOperator{\res}{res}
\DeclareMathOperator{\std}{Std}
\DeclareMathOperator{\colstd}{ColStd}

\DeclareMathOperator{\het}{ht}

\DeclareMathOperator{\gar}{Gar}

\newcommand{\Mod}[1]{\ (\mathrm{mod}\ #1)}
\newcommand\sesqui{1}

\newcommand\one{m{+}1}

\newcommand\enminusone{n{-}1}
\newcommand\ten{10}
\newcommand\eleven{11}
\newcommand\twelve{12}
\newcommand\thirteen{13}
\newcommand\fourteen{14}

\newcommand\emplusone{{m}+{1}}
\newcommand\emplustwo{m{+}2}
\newcommand\emplusthree{m{+}3}
\newcommand\emplusfour{m{+}4}
\newcommand\emplusiminusone{m{+}i{-}1}
\newcommand\emplusi{m{+}i}
\newcommand\emplusiplusone{m{+}i{+}1}
\newcommand\emplusiplustwo{m{+}i{+}2}
\newcommand\aye{A_i}
\newcommand\ayezero{A_0}
\newcommand\aone{a_1}
\newcommand\aem{a_m}
\newcommand\btwo{b_2}
\newcommand\bemplusone{b_{m{+}1}}

\newcommand\emset{\varnothing}
\newcommand\gr{\Yfillcolour{lightgray}}
\newcommand\wh{\Yfillcolour{white}}
\newcommand{\Psid}[2]{\Psi\hspace{-4pt}\underset{\scriptscriptstyle{#2}}{\overset{\scriptscriptstyle{#1}}{\downarrow}}}
\newcommand{\Psiu}[2]{\Psi\hspace{-4pt}\overset{\scriptscriptstyle{#2}}{\underset{\scriptscriptstyle{#1}}{\uparrow}}}
\newcommand{\s}[2]{s\hspace{-4pt}\underset{\scriptscriptstyle{#2}}{\overset{\scriptscriptstyle{#1}}{\downarrow}}}
\newcommand{\su}[2]{s\hspace{-4pt}\overset{\scriptscriptstyle{#2}}{\underset{\scriptscriptstyle{#1}}{\uparrow}}}

\newcommand\ttr{\mathtt{R}}
\newcommand\tts{\mathtt{S}}
\newcommand\ttt{\mathtt{T}}
\newcommand\ttu{\mathtt{U}}

\newcommand\ttw{\mathtt{W}}
\newcommand\ttx{\mathtt{X}}
\newcommand\tty{\mathtt{Y}}
\newcommand\ttz{\mathtt{Z}}

\tikzstyle{dotted}=                  [dash pattern=on \pgflinewidth off 2pt]
\tikzstyle{dashed}=                  [dash pattern=on 3pt off 3pt]
\tikzstyle{dashdotted}=              [dash pattern=on 3pt off 2pt on \the\pgflinewidth off 2pt]
\tikzstyle{densely dashdotted}=      [dash pattern=on 3pt off 1pt on \the\pgflinewidth off 1pt]
\tikzstyle{loosely dashdotted}=      [dash pattern=on 3pt off 4pt on \the\pgflinewidth off 4pt]

\allowdisplaybreaks

\author{Louise Sutton\thanks{Present address: National University of Singapore, 10 Lower Kent Ridge Road, Singapore 119076}\\\normalsize Queen Mary University of London\\ Mile End Road\\ London E1 4NS, UK \\\texttt{\normalsize matloui@nus.edu.sg}}

\newcommand\runninghead[1]{\gdef\@runninghead{#1}}

\newcommand\auth{Louise Sutton}

\newcommand\toptitle{\date{}\maketitle\markboth{\auth}{\@runninghead}\pagestyle{myheadings}}

\begin{document}

\title{Specht modules labelled by hook bipartitions I}

\runninghead{Specht modules labelled by hook bipartitions I}

\toptitle

\begin{abstract}
Brundan, Kleshchev and Wang endow the Specht modules $S_{\lambda}$ over the cyclotomic Khovanov--Lauda--Rouquier algebra $\mathscr{H}_n^{\Lambda}$ with a homogeneous $\mathbb{Z}$-graded basis. In this paper, we begin the study of graded Specht modules labelled by hook bipartitions $((n-m),(1^m))$ in level $2$ of $\mathscr{H}_n^{\Lambda}$, which are precisely the Hecke algebras of type $B$, with quantum characteristic at least three. We give an explicit description of the action of the Khovanov--Lauda--Rouquier algebra generators $\psi_1,\dots,\psi_{n-1}$ on the basis elements of $S_{((n-m),(1^m))}$. Introducing certain Specht module homomorphisms, we construct irreducible submodules of these Specht modules, and thereby completely determine the composition series of Specht modules labelled by hook bipartitions.
\end{abstract}

\section{Introduction}

The main problem in the representation theory of the symmetric group $\mathfrak{S}_n$ is to completely understand the modular irreducible $\mathbb{F}\mathfrak{S}_n$-modules. James' combinatorial construction \cite{James} of the ordinary irreducible $\mathfrak{S}_n$-representations as Specht modules $S_{\lambda}$, labelled by \emph{partitions $\lambda$ of $n$}, is well known. The modular irreducible representations $D_{\mu}$ arise as heads of Specht modules labelled by \emph{$e$-regular partitions $\mu$ of $n$}. In general, the \emph{decomposition numbers} for $\mathbb{F}\mathfrak{S}_n$, which are the multiplicities $[S_{\lambda}:D_{\mu}]$ of irreducible modules $D_{\mu}$ arising as composition factors of Specht modules $S_{\lambda}$, for all partitions $\lambda$ of $n$ and for all $e$-regular partitions $\mu$ of $n$, are unknown.

In odd characteristic $p$, Peel \cite{P} studied the decomposition numbers for the \emph{hook representations} of the symmetric group; a hook representation is a Specht module $S_{(n-m,1^m)}$ labelled by a \emph{hook partition} $(n-m,1^m)$, for $0\leqslant{m}<{n}$. When $p$ does not divide $n$, Peel determined that every hook representation is, in fact, an irreducible representation of $\mathbb{F}\mathfrak{S}_n$. Otherwise, when $p$ divides $n$, Peel introduced a Specht module homomorphism $\Theta_m:S_{(n-m,1^m)}\rightarrow{S_{(n-m-1,1^{m+1})}}$ over $\mathfrak{S}_n$ and showed that $S_{(n-m,1^m)}$ has the composition series 
\[0\subset{\ker \Theta_m}\subset{S_{(n-m,1^m)}},\] for $0<m<n-1$. The irreducible trivial and sign representations are well known to be $S_{(n)}$ and $S_{(1^n)}$, respectively, and thus, we know that the rows of the decomposition matrix of $\mathbb{F}\mathfrak{S}_n$ labelled by hook representations have the following form:
\[
\begin{blockarray}{cccccccccc}
\begin{block}{c(ccccc|cccc)}
S_{(n)} & 1 &&&&  									&  &    & 	\\
S_{(n-1,1)} & 1 & 1 && \text{\huge$0$} & 			&  &    & 	\\
S_{(n-2,1^2)} & & 1 & 1 &&  						&  &  \text{\Huge$0$}   & 	\\
\vdots &&&\ddots &\ddots &  						&  & & 	\\
S_{(2,1^{n-2})} && \text{\huge$0$} && 1 & 1   		&  &    & 	\\
S_{(1^n)} &&&&& 1  									&  &    & 	\\
\end{block}
\end{blockarray}\:.
\]

Let $e$ be the smallest positive integer such that $q$ is a primitive $e^{\text{th}}$ root of unity, for a non-trivial invertible element $q\in\mathbb{F}^{\times}$, where we set $e=\infty$ if no such integer exists. We will refer to $e$ as the \emph{quantum characteristic}. The symmetric group algebras can be generalised to the larger family of algebras, the Iwahori--Hecke algebras $\mathscr{H}_{\mathbb{F},q}(\mathfrak{S}_n)$  of type $A$, which are dependent on $q$ as well as the ground field $\mathbb{F}$. These algebras are $q$-deformations of the symmetric group algebras, and thus, we recover the symmetric group algebra when $q=1$. Dipper and James \cite{DJ} showed that a correspondence exists between the representation theory of Iwahori--Hecke algebras of type $A$ at an $e^{\text{th}}$ root of unity and the modular representation theory of the symmetric group, and thus we can study the former representations to shed light on the latter ones.

Using the representation of the quantum affine algebra $U_v(\widehat{sl_n})$, for every pair of partitions $\lambda$ and $e$-regular $\mu$, Lascoux, Leclerc and Thibon \cite{LLT} introduced the polynomials $d_{\lambda,\mu}(v)$ with integer coefficients. They conjectured that these polynomials are $v$-analogues of decomposition numbers for Hecke algebras of type $A$ at a complex $e^{\text{th}}$ root of unity, called \emph{$v$-decomposition numbers}. This conjecture was subsequently proved by Ariki \cite{A2}. Thus $d_{\lambda,\mu}(1)$ are the usual decomposition numbers for Hecke algebras of type $A$. These polynomials appear as coefficients of the canonical basis elements of $U_v(\widehat{sl_n})$. Working solely with the canonical bases of the quantum affine algebra, Chuang, Miyachi and Tan \cite{CMT} introduce a $v$-analogue to Peel's work on hook representations. These $v$-decomposition numbers $d_{(n-m,1^m),\mu}=[S_{(n-m,1^m)}:D_{\mu}]_v$, where $\mu$ is $e$-regular, are monic monomials when non-zero.

In fact, Ariki \cite{A2} proved Lascoux, Leclerc and Thibon's conjecture for a larger family of algebras, the \emph{Ariki--Koike algebras} or equivalently the \emph{cyclotomic Hecke algebras}, introduced by Ariki and Koike \cite{AK} and further developed by Brou\'{e} and Malle \cite{BM} and Ariki \cite{A1}. Each cyclotomic Hecke algebra is associated to a complex reflection group $\left(\mathbb{Z}/l\mathbb{Z}\right)\wr\mathfrak{S}_n$ of type $G(l,m,n)$ in the Shephard--Todd classification \cite{ST}. We recover the Iwahori--Hecke algebras of type $A$ when $l=1$, in particular the symmetric group algebras when $e=\text{char}(\mathbb{F})$, and we recover the Iwahori--Hecke algebras of type $B$ when $l=2$.

Khovanov and Lauda \cite{KL}, and independently Rouquier \cite{Rou}, introduced an even larger family of algebras, the \emph{Khovanov--Lauda--Rouquier algebras} $\mathscr{H}_n$, which are naturally $\mathbb{Z}$-graded. Astonishingly, Brundan and Kleshchev \cite{BK} showed that each cyclotomic quotient $\mathscr{H}_n^{\Lambda}$ of the Khovanov--Lauda--Rouquier algebra is isomorphic to a cyclotomic Hecke algebra of type $A$. This remarkable set of papers motivated the study of \emph{graded} representation theory of Khovanov--Lauda--Rouquier algebras, and, in particular, of the symmetric group algebras.

In \cite{BKW}, Brundan, Kleshchev and Wang non-trivially $\mathbb{Z}$-grade Specht modules over $\mathscr{H}_n^{\Lambda}$, which yields a recursive combinatorial formula for their \emph{graded} dimensions. Thus we can study \emph{graded} Specht modules and their corresponding \emph{graded} decomposition numbers $d_{\lambda,\mu}=[S_{\lambda}:D_{\mu}]_v$, which encode grading shifts of their composition factors. Brundan and Kleshchev in \cite{BK3} show that these graded decomposition numbers are the same as the $v$-decomposition numbers as mentioned above, and provide a generalised graded analogue of Lascoux, Leclerc and Thibon's conjecture.

These gradings provide a deeper structure to the representation theory of the symmetric group, and more generally, to the representation theory of the Khovanov--Lauda--Rouquier algebras. By studying a set of graded Specht modules that have particularly nice presentations, namely Specht modules labelled by hook bipartitions $((n-m),(1^m))$, we aim to provide insight into the challenging problem of determining graded decomposition numbers. We draw on inspiration from \cite{CMT} and \cite{P} to study the structure of Specht modules labelled by hook bipartitions for $\mathscr{H}_n^{\Lambda}$ with quantum characteristic at least three.

This paper is structured as follows. In \cref{sec:Hecke algebras}, we set up basic notation and introduce the Khovanov--Lauda--Rouquier algebras $\mathscr{H}_n$ and their cyclotomic quotients $\mathscr{H}_n^{\Lambda}$. We recall combinatorial theory in \cref{sec:comb} that is integral in our approach of the representation theory of $\mathscr{H}_n^{\Lambda}$. \cref{sec:Specht} introduces our main objects of study, graded Specht $\mathscr{H}_n^{\Lambda}$-modules. We study Specht modules labelled by hook bipartitions $S_{((n-m),(1^m))}$ in \cref{sec:hookbipartSpecht}: we give an explicit description of the action of the $\mathscr{H}_n^{\Lambda}$-generators $\psi_1,\dots,\psi_{n-1}$ on the basis elements of $S_{((n-m),(1^m))}$, and introduce crucial Specht module homomorphisms. We advance in \cref{sec:compseries} by showing that certain $\mathscr{H}_n^{\Lambda}$-modules arising from these homomorphisms are irreducible, and in fact, appear as composition factors of $S_{((n-m),(1^m))}$. We end by providing the composition series for $S_{((n-m),(1^m))}$ with $e\in\{3,4,\dots\}$, which split into four cases depending on whether $\kappa_2\equiv{\kappa_1-1}\Mod{e}$ or not and depending on whether $n\equiv{\kappa_2-\kappa_1+1}\Mod{e}$ or not.

In further work, we completely determine the ungraded decomposition matrices for $\mathscr{H}_n^{\Lambda}$ comprising rows corresponding to hook bipartitions, together with their graded analogues.

\section{Hecke algebras}\label{sec:Hecke algebras}

We introduce the Khovanov--Lauda--Rouquier algebras in this section. Let $\mathbb{F}$ be an arbitrary field throughout.

\subsection{The symmetric group}

Let $\mathfrak{S}_n$ be the symmetric on $n$ letters. We denote the simple transposition $(i,i+1)$ of $\mathfrak{S}_n$ by $s_i$ for all $i\in\{1,\dots,n-1\}$. Then $\mathfrak{S}_n$ is generated by the standard Coxeter generators $s_1,\dots,s_{n-1}$, together with the identity element $1_{\mathfrak{S}_n}$. For $1\leqslant{i}\leqslant{j}\leqslant{n-1}$, we define 
\[\s{j}{i}:=s_j{s_{j-1}}\dots{s_i},\qquad\su{i}{j}:=s_i{s_{i+1}}\dots{s_j}.\]
	
We say that a \emph{reduced expression} for a permutation $\pi\in\mathfrak{S}_n$ is a minimal length word $\pi=s_{r_1}\dots{s_{r_m}}$ for $1\leqslant{r_i}<n$ and $1\leqslant{i}\leqslant{m}$. Let $\leqslant$ be the \emph{Bruhat order} on $\mathfrak{S}_n$, defined as follows. For $\pi_1,\pi_2\in\mathfrak{S}_n$, we write $\pi_1\leqslant{\pi_2}$ if there is a reduced expression for $\pi_1$ which is a subexpression of a reduced expression for $\pi_2$.

We define a \emph{shift} homomorphism of symmetric groups $\operatorname{shift}:\mathfrak{S}_{n-1}\rightarrow\mathfrak{S}_n$ by $\operatorname{shift}(s_i)=s_{i+1}$ for every $i$.

\subsection{Lie-theoretic notation}

Let $e$ be the quantum characteristic as introduced above. Define $I:=\mathbb{Z}/{e\mathbb{Z}}$. If $e$ is finite, then we identify $I$ with the set $\{0,1,\dots,e-1\}$, whereas, if $e$ is infinite, then we identify $I$ with the set of integers.

We let $\Gamma$ be the quiver with vertex set $I$ and directed edges $i\rightarrow{i+1}$ for each $i\in{I}$. If no directed edge exists between two vertices $i$ and $j$ such that $i\neq{j}$, we write $i\not\relbar{j}$. If $e$ is infinite, then $\Gamma$ is the integral linear quiver (of type $A_{\infty}$), otherwise $\Gamma$ is the cyclic quiver on $e$ vertices (of type $A_{e-1}^{(1)}$). The associated Cartan matrix $C_{\Gamma}=(c_{i,j})_{i,j\in{I}}$ is defined by
\begin{align*}
c_{i,j}:=
\begin{cases}
2 &\text{ if }i=j,\\
0 & \text{ if }j\not\relbar{i},\\
-1 & \text{ if $i\rightarrow{j}$ or $i\leftarrow{j}$,}\\
-2 & \text{ if }i\leftrightarrows{j}.
\end{cases}
\end{align*}
The notation $i\leftrightarrows{j}$ indicates that $i=j-1=j+1$, which only occurs when $e=2$.

The generalised Cartan matrix $C_{\Gamma}$ corresponds to a Kac--Moody algebra $\mathfrak{g}(C_{\Gamma})$, as given in \cite{Kac}. It follows that we have the simple roots $\{\alpha_i\mid{i\in{I}}\}$, the fundamental dominant weights $\{\Lambda_i\mid{i\in{I}}\}$, and the invariant symmetric bilinear form $(\:,\:)$ such that $(\alpha_i,\alpha_j)=c_{i,j}$ and $(\Lambda_i,\alpha_j)=\delta_{ij}$, for all $i,j\in{I}$. Let $Q_+:=\bigoplus_{i\in{I}}\mathbb{Z}_{\geqslant{0}}\alpha_i$ be the positive part of the root lattice. A \emph{root} $\alpha\in{Q_+}$ is a linear combination $\sum_{i\in{I}}a_i\alpha_i$ of its simple roots where $a_i\in\mathbb{Z}_{\geqslant{0}}$, and the \emph{height of $\alpha$} is the sum $\sum_{i\in{I}}a_i$, denoted by $\het(\alpha)$.

We now fix a \emph{level} $l\in\mathbb{N}$.
The symmetric group $\mathfrak{S}_l$ acts on the left by place permutation on the set $I^l$ of all $l$-tuples. An \emph{$e$-multicharge} of $l$ is an ordered $l$-tuple $\kappa=(\kappa_1,\dots,\kappa_l)\in{I^l}$. We define its associated dominant weight of level $l$ to be $\Lambda:=\Lambda_{\kappa_1}+\dots+\Lambda_{\kappa_l}$.

\subsection{Khovanov--Lauda--Rouquier algebras}

The Khovanov--Lauda--Rouquier algebras were discovered by Khovanov and Lauda \cite{KL}, and independently, Rouquier \cite{Rou}. Brundan and Kleshchev transformed their work to give the following presentation.

\begin{defn}\cite{BK}\label{defKLR}
Let $\alpha\in{Q_+}$ such that $\het(\alpha)=n$, and define the set
\[I^{\alpha}=\{i\in{I^n}\ | \ \alpha_{i_1}+\dots+\alpha_{i_n}=\alpha\}.\]
Then the algebra $\mathscr{H}_{\alpha}$ is defined to be the unital associative $\mathbb{F}$-algebra generated by the elements
\begin{equation*}
\{e(\mathbf{i})\:|\:\mathbf{i}\in{I^{\alpha}}\}\cup\{y_1,\dots,y_n\}\cup\{\psi_1,\dots,\psi_{n-1}\}
\end{equation*}
subject only to the following relations:

\begin{align}
e(\mathbf{i}) e(\mathbf{j}) & = \delta_{\mathbf{i},\mathbf{j}} e(\mathbf{i});
\hspace{15mm} 
\textstyle{\sum_{\mathbf{i}\in{I^{\alpha}}}} e(\mathbf{i}) = 1; \nonumber \\
y_r e(\mathbf{i}) & = e(\mathbf{i}) y_r; 
\hspace{22.2mm}
\psi_r e(\mathbf{i}) = e(s_r\mathbf{i}) \psi_r;									\label{psiidemp} \\
y_r y_s & = y_s y_r; \nonumber \\
\psi_r y_s & = y_s \psi_r
	\hspace{57.5mm} \text{if $s\neq{r,r+1}$}; 									\label{ypsicomm} \\
\psi_r \psi_s & = \psi_s \psi_r
	\hspace{57.5mm} \text{if $|r-s|>1$}; 										\label{psicomm} \\
\psi_r y_{r+1} e(\mathbf{i}) & = (y_r \psi_r + \delta_{i_r,i_{r+1}}) e(\mathbf{i}); 																	\label{psiyup}\\
y_{r+1} \psi_r e(\mathbf{i}) & = (\psi_r y_r +\delta_{i_r,i_{r+1}}) e(\mathbf{i}); 
																				\label{ypsidown} \\
\psi_r^2 e(\mathbf{i}) & = 
\begin{cases}																	\label{psipsi}
0 
	& \hspace{18.7mm}\text{if $i_r = i_{r+1}$}, \\
e(\mathbf{i}) 
	& \hspace{18.7mm}\text{if $i_{r+1} \neq i_r,i_r \pm 1$}, \\
(y_{r+1} - y_r) e(\mathbf{i}) 
	& \hspace{18.7mm}\text{if $i_r \rightarrow i_{r+1}$}, \\
(y_r - y_{r+1}) e(\mathbf{i}) 
	& \hspace{18.7mm}\text{if $i_r \leftarrow i_{r+1}$}, \\
(y_{r+1} - y_r)(y_r - y_{r+1}) e(\mathbf{i}) 
	& \hspace{18.7mm}\text{if $i_r \leftrightarrows i_{r+1}$};
\end{cases} \\
\psi_r \psi_{r+1} \psi_r e(\mathbf{i}) & =
\begin{cases}																	\label{psibraid}
(\psi_{r+1} \psi_r \psi_{r+1} + 1) e(\mathbf{i}) 
	& \text{if $i_{r+2} = i_r \rightarrow i_{r+1}$}, \\
(\psi_{r+1} \psi_r \psi_{r+1} - 1) e(\mathbf{i}) 
	& \text{if $i_{r+2} = i_r \leftarrow i_{r+1}$}, \\
(\psi_{r+1} \psi_r \psi_{r+1} - 2 y_{r+1} + y_r + y_{r+2}) e(\mathbf{i})
	&\text{if $i_{r+2} = i_r \leftrightarrows i_{r+1}$}, \\
\psi_{r+1} \psi_r \psi_{r+1} e(\mathbf{i}) 
	& \text{otherwise},
\end{cases}
\end{align}
for all admissible $\mathbf{i},\mathbf{j},r,s$.
\end{defn}

\begin{thm}\cite[Corollary 1]{BK}
The algebra $\mathscr{H}_{\alpha}$ is uniquely $\mathbb{Z}$-graded. 
\end{thm}

We now define the affine \emph{Khovanov--Lauda--Rouquier algebra} $\mathscr{H}_n$ to be the direct sum
\[\bigoplus_{\substack{\alpha\in{Q_+}\\\het(\alpha)=n}}\mathscr{H}_{\alpha},\]
so that $\mathscr{H}_n$ is also non-trivially $\mathbb{Z}$-graded.

We now introduce a \emph{shift} homomorphism of algebras corresponding to the shift homomorphism of symmetric groups as defined above, which is a special case of the homomorphisms as defined by Fayers--Speyer in \cite{FS}.
Let $\alpha,\beta\in{Q_+}$ be such that $\het(\alpha)=n$ and $\het(\beta)=n-1$. For $\mathbf{i}\in{I}^{\beta}$, we set $J_{\mathbf{i}}:=\left\{\mathbf{j}\in{I}^{\alpha}\:|\:j_{s+1}=i_s\text{ for }1\leqslant{s}\leqslant{n-1}\right\}$ and $e(\mathbf{i})^{+1}=\sum_{\mathbf{j}\in{J_i}}e(\mathbf{j})$. We now define the homomorphism $\operatorname{shift}:\mathscr{H}_{\beta}\rightarrow\mathscr{H}_{\alpha}$ by
\[
e(\mathbf{i})\mapsto e(\mathbf{i})^{+1},\quad
\psi_re(\mathbf{i})\mapsto\psi_{r+k}e(\mathbf{i})^{+1},\quad
y_re(\mathbf{i})\mapsto y_{r+k}e(\mathbf{i})^{+1}.
\]

\subsection{Cyclotomic Khovanov--Lauda--Rouquier algebras}

For a positive root $\alpha\in{Q_+}$, the cyclotomic algebras $\mathscr{H}_{\alpha}^{\Lambda}$ are defined to be the quotients of $\mathscr{H}_{\alpha}$, subject to the \emph{cyclotomic relations}
\[y_1^{(\Lambda,\alpha_{i_1})} e(\mathbf{i}) =0,\]
for all $\mathbf{i}\in{I^{\alpha}}$. These cyclotomic relations are homogeneous, so that $\mathscr{H}_{\alpha}^{\Lambda}$ inherits a non-trivial $\mathbb{Z}$-grading.
We define the \emph{cyclotomic Khovanov--Lauda--Rouquier algebra} $\mathscr{H}_n^{\Lambda}$ to be the direct sum
\[\bigoplus_{\substack{\alpha\in{Q_+}\\\het(\alpha)=n}}\mathscr{H}_{\alpha}^{\Lambda}.\]

We introduce Brundan and Kleshchev's remarkable \emph{Graded Isomorphism Theorem}, connecting the representation theory of the cyclotomic Hecke algebras with the cyclotomic Khovanov--Lauda--Rouquier algebras.

\begin{thm}\cite[Main Theorem]{BK}
If $e=\infty$ or $\text{char}(\mathbb{F})\nmid{e}$, then $\mathscr{H}_n^{\Lambda}$ is isomorphic to a cyclotomic Hecke algebra (of type $A$).
\end{thm}

In particular, $\mathscr{H}_n^{\Lambda}\cong\mathbb{F}\mathfrak{S}_n$ when $e=\text{char}(\mathbb{F})$ and $l=1$. Thus the cyclotomic Hecke algebras (of type $A$), and hence the symmetric group algebras, are non-trivially $\mathbb{Z}$-graded.

\section{Combinatorics}\label{sec:comb}

We introduce many well known combinatorial objects in this section, which we will use later to construct Specht modules for the cyclotomic Khovanov--Lauda--Rouquier algebras.

\subsection{Young diagrams and partitions}

A \emph{composition} of $n$ is a sequence $\lambda=(\lambda_1,\lambda_2,\lambda_3,\dots)$ of non-negative integers such that $\textstyle{\sum_{i=1}^{\infty}\lambda_i=n}$. For $i\geqslant{1}$, we refer to the integers $\lambda_i$ as the \emph{parts} of $\lambda$.
A \emph{partition} of $n$ is a composition $\lambda$ for which $\lambda_i\geqslant{\lambda_{i+1}}$ for all $i\geqslant{1}$. We denote the \emph{empty partition} $(0,0,\dots)$ by $\varnothing$ and define $(1^0):=\varnothing$.

We fix a positive integer $l$ and an $e$-multicharge $\kappa=(\kappa_1,\dots,\kappa_l)$.
We write $|\lambda^{(i)}|=\lambda_1^{(i)}+\lambda_2^{(i)}+\cdots$, and define an $l$-\emph{multicomposition} of $n$ to be an ordered $l$-tuple $\lambda=(\lambda^{(1)},\dots,\lambda^{(l)})$ of compositions such that $\textstyle{\sum_{i=1}^l|\lambda^{(i)}|=n}$. We refer to $\lambda^{(i)}$ as the $i^{\text{th}}$ \emph{component} of $\lambda$. When each component of an $l$-multicomposition $\lambda$ is a partition, we say that $\lambda$ is an $l$-\emph{multipartition}.
We abuse notation and also write $\varnothing$ for the \emph{empty multipartition} $(\varnothing,\dots,\varnothing)$. We denote the set of all $l$-multipartitions of $n$ by $\mathscr{P}_n^l$.

Given $l$-multicompositions $\lambda$ and $\mu$ of $n$, we say that $\lambda$ \emph{dominates} $\mu$, if
\[\sum_{i=1}^{m-1}|\lambda^{(i)}| + \sum_{j=1}^{k}\lambda_j^{(m)}
\geqslant\sum_{i=1}^{m-1}|\mu^{(i)}| + \sum_{j=1}^{k}\mu_j^{(m)},\]
for all $1\leqslant{m}\leqslant{l}$ and $k\geqslant{1}$. We write $\lambda\unrhd\mu$ to mean that $\lambda$ dominates $\mu$.

The \emph{Young diagram} of the $l$-multicomposition $\lambda=(\lambda^{(1)},\dots,\lambda^{(l)})$ is defined by

\[ [\lambda]:= \left\{ (i,j,m)\in \mathbb{N} \times \mathbb{N} \times \{1,\dots,l\}\:|\:1\leqslant{j}\leqslant{\lambda_i^{(m)}} \right\}. \]

Each element $(i,j,m)\in[\lambda]$ is called a \emph{node} of $\lambda$, and in particular, an $(i,j)$-node of $\lambda^{(m)}$. We draw the Young diagram of an $l$-multipartition as a column vector of Young diagrams $[\lambda^{(1)}],\dots, [\lambda^{(l)}]$ where $[\lambda^{(i)}]$ lies above $[\lambda^{(i+1)}]$ for all $i\geqslant{1}$. For example, $((5,3),(2^2,1))$ has the Young diagram

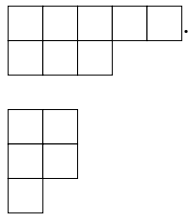
\[\gyoung(;;;;;,;;;,:,;;,;;,;).\]

The $e$-residue of a node $A=(i,j,m)$ lying in the space $\mathbb{N}\times\mathbb{N}\times\{1,\dots,l\}$ is defined by
\[\res A := \kappa_m+j-i \Mod{e}.\]

\subsection{Tableaux}

Let $\lambda=(\lambda^{(1)},\dots,\lambda^{(l)})\in\mathscr{P}_n^l$. A $\lambda$-tableau $\ttt=(\ttt^{(1)},\dots,\ttt^{(l)})$ is a bijection $\ttt:[\lambda]\rightarrow\{1,\dots,n\}$. Usually, we depict a $\lambda$-tableau $\ttt$ by inserting entries $1,\dots,n$ into the Young diagram $[\lambda]$; we say that the entry lying in node $(i,j,m)\in[\lambda]$ is the $(i,j,m)$-entry of $\ttt$, denoted $\ttt(i,j,m)$. We refer to the $\lambda^{(i)}$-tableau $\ttt^{(i)}$ as the \emph{$i^{\text{th}}$ component} of $\ttt$ for all $i\in\{1,\dots,l\}$. We say that $\ttt$ is \emph{row-standard} if the entries in each row increase from left to right along the rows of each component of $\ttt$. Similarly, we say that $\ttt$ is \emph{column-standard} if the entries in each column increase from top to bottom down the columns of each component of $\ttt$; we denote the set of all column-standard $\lambda$-tableaux by $\colstd(\lambda)$. If $\ttt$ is both row-standard and column-standard, then $\ttt$ is called \emph{standard}; we denote the set of all standard $\lambda$-tableaux by $\std(\lambda)$.

The \emph{column-initial tableau} $\ttt_{\lambda}$ is the $\lambda$-tableau whose entries $1,\dots,n$ appear in order down consecutive columns, working from left to right in components $l,l-1,\dots,1$, in turn.
For example,
\[\ttt_{((5,3),(2^2,1))}=\gyoung(;6;8;\ten;\twelve;\thirteen,;7;9;\eleven,,;1;4,;2;5,;3).\]

Given a $\lambda$-tableau $\ttt$, the symmetric group $\mathfrak{S}_n$ acts naturally on the left of $\ttt$. We define the permutation $w_{\ttt}\in\mathfrak{S}_n$ from
\[w_{\ttt} \ttt_{\lambda} =\ttt.\]
For example, if
\[\tts =\gyoung(;4;8;\ten;\eleven;\twelve,;7;9;\thirteen,,;1;5,2;6,;3),\]
then $w_{\tts} \ttt_{((5,3),(2^2,1))} = \tts$ where $w_{\tts}=(4\: 5\: 6)(11\: 13\: 12)$.

Let $\ttt$ be a $\lambda$-tableau. We write $r=\ttt(i,j,m)$ to denote that the integer entry $r$ lies in node $(i,j,m)\in[\lambda]$, and set $\res_{\ttt}(r)=\res (i,j,m)$. The \emph{residue sequence} of $\ttt$ is defined to be
\[\mathbf{i}_{\ttt}=(\res_{\ttt}(1),\dots,\res_{\ttt}(n)).\]
We set $\mathbf{i}_{\lambda}=\mathbf{i}_{\ttt_{\lambda}}$. For example, when $e=3$ and $\kappa=(0,1)$, the $3$-residues of the nodes in the Young diagram of $((5,3),(2^2,1))$ are given by
\[\gyoung(;0;1;2;0;1,;2;0;1,,;1;2,;0;1,;2),\]
so that $\mathbf{i}_{((5,3),(2^2,1))}=(1,0,2,2,1,0,2,1,0,2,1,0,1)$ and $\mathbf{i}_{\tts}=(1,0,2,0,2,1,2,1,0,2,0,1,1)$. We now define the idempotent generator of $\mathscr{H}_n^{\Lambda}$ with respect to $\ttt$ to be $e_{\ttt}:=e(\mathbf{i}_{\ttt})$.

Let $\lambda\in\mathscr{P}_n^l$ and $\ttt$ be a $\lambda$-tableau. Suppose that $\ttt(i_1,j_1,m)=r$ and $\ttt(i_2,j_2,m)=s$ (so that $r$ and $s$ both lie in the $m^{\text{th}}$ component of $\ttt$) such that $1\leqslant{r}\neq{s}\leqslant{n}$. We write $r\rightarrow_{\ttt}s$ if $i_1=i_2$ and $j_1<j_2$ and we write $r\downarrow_{\ttt}s$ if $i_1<i_2$ and $j_1=j_2$.

\begin{lem}\cite[Lemma 3.3]{BK}\label{lem:std}
Let $\lambda\in\mathscr{P}_n^l$ and $\ttt\in\std(\lambda)$. Then $s_r\ttt$ is also standard if and only if neither $r\rightarrow_{\ttt}r+1$ nor $r\downarrow_{\ttt}r+1$.
\end{lem}

We let $\lambda\in\mathscr{P}_n^l$ and now define a dominance order on $\lambda$-tableaux with respect to the Bruhat order on $\mathfrak{S}_n$. Let $\tts$ and $\ttt$ be $\lambda$-tableaux with corresponding reduced expressions $w_{\tts}$ and $w_{\ttt}$, respectively. Then we say that \emph{$\ttt$ dominates $\tts$}, written $\tts\trianglelefteq{\ttt}$, if and only if $w_{\tts}\leqslant{w_{\ttt}}$.

\section{Graded Specht modules}\label{sec:Specht}

In this section we introduce the main objects of our study, graded Specht modules, following the theory of Brundan, Kleshchev and Wang in \cite{BKW}. We will work with the \emph{dual} Specht module throughout, however, we will refer to it  as the Specht module itself for brevity, consistent with James' classical construction of Specht modules over $\mathbb{F}\mathfrak{S}_n$.

Recall that the presentation of Specht modules for $\mathbb{F}\mathfrak{S}_n$, as constructed by James, includes \emph{Garnir relations} \cite[\S 7]{James}, which are far from straightforward to write down. We will see that Specht modules as $\mathscr{H}_n^{\Lambda}$-modules must also satisfy Garnir relations that are arguably even more complicated than those for the symmetric group. We now present the combinatorics needed to define these Garnir relations; see \cite[\S 7]{KMR} for further details.

\subsection{Garnir tableaux and Garnir relations}\label{Garnir}

For $\lambda\in\mathscr{P}_n^l$, we call a node $A=(i,j,m)\in[\lambda]$ such that $(i,j+1,m)\in[\lambda]$ a \emph{\textup{(}column\textup{)} Garnir node of $\lambda$}. The \emph{\textup{(}column\textup{)} $A$-Garnir belt} $\mathbf{B}_A$ is defined to be the set of nodes
\[\mathbf{B}_A=\left\{(k,j,m)\in[\lambda]\:|\:k\geqslant{i}\right\}\cup
\left\{(k,j+1,m)\in[\lambda]\:|\:1\leqslant{k}\leqslant{i}\right\}.\]
For example, $\mathbf{B}_{(3,1,1)}$ in $((4^2,2,1^2),(2))$ is shaded in the following Young diagram
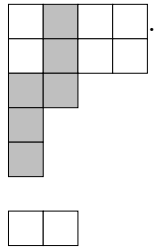
\[\gyoung(;!\gr;!\wh;;,;!\gr;!\wh;;!\gr,;;,;,;!\wh,,;;).\]

Let $r=\ttt_{\lambda}(i,j,m)$ and $s=\ttt_{\lambda}(i,j+1,m)$. We place the entries $r,r+1,\dots,s$ in $\mathbf{B}_A$ in order from top right to bottom left. The resulting $\lambda$-tableau is called the \emph{\textup{(}column\textup{)} $A$-Garnir tableau}, denoted $G_A$. The (column) $(3,1,1)$-Garnir $((4^2,2,1^2),(2))$-tableau is
\[G_{(3,1,1)}=\gyoung(;3!\gr5!\wh\eleven;\thirteen,;4!\gr6!\wh\twelve;\fourteen,!\gr8;7,;9,;\ten,,!\wh1;2)
=\su{7}{9}\su{6}{8}\su{5}{7}\ttt_{((4^2,2,1^2),(2))}.\]

A \emph{\textup{(}column\textup{)} $A$-brick} is a set of $e$ consecutive nodes
\[
\left\{(a,b,m),(a+1,b,m),\dots,(a+e-1,b,m)
\right\}\subseteq{\mathbf{B}_A}
\]
such that $\res(a,b,m)=\res A$. Suppose that there are $k$ bricks lying in the Garnir belt $\mathbf{B}_A$. If $k>0$, then we label the bricks $B_A^1,B_A^2,\dots,B_A^k$ in $\mathbf{B}_A$ from top to bottom, firstly down column $j+1$ and then down column $j$. For $e=3$, the $(3,1,1)$-Garnir belt in our running example has two bricks, labelled in the following Young diagram

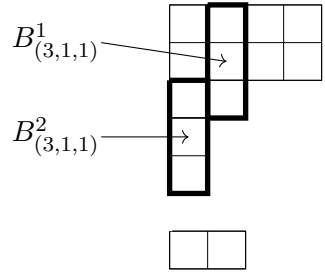
\[
\Yboxdim{1cm}
\begin{tikzpicture}[scale=.5]
\tyng(0cm,0cm,4^2,2,1^2,0,2)
\draw[white,line width=2pt](0,-4)--(0,-5);
\draw[line width=2pt](1,1)--(2,1)--(2,-2)--(1,-2)--(1,1);
\draw[line width=2pt] (0,-1)--(1,-1)--(1,-4)--(0,-4)--(0,-1);
\node at (-3,0) {$B_{(3,1,1)}^1$};
\node at (-3,-2.5) {$B_{(3,1,1)}^2$};
\draw[->](-1.8,0)--(1.5,-0.5);
\draw[->] (-1.8,-2.5)--(0.5,-2.5);
\end{tikzpicture}
\]

Let $n_A$ be the smallest number in the Garnir tableau $G_A$ in $\mathbf{B}_A$ that also lies in a brick. We define \emph{brick permutations} of $\mathfrak{S}_n$ by
\[w_A^r:=\prod_{a=n_A+e(r-1)}^{n_A+re-1} (a,a+e) \in \mathfrak{S}_n\]
for each $r\in\{1,\dots,k-1\}$. Informally, the brick permutation $w_A^r$ swaps the $r^{\text{th}}$ and $(r+1)^{\text{th}}$ bricks in $\mathbf{B}_A$. Let the \emph{\textup{(}column\textup{)} brick permutation group} be
\[\mathfrak{S}_k\cong\mathfrak{S}_A=\left\langle{w_A^1,w_A^2,\dots,w_A^{k-1}}\right\rangle
\subseteq{\mathfrak{S}_n}.\]

We let $\ttt_A$ be the $\lambda$-tableau obtained by placing the bricks $B_A^1,B_A^2,\dots,B_A^k$ successively down column $j$ and then down column $j+1$ in $[\lambda]$.
The set of \emph{$A$-Garnir $\lambda$-tableaux} is defined to be
\[\gar_A=\left\{\ttt\in\colstd(\lambda)\:|\:\ttt=w\ttt_A \text{ for a brick permutation } w\in\mathfrak{S}_A\right\}.\]
By the construction of $\gar_A$, we know that $\mathbf{i}_{\ttt}=\mathbf{i}_{G_A}$ for all $\ttt\in\gar_A$. We thus set $\mathbf{i}_A:=\mathbf{i}_{G_A}$ to be the residue sequence of every tableau lying in $\gar_A$.

In our running example, notice that $\ttt_{(3,1,1)}=\ttt_{((4^2,2,1^2),(2))}$. We now observe that there is only one brick permutation $w_{(3,1,1)}^1=\s{7}{5}\s{8}{6}\s{9}{7}$, so that $\mathfrak{S}_{(3,1,1)}$ is generated by $w_{(3,1,1)}^1$. We have $w_{(3,1,1)}^1\ttt_{((4^2,2,1^2),(2))}=G_{(3,1,1)}$, and hence

\[\gar_{(3,1,1)}=\left\{G_{(3,1,1)},\ttt_{((4^2,2,1^2),(2))}\right\}.\]

From this set of tableaux, we obtain the \emph{Garnir elements} $g_A$ for each Garnir node $A\in[\lambda]$. In general, the Garnir elements are very complicated to compute; we refer the reader to \cite[\S 7.5]{KMR} for further details. From our example, we find that the Garnir element of $(3,1,1)$ is 

\[g_{(3,1,1)}=\psi_7\psi_6\psi_5\psi_8\psi_7\psi_6\psi_9\psi_8\psi_7
e(\mathbf{i}_{(3,1,1)})-2e(\mathbf{i}_{(3,1,1)}).\]

For this paper, we specifically require the Garnir elements of Garnir nodes lying in the bipartitions $((n-m),(1^m))$ and $((n-m,1^m),\varnothing)$, which are particularly easy to find.

\subsubsection{Garnir elements of $((n-m),(1^m))$}\label{Garnir1}

Let $\lambda=((n-m),(1^m))$ and $A_i=(1,i,1)$ for $i\in\mathbb{N}$. Then the complete set of Garnir nodes of $\lambda$ is
\[\left\{A_i\:|\:1\leqslant{i}\leqslant{n-m-1}\right\}.\]
The $A_i$-Garnir belt $\mathbf{B}_{A_i}$ consists only of the two consecutive nodes $(1,i,1)$ and $(1,i+1,1)$ in the first component of $\lambda$, as shown in the following shaded Young diagram
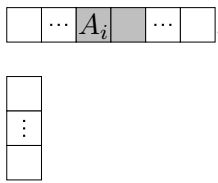
\[\gyoung(;;\hdts!\gr;\aye;!\wh;\hdts;,,;,|\sesqui\vdts,;).\]
Thus, the $A_i$-Garnir $\lambda$-tableau is 

\[G_{A_i}=\gyoungxy(3,1,;\emplusone;\hdts;\emplusiminusone!\gr\emplusiplusone;\emplusi!\wh\emplusiplustwo;\hdts;n,,;1,;2,|\sesqui\vdts,;m).\]
Notice that $\ttt_{A_i}=G_{A_i}$.
We write $G_{A_i}=s_{m+i}\ttt_{\lambda}$. For $e\geqslant{3}$, $G_{A_i}$ has no $A_i$-bricks and hence
\[
\gar_{A_i}=\left\{G_{A_i}\right\}.
\]
We set $\psi_{G_{A_i}}=\psi_{m+i}$. It follows that the Garnir element of $A_i$ is defined to be
\[
g_{A_i}=e(\mathbf{i}_{A_i})\psi_{G_{A_i}}=e(\mathbf{i}_{A_i})\psi_{m+i}=
\psi_{m+i}e(\mathbf{i}_{\lambda})\quad(\text{by \cref{psiidemp}}).
\]

\subsubsection{Garnir elements of $((n-m,1^m),\varnothing)$}\label{Garnir2}

Let $\lambda=((n-m,1^m),\varnothing)$ and $A_i=(1,i+1,1)$ for all $i\in\{0,\dots,n-m-2\}$. From the Garnir elements of Garnir nodes in $((n-m),(1^m))$ given above, it follows that $\lambda$ has Garnir elements
\[
g_{A_i}=\psi_{m+i+1}e(\mathbf{i}_{\lambda}), \quad \forall{i}\in\{1,\dots,n-m-2\}.
\]
We first find the Garnir element of node $A_0=(1,1,1)$. The $(1,1,1)$-Garnir belt is
$\mathbf{B}_{A_0}=
\left\{(j,1,1)\:|\:1\leqslant{j}\leqslant{m+1}\right\}
\cup\left\{(1,2,1)\right\}$, depicted by the shaded area in the following Young diagram of $\lambda$
\[
\gyoung(!\gr;\ayezero;!\wh;;\hdts;!\gr,;,|\sesqui\vdts,;,,:\emset).
\]
The $A_0$-Garnir tableau is
\[
G_{A_0}=\gyoungxy(2.2,1,!\gr2;1!\wh\emplusthree;\emplusfour;\hdts;n,!\gr3,|\sesqui\vdts,;\emplustwo,,:\emset),
\]
where $G_{A_0}=s_1s_2\dots{s_{m+1}}\ttt_{((n-m,1^m),\varnothing)}$. 
Notice that $\ttt_{A_0}=G_{A_0}$. There are $\lfloor\tfrac{m+1}{e}\rfloor$ bricks $B_{A_0}^1$, $B_{A_0}^2,\dots,B_{A_0}^{\lfloor{(m+1)}/{e}\rfloor}$ lying in the first column of $G_{A_0}$, so that
\[
\mathfrak{S}_{A_0}=
\left\langle{w_{A_0}^1,w_{A_0}^2,\dots,w_{A_0}^{\lfloor{(m+1-e)}/{e}\rfloor}}\right\rangle
\cong \mathfrak{S}_{\lfloor{(m+1)}/{e}\rfloor}.
\]
It is clear that $w_{A_0}^rG_{A_0}$ is not a column-standard $\lambda$-tableau for each $r\in\{1,\dots,\lfloor{(m+1-e)}/{e}\rfloor\}$. Hence,
\[
\gar_{A_0}=\left\{G_{A_0}\right\}.
\]
We set $\psi_{G_{A_0}}=\psi_1\psi_2\dots\psi_{m+1}$. It follows that the Garnir element of $A_0$ is 
\[
g_{A_0}=e(\mathbf{i}_{A_0})\psi_{G_{A_0}}
=e(\mathbf{i}_{A_0})\psi_1\psi_2\dots\psi_{m+1}
=\psi_1\psi_2\dots\psi_{m+1}e(\mathbf{i}_{\lambda})\quad(\text{by \cref{psiidemp}}).
\]

\subsection{Homogeneous presentation of Specht modules}\label{Spechtpresentation}

Kleshchev, Mathas and Ram provide the following presentation of Specht modules.

\begin{defn}\cite[Definition 7.11]{KMR}\label{Spechtdef}
Let $\alpha\in{Q_+}$ such that $\het(\alpha)=n$ and let $\lambda\in\mathscr{P}_n^l$. The \textup{(}column\textup{)} \emph{Specht module} $S_{\lambda}$ is the $\mathscr{H}_{\alpha}$-module generated by $z_{\lambda}$ subject only to the \emph{defining relations}:
\begin{itemize}
\item{$e(\mathbf{i_{\lambda}})z_{\lambda}=z_{\lambda}$;}
\item{$y_r z_{\lambda}=0$ for all $r\in\{1,\dots,n\}$;}
\item{$\psi_r z_{\lambda}=0$ for all $i\in\{1,\dots,n-1\}$ such that $r$ and $r+1$ lie in the same column of $\ttt_{\lambda}$;}
\item{$g_Az_{\lambda}=0$ for all Garnir nodes $A$ in $[\lambda]$.}
\end{itemize}
\end{defn}

\subsection{A standard homogeneous basis of Specht modules}

We let every $w\in\mathfrak{S}_n$ have a fixed reduced expression $w=s_{r_1}s_{r_2}\dots s_{r_k}$ throughout, and refer to it as the \emph{preferred reduced expression of $w$}. We define the associated element of $\mathscr{H}_n^{\Lambda}$
\[
\psi_w
:=\psi_{r_1}\psi_{r_2}\dots{\psi_{r_k}},
\]
which in general depends on the choice of a preferred reduced expression of $w$.
For $\lambda\in\mathscr{P}_n^l$ and a $\lambda$-tableau $\ttt$, recall that $w_{\ttt}\in\mathfrak{S}_n$ is defined from $\ttt=w_{\ttt}\ttt_{\lambda}$. We now define the vector
\[
v_{\ttt}:=\psi_{w_{\ttt}}z_{\lambda}\in S_{\lambda}.
\]
In particular, we have $v_{\ttt_{\lambda}}=z_{\lambda}$.

\begin{lem}\label{lem:idemp}
Let $\lambda\in\mathscr{P}_n^l$, $\ttt$ be a $\lambda$-tableau and suppose that $v_{\ttt}=\psi_{w_{\ttt}}z_{\lambda}\in S_{\lambda}$ for some reduced expression $w_{\ttt}\in\mathfrak{S}_n$. Then $e(\mathbf{i})v_{\ttt}=\delta_{\mathbf{i},w_{\ttt}\mathbf{i}_{\lambda}}v_{\ttt}$.
\end{lem}

\begin{proof}
By employing \cref{psiidemp}, we have that $e(\mathbf{i})v_{\ttt}=e(\mathbf{i})\psi_{w_{\ttt}}z_{\lambda}=\psi_{w_{\ttt}}e(w_{\ttt}^{-1}\mathbf{i})z_{\lambda}$. If $\mathbf{i}=w_{\ttt}\mathbf{i}_{\lambda}$, then $\psi_{w_{\ttt}}e(w_{\ttt}^{-1}\mathbf{i})z_{\lambda}=w_{\ttt}e({\mathbf{i}}_{\lambda})z_{\lambda}=w_{\ttt}z_{\lambda}=v_{\ttt}$, by the first defining relation in \cref{Spechtdef}. However, we have that $w_{\ttt}^{-1}\mathbf{i}\neq\mathbf{i}_{\lambda}$ if $\mathbf{i}\neq w_{\ttt}\mathbf{i}_{\lambda}$, and hence $\psi_{w_{\ttt}}e(w_{\ttt}^{-1}\mathbf{i})z_{\lambda}=0$ by \cref{Spechtdef}.
\end{proof}

Whilst the elements $v_{\ttt}$ of $S_{\lambda}$ also depend on the choice of a preferred reduced expression, in general, the following result does not.

\begin{thm}\cite[Corollary 4.6]{BKW}
\label{gradedmod}
For $\lambda\in\mathscr{P}_n^l$, the set of vectors $\{v_{\ttt}\mid{\ttt\in\std(\lambda)}\}$ is a homogeneous $\mathbb{F}$-basis of $S_{\lambda}$. Moreover, $v_{\tts}$ for any $\lambda$-tableau $\tts$ can be written as a linear combination of $\mathbb{F}$-basis elements $v_{\ttt}$ such that $\tts\trianglerighteq{\ttt}$.
\end{thm}
We call this basis the \emph{standard homogeneous basis of $S_{\lambda}$} and remark that Specht modules are naturally $\mathbb{Z}$-graded $\mathscr{H}_n^{\Lambda}$-modules.

\section{Specht modules labelled by hook bipartitions}\label{sec:hookbipartSpecht}

In this section, we begin our study of a particular family of Specht modules in level two of $\mathscr{H}_n^{\Lambda}$, namely those labelled by \emph{hook bipartitions}, with quantum characteristic at least three. We thus fix $e\in\{3,4,\dots\}$ and $l=2$ from now on.
We first give the presentations of Specht modules labelled by hooks and Specht modules labelled by hook bipartitions. We then describe the standard basis elements of the latter Specht modules and show how $\mathscr{H}_n^{\Lambda}$ acts on these elements.
Moreover, we find non-trivial homomorphisms between Specht modules labelled by certain bipartitions, in particular hook bipartitions, which fit into exact sequences.

We define a \emph{hook bipartition of $n$} to be a bipartition of the form $((n-m),(1^m))$ for all $m\in\{0,\dots,n\}$. We will refer to the first component of a hook bipartition as its \emph{arm} and to its second component as its \emph{leg}.

\subsection{Specht module presentations for hooks and hook bipartitions}

It follows from \cref{Spechtdef} that we can explicitly write down the Specht module presentations of $S_{((n-m),(1^m))}$ and $S_{((n-m,1^m),\varnothing)}$ (and hence, also for $S_{(\varnothing,(n-m,1^m))}$) since we determined the Garnir relations for these Specht modules in \cref{Garnir1,Garnir2}, respectively.

\begin{defn}\label{cor:spechtpres}
	\begin{enumerate}
		\item The Specht module $S_{((n-m),(1^{m}))}$ has a presentation given by
		\begin{equation*}
		\stretchleftright[250]
		{\langle}
		{
			z_{((n-m),(1^{m}))}\left |
			\begin{array}{l l}
			e(\mathbf{i}_{((n-m),(1^{m}))})z_{((n-m),(1^{m}))}=z_{((n-m),(1^{m}))}, \\
			y_rz_{((n-m),(1^{m}))}=0\ \forall\ r\in\{1,\dots,n\}, \\
			\psi_rz_{((n-m),(1^{m}))}=0\ \forall\ r\in\{1,\dots,m-1\}\cup\{m+1,\dots,n-1\}
			\end{array}\right.
		}
		{\rangle}.
		\end{equation*}
		
		\item The Specht module $S_{((n-m,1^m),\varnothing)}$ has a presentation given by
		\begin{equation*}
		\stretchleftright[350]
		{\langle}
		{
			z_{((n-m,1^m),\varnothing)}\left |
			\begin{array}{l l}
			e(\mathbf{i}_{((n-m,1^m),\varnothing)})z_{((n-m,1^m),\varnothing)}=
			z_{((n-m,1^m),\varnothing)}, \\
			y_rz_{((n-m,1^m),\varnothing)}=0\ \forall\ r\in\{1,\dots,n\}, \\
			\psi_rz_{((n-m,1^m),\varnothing)}=0 \ \forall\ r\in\{1,\dots,m\}\cup\{m+2,\dots,n-1\}, \\
			\psi_1\dots\psi_{m+1}z_{((n-m,1^m),\varnothing)}=0
			\end{array}\right.
		}
		{\rangle}.
		\end{equation*}
	\end{enumerate}
\end{defn}

\subsection{Homogeneous basis elements of $S_{((n-m),(1^m))}$}

Given a standard $((n-m),(1^m))$-tableau $\ttt$, we write $a_j:=\ttt(j,1,2)$ for all $j\in\{1,\dots,m\}$. Then $\ttt$ is completely determined by $a_1,\dots,a_m$. We write
	\[\ttt=w_{\ttt}\ttt_{((n-m),(1^m))}\in\std((n-m),(1^m)),\]
where
	\[w_{\ttt}=\s{a_1-1}{1}\s{a_2-1}{2}\dots\s{a_m-1}{m}\in\mathfrak{S}_n\]
is a reduced expression.
If $a_i=i$ for all $i\in\{1,\dots,m\}$, then $\ttt=\ttt_{((n-m),(1^m))}$.
For $1\leqslant{i}\leqslant{j}<n$, we define
	\[\Psid{j}{i}:=\psi_j\psi_{j-1}\dots\psi_i;\qquad \Psiu{i}{j}:=\psi_i\psi_{i+1}\dots\psi_{j}.\]
We can now write
	\[v_{\ttt}=\psi_{w_{\ttt}}z_{((n-m),(1^m))}\in S_{((n-m),(1^m))},\]
where
	\[\psi_{w_{\ttt}}=\Psid{a_1-1}{1}\Psid{a_2-1}{2}\dots\Psid{a_m-1}{m}\in\mathscr{H}_n^{\Lambda}.\]
The vectors $v_{\ttt}\in S_{((n-m),(1^m))}$, as $\ttt$ runs over all standard $((n-m),(1^m))$-tableaux, form a basis for $S_{((n-m),(1^m))}$.
For brevity, we write $v(a_1,\dots,a_m):=v_{\ttt}$. Hence, if $a_i=i$ for all $i\in\{1,\dots,m\}$, then $v(1,\dots,m)=v_{\ttt}=z_{((n-m),(1^m))}$.

\subsection{The action of $\mathscr{H}_n^{\Lambda}$ on $S_{((n-m),(1^m))}$}

We now study the explicit action of the cyclotomic Khovanov--Lauda--Rouquier algebra on basis elements of Specht modules labelled by hook bipartitions. Similarly, Speyer \cite[\S{5}]{Spey} studied the action of the Iwahori--Hecke algebra of type $A$ on Specht modules labelled by hook partitions in quantum characteristic two, and used this action to determine the decomposability of these Specht modules. The computations given in this section and \cref{subsec:furtheraction}, which rely on the presentation of $\mathscr{H}_n^{\Lambda}$, are much like those presented by Speyer.

We determine when the basis elements $v_{\ttt}$ of $S_{((n-m),(1^m))}$ are killed by the generators $\psi_1,\dots,\psi_{n-1}$ of $\mathscr{H}_n^{\Lambda}$.

\begin{lem}\label{psigeneratorlem}
Suppose that $1\leqslant{a_1}<\dots<a_m\leqslant{n}$ and $1\leqslant{i}<n$.
	\begin{enumerate}
	\item{ Let $i\not\equiv{2+\kappa_2-\kappa_1}\Mod{e}$. Then $\psi_iv(a_1,\dots,a_m)=0$ if $i,i+1\in\{a_1,\dots,a_m\}$.	}
	\item{ Let $i\not\equiv{\kappa_2-\kappa_1}\Mod{e}$. Then $\psi_iv(a_1,\dots,a_m)=0$ if $i,i+1\not\in\{a_1,\dots,a_m\}$.	}
	\end{enumerate}
\end{lem}

\begin{proof}
We proceed by induction on the sum $a_1+\dots+a_m$.
\begin{enumerate}
\item{ Let $r$ be such that $a_r=i$ and $a_{r+1}=i+1$. If $i=r$, then all of the terms $\Psid{a_1-1}{1},\dots,\Psid{a_{r+1}-1}{r+1}$ are trivial. Thus, by employing \cref{psicomm} and the first part of \cref{cor:spechtpres}, we have
\begin{align*}
\psi_iv(a_1,\dots,a_m)
&=\psi_i\Psid{a_{r+2}-1}{r+2}\dots\Psid{a_m-1}{m}z_{((n-m),(1^m))}
\\&
=\Psid{a_{r+2}-1}{r+2}\dots\Psid{a_m-1}{m}\psi_iz_{((n-m),(1^m))}
=0.
\end{align*}
Now assume that $i\geqslant{r+1}$, and using \cref{psicomm}, we observe that
\begin{align*}
&\psi_iv(a_1,\dots,a_m)\\
&=\Psid{a_1-1}{1}\dots\Psid{a_{r-1}-1}{r-1}
\psi_i\psi_{i-1}\psi_i
\Psid{i-2}{r}\Psid{i-1}{r+1}\Psid{a_{r+2}-1}{r+2}\dots\Psid{a_m-1}{m}z_{((n-m),(1^m))}\\
&=\Psid{a_1-1}{1}\dots\Psid{a_{r-1}-1}{r-1}
\psi_i\psi_{i-1}\psi_i\psi_{w_{\ttt}}z_{((n-m),(1^m))},
\end{align*}

where $w_{\ttt}=\s{i-2}{r}\s{i-1}{r+1}\s{a_{r+2}-1}{r+2}\dots\s{a_m-1}{m}$ such that $\ttt=w_{\ttt}\ttt_{((n-m),(1^m))}$.
We observe that $\ttt(r,1,2)=i-1$, $\ttt(r+1,1,2)=i$ and $\ttt(1,i-r,1)=i+1$, and moreover, $\res_{\ttt}(r,1,2)\equiv{\kappa_2+1-r}\Mod{e}$, $\res_{\ttt}(r+1,1,2)\equiv\kappa_2-r\Mod{e}$ and $\res_{\ttt}(1,i-r,1)\not\equiv\kappa_2-r+1\Mod{e}$. Recall that $e_{\ttt}$ is the idempotent with respect to $\ttt$. By applying \cref{psiidemp} and \cref{psibraid}, we thus have

\begin{align*}
&\Psid{a_1-1}{1}\dots\Psid{a_{r-1}-1}{r-1}
\psi_i\psi_{i-1}\psi_i
\Psid{i-2}{r}\Psid{i-1}{r+1}\Psid{a_{r+2}-1}{r+2}\dots\Psid{a_m-1}{m}z_{((n-m),(1^m))}\\
&=\Psid{a_1-1}{1}\dots\Psid{a_{r-1}-1}{r-1}
\psi_i\psi_{i-1}\psi_i
\Psid{i-2}{r}\Psid{i-1}{r+1}\Psid{a_{r+2}-1}{r+2}\dots\Psid{a_m-1}{m}e(\mathbf{i}_{((n-m),(1^m))})z_{((n-m),(1^m))}\\
&=\Psid{a_1-1}{1}\dots\Psid{a_{r-1}-1}{r-1}
\psi_i\psi_{i-1}\psi_ie_{\ttt}
\Psid{i-2}{r}\Psid{i-1}{r+1}\Psid{a_{r+2}-1}{r+2}\dots\Psid{a_m-1}{m}z_{((n-m),(1^m))}\\
&=\Psid{a_1-1}{1}\dots\Psid{a_{r-1}-1}{r-1}
\psi_{i-1}\psi_i\psi_{i-1}
\Psid{i-2}{r}\Psid{i-1}{r+1}\Psid{a_{r+2}-1}{r+2}\dots\Psid{a_m-1}{m}z_{((n-m),(1^m))}\\
&=\Psid{a_1-1}{1}\dots\Psid{a_{r-1}-1}{r-1}
\psi_{i-1}\psi_i\psi_{i-1}
v(1,\dots,r-1,i-1,i,a_{r+2},\dots,a_m),
\end{align*}
where $\psi_{i-1}v(1,\dots,r-1,i-1,i,a_{r+2},\dots,a_m)$ equals zero by induction if $i-1\not\equiv{2+\kappa_2-\kappa_1}\Mod{e}$. Now suppose that $i\equiv{\kappa_2-\kappa_1+3}\Mod{e}$. We note that the terms $\Psid{i-2}{r}$ and $\Psid{i-1}{r+1}$ are trivial if $i=r+1$. Thus, by applying \cref{psicomm} and \cref{cor:spechtpres}, we have
\begin{align*}
\psi_{i-1}v(1,\dots,r-1,i-1,i,a_{r+2},\dots,a_m)
&=\psi_{i-1}\Psid{a_{r+2}-1}{r+2}\dots\Psid{a_m-1}{m}z_{((n-m),(1^m))}\\
&=\Psid{a_{r+2}-1}{r+2}\dots\Psid{a_m-1}{m}\psi_{i-1}z_{((n-m),(1^m))}
=0.
\end{align*}
We now suppose that $i\geqslant{r+2}$. Then, by using \cref{psibraid}, $\psi_iv(a_1,\dots,a_m)$ becomes
\begin{align*}
&\Psid{a_1-1}{1}\dots\Psid{a_{r-1}-1}{r-1}
\psi_{i-1}\psi_i\psi_{i-1}\psi_{i-2}\psi_{i-1}
\Psid{i-3}{r}\Psid{i-2}{r+1}\Psid{a_{r+2}-1}{r+2}\dots\Psid{a_m-1}{m}z_{((n-m),(1^m))}\\
&=\Psid{a_1-1}{1}\dots\Psid{a_{r-1}-1}{r-1}
\psi_{i-1}\psi_i
(\psi_{i-2}\psi_{i-1}\psi_{i-2}+1)
\Psid{i-3}{r}\Psid{i-2}{r+1}\Psid{a_{r+2}-1}{r+2}\dots \\
&\qquad\dots\Psid{a_m-1}{m}z_{((n-m),(1^m))}.
\end{align*}
By splitting this sum into its two terms, the first term becomes
\begin{align*}
&\Psid{a_1-1}{1}\dots\Psid{a_{r-1}-1}{r-1}
\psi_{i-1}\psi_i
\psi_{i-2}\psi_{i-1}\psi_{i-2}
v(1,\dots,r-1,i-2,i-1,a_{r+2},\dots,a_m),
\end{align*}
where $\psi_{i-2}v(1,\dots,r-1,i-2,i-1,a_{r+2},\dots,a_m)$ is zero by induction since $i-2\not\equiv{2+\kappa_2-\kappa_1}\Mod{e}$.
The second term is
\[
\Psid{a_1-1}{1}\dots\Psid{a_{r-1}-1}{r-1}
\psi_{i-1}\psi_i\Psid{i-3}{r}\Psid{i-2}{r+1}
\Psid{a_{r+2}-1}{r+2}\dots\Psid{a_m-1}{m}z_{((n-m),(1^m))},
\]
which becomes
\[
\Psid{a_1-1}{1}\dots\Psid{a_{r-1}-1}{r-1}
\Psid{i-3}{r}\Psid{i-1}{r+1}\psi_iz_{((n-m),(1^m))}=0
\]
by \cref{cor:spechtpres} if $r+1=m$. We now assume that $r+1<m$, so that the second term becomes
\begin{align*}
&\Psid{a_1-1}{1}\dots\Psid{a_{r-1}-1}{r-1}
\Psid{i-3}{r}\Psid{i-1}{r+1}\Psid{a_{r+2}-1}{i+2}
\psi_i\psi_{i+1}\psi_i
\Psid{i-1}{r+2}\Psid{a_{r+3}-1}{r+3}\dots
\dots\Psid{a_m-1}{m}z_{((n-m),(1^m))},
\end{align*}
by using \cref{psicomm}.
If $e\neq{3}$, then by using \cref{psibraid} this term becomes
\begin{align*}
&\Psid{a_1-1}{1}\dots\Psid{a_{r-1}-1}{r-1}
\Psid{i-3}{r}\Psid{i-1}{r+1}\Psid{a_{r+2}-1}{i+2}
\psi_{i+1}\psi_i\psi_{i+1}
\Psid{i-1}{r+2}\Psid{a_{r+3}-1}{r+3}\dots
\Psid{a_m-1}{m}z_{((n-m),(1^m))},
\end{align*}
whereas if $e=3$, then by using \cref{psibraid} the term becomes
\begin{align*}
&\Psid{a_1-1}{1}\dots\Psid{a_{r-1}-1}{r-1}
\Psid{i-3}{r}\Psid{i-1}{r+1}\Psid{a_{r+2}-1}{i+2}
(\psi_{i+1}\psi_i\psi_{i+1}-1)
\Psid{i-1}{r+2}\Psid{a_{r+3}-1}{r+3}\dots
\Psid{a_m-1}{m}z_{((n-m),(1^m))}.
\end{align*}

We first assume that $e$ is arbitrary and see that the first term in its sum (and only term if $e\neq{3}$) is
\begin{align*}
&\Psid{a_1-1}{1}\dots\Psid{a_{r-1}-1}{r-1}
\Psid{i-3}{r}\Psid{i-1}{r+1}\Psid{a_{r+2}-1}{i+2}
\psi_{i+1}\psi_i\psi_{i+1}
\Psid{i-1}{r+2}\Psid{a_{r+3}-1}{r+3}\dots\Psid{a_m-1}{m}z_{((n-m),(1^m))},
\end{align*}
and moreover, if $m=r+2$, this equals
\begin{align*}
&\Psid{a_1-1}{1}\dots\Psid{a_{m-3}-1}{m-3}\Psid{i-3}{m-2}\Psid{i-2}{m-1}\Psid{a_{m}-1}{i+2}
\psi_{i+1}\psi_i\psi_{i+1}\Psid{i-1}{m}z_{((n-m),(1^m))}\\
&=\Psid{a_1-1}{1}\dots\Psid{a_{m-3}-1}{m-3}\Psid{i-3}{m-2}\Psid{i-2}{m-1}\Psid{a_{m}-1}{i+2}
\psi_{i+1}\psi_i\Psid{i-1}{m}\psi_{i+1}z_{((n-m),(1^m))}=0,
\end{align*}
by \cref{psicomm} and \cref{cor:spechtpres}.
Now supposing that $m>r+2$, this term is
\begin{align*}
&\Psid{a_1-1}{1}\dots\Psid{a_{r-1}-1}{r-1}
\Psid{i-3}{r}\Psid{i-1}{r+1}\Psid{a_{r+2}-1}{i+2}
\psi_{i+1}\psi_i\psi_{i+1}
v(1,\dots,r+1,i,a_{r+3},\dots,a_m),
\end{align*}
where $\psi_{i+1}v(1,\dots,r+1,i,a_{r+3},\dots,a_m)$ equals zero by the inductive hypothesis of part $(2)$ of the lemma if $e\neq{4}$, in which case $i+1\not\equiv{\kappa_2-\kappa_1}\Mod{e}$. We now assume that $e=4$. Then, by using \cref{psicomm} and \cref{psibraid}, this term becomes
\begin{align*}
&\Psid{a_1-1}{1}\dots\Psid{a_{r-1}-1}{r-1}
\Psid{i-3}{r}\Psid{i-1}{r+1}\Psid{a_{r+2}-1}{r+2}
\Psid{a_{r+3}-1}{i+3}
\psi_{i+1}\psi_{i+2}\psi_{i+1}
\Psid{i}{r+3}
\Psid{a_{r+4}-1}{r+4}\dots\\
&\qquad\dots\Psid{a_m-1}{m}z_{((n-m),(1^m))}\\
&=\Psid{a_1-1}{1}\dots\Psid{a_{r-1}-1}{r-1}
\Psid{i-3}{r}\Psid{i-1}{r+1}\Psid{a_{r+2}-1}{r+2}
\Psid{a_{r+3}-1}{i+3}
(\psi_{i+2}\psi_{i+1}\psi_{i+2}-1)
\Psid{i}{r+3}
\Psid{a_{r+4}-1}{r+4}\dots\\
&\qquad\dots\Psid{a_m-1}{m}z_{((n-m),(1^m))}.
\end{align*}
If $m=r+3$, then by applying \cref{psicomm} and \cref{cor:spechtpres}, the first term of this expression equals
\begin{align*}
&\Psid{a_1-1}{1}\dots\Psid{a_{m-4}-1}{m-4}\Psid{i-3}{m-3}\Psid{i-1}{m-2}
\Psid{a_{m-1}-1}{m-1}\Psid{a_m-1}{i+3}\psi_{i+2}\psi_{i+1}\psi_{i+2}\Psid{i}{m}z_{((n-m),(1^m))}\\
&=\Psid{a_1-1}{1}\dots\Psid{a_{m-4}-1}{m-4}\Psid{i-3}{m-3}\Psid{i-1}{m-2}
\Psid{a_{m-1}-1}{m-1}\Psid{a_m-1}{i+3}\psi_{i+2}\psi_{i+1}\Psid{i}{m}\psi_{i+2}z_{((n-m),(1^m))}=0,
\end{align*}
and the second term becomes
\begin{align*}
&-\Psid{a_1-1}{1}\dots\Psid{a_{m-4}-1}{m-4}\Psid{i-3}{m-3}\Psid{i-1}{m-2}\Psid{a_{m-1}-1}{m-1}\Psid{a_{m}-1}{i+3}\Psid{i}{m}z_{((n-m),(1^m))}\\
&=-\Psid{a_1-1}{1}\dots\Psid{a_{m-4}-1}{m-4}\Psid{i-3}{m-3}\Psid{i-1}{m-2}\Psid{a_{m-1}-1}{m-1}\Psid{a_{m}-1}{i+4}\Psid{i}{m}\psi_{i+3}z_{((n-m),(1^m))}=0.
\end{align*}
Now supposing that $m>r+3$, then the first term is
\begin{align*}
&\Psid{a_1-1}{1}\dots\Psid{a_{r-1}-1}{r-1}
\Psid{i-3}{r}\Psid{i-1}{r+1}\Psid{a_{r+2}-1}{r+2}
\Psid{a_{r+3}-1}{i+3}
\psi_{i+2}\psi_{i+1}\\
&\qquad\cdot\psi_{i+2}
v(1,\dots,r+2,i+1,a_{r+4},\dots,a_m),
\end{align*}
where $\psi_{i+2}v(1,\dots,r+2,i+1,a_{r+4},\dots,a_m)$ equals zero by the inductive hypothesis of part $(2)$ of the lemma, since $i+2\not\equiv{\kappa_2-\kappa_1}\Mod{4}$.
We now look at the second term and first suppose that $a_{r+3}=i+3$. Then we have
\begin{align*}
&-\Psid{a_1-1}{1}\dots\Psid{a_{r-1}-1}{r-1}\Psid{i-3}{r}\Psid{i-1}{r+1}
\psi_{i+1}
(\psi_i\psi_{i-1}\psi_i)
\Psid{i-2}{r+2}\Psid{i-1}{r+3}
\Psid{a_{r+4}-1}{r+4}\dots
\Psid{a_m-1}{m}z_{((n-m),(1^m))}\\
&=-\Psid{a_1-1}{1}\dots\Psid{a_{r-1}-1}{r-1}\Psid{i-3}{r}\Psid{i-1}{r+1}\psi_{i+1}
(\psi_{i-1}\psi_i\psi_{i-1})
\Psid{i-2}{r+2}\Psid{i-1}{r+3}
\Psid{a_{r+4}-1}{r+4}\dots
\Psid{a_m-1}{m}z_{((n-m),(1^m))}\\
&=-\Psid{a_1-1}{1}\dots\Psid{a_{r-1}-1}{r-1}\Psid{i-3}{r}\Psid{i-1}{r+1}
\psi_{i+1}\psi_{i-1}\psi_i
(\psi_{i-1}\psi_{i-2}\psi_{i-1})
\Psid{i-3}{r+2}\Psid{i-2}{r+3}
\Psid{a_{r+4}-1}{r+4}\dots\\
&\qquad\dots\Psid{a_m-1}{m}z_{((n-m),(1^m))}\\
&=-\Psid{a_1-1}{1}\dots\Psid{a_{r-1}-1}{r-1}\Psid{i-3}{r}\Psid{i-1}{r+1}
\psi_{i+1}\psi_{i-1}\psi_i
(\psi_{i-2}\psi_{i-1}\psi_{i-2}+1)
\Psid{i-3}{r+2}\Psid{i-2}{r+3}
\Psid{a_{r+4}-1}{r+4}\dots\\
&\qquad\dots\Psid{a_m-1}{m}z_{((n-m),(1^m))},
\end{align*}
by \cref{psicomm} and \cref{psibraid}. By splitting this sum into its two terms, the first term becomes
\begin{align*}
&-\Psid{a_1-1}{1}\dots\Psid{a_{r-1}-1}{r-1}\Psid{i-3}{r}\Psid{i-1}{r+1}
\psi_{i+1}\psi_{i-1}\psi_i
\psi_{i-2}\psi_{i-1}\\
&\qquad\cdot\psi_{i-2}
v(1,\dots,r+1,i-2,i-1,a_{r+4},\dots,a_m),
\end{align*}
where $\psi_{i-2}v(1,\dots,r+1,i-2,i-1,a_{r+4},\dots,a_m)$ equals zero by the inductive hypothesis of part $(1)$ of the lemma. The second term becomes
\begin{align*}
&-\Psid{a_1-1}{1}\dots\Psid{a_{r-1}-1}{r-1}\Psid{i-3}{r}\Psid{i-1}{r+1}
\psi_{i+1}\psi_{i-1}
\Psid{i-3}{r+2}\psi_i
v(1,\dots,r+2,i-1,a_{r+4},\dots,a_m),
\end{align*}
where $\psi_iv(1,\dots,r+2,i-1,a_{r+4},\dots,a_m)$ equals zero by the inductive hypothesis of part $(2)$ of the lemma.
Instead suppose that $a_{r+3}>{i+3}$. Then, by \cref{psicomm} and \cref{psibraid}, the second term becomes
\begin{align*}
&-\Psid{a_1-1}{1}\dots\Psid{a_{r-1}-1}{r-1}\Psid{i-3}{r}\Psid{i-1}{r+1}\Psid{a_{r+2}-1}{r+2}\Psid{a_{r+3}-1}{i+4}\Psid{i}{r+3}\Psid{a_{r+4}-1}{i+5}
\psi_{i+3}\psi_{i+4}\psi_{i+3}\\
&\qquad\cdot
\Psid{i+2}{r+4}\Psid{a_{r+5}-1}{r+5}\dots\Psid{a_m-1}{m}z_{((n-m),(1^m))}\\
&=-\Psid{a_1-1}{1}\dots\Psid{a_{r-1}-1}{r-1}\Psid{i-3}{r}\Psid{i-1}{r+1}\Psid{a_{r+2}-1}{r+2}\Psid{a_{r+3}-1}{i+4}\Psid{i}{r+3}\Psid{a_{r+4}-1}{i+5}
\psi_{i+4}\psi_{i+3}\psi_{i+4}\\
&\qquad\cdot
\Psid{i+2}{r+4}\Psid{a_{r+5}-1}{r+5}\dots\Psid{a_m-1}{m}z_{((n-m),(1^m))}\\
&=-\Psid{a_1-1}{1}\dots\Psid{a_{r-1}-1}{r-1}\Psid{i-3}{r}\Psid{i-1}{r+1}\Psid{a_{r+2}-1}{r+2}\Psid{a_{r+3}-1}{i+4}\Psid{i}{r+3}\Psid{a_{r+4}-1}{i+5}
\psi_{i+4}\psi_{i+3}\\&\qquad\cdot\psi_{i+4}
v(1,\dots,r+3,i+3,a_{r+5},\dots,a_m),
\end{align*}
where $\psi_{i+4}v(1,\dots,r+3,i+3,a_{r+5},\dots,a_m)$ equals zero by the inductive hypothesis of part $(2)$ of the lemma.

We now assume that $e=3$. Then, by using \cref{psicomm}, the second term from above becomes
\begin{align*}
&-\Psid{a_1-1}{1}\dots\Psid{a_{r-1}-1}{r-1}
\Psid{i-3}{r}\Psid{a_{r+2}-1}{i+2}\psi_{i-1}\Psid{i-2}{r+1}
\Psid{i-1}{r+2}\Psid{a_{r+3}-1}{r+3}\dots\Psid{a_m-1}{m}z_{((n-m),(1^m))}.
\end{align*}
If $i=r+2$, then the terms $\Psid{i-3}{r}$, $\Psid{i-2}{r+1}$ and $\Psid{i-1}{r+2}$ are trivial, so we thus have
\begin{align*}
&-\Psid{a_1-1}{1}\dots\Psid{a_{r-1}-1}{r-1}
\Psid{a_{r+2}-1}{i+2}\psi_{i-1}\Psid{a_{r+3}-1}{r+3}\dots\Psid{a_m-1}{m}z_{((n-m),(1^m))}\\
&=-\Psid{a_1-1}{1}\dots\Psid{a_{r-1}-1}{r-1}
\Psid{a_{r+2}-1}{i+2}\Psid{a_{r+3}-1}{r+3}\dots\Psid{a_m-1}{m}\psi_{i-1}z_{((n-m),(1^m))}
=0,
\end{align*}
by \cref{psicomm} and \cref{cor:spechtpres}.
We now assume that $i\geqslant{r+3}$ and rewrite this expression to be
\begin{align*}
&-\Psid{a_1-1}{1}\dots\Psid{a_{r-1}-1}{r-1}\Psid{a_{r+2}-1}{i+2}
\Psid{i-3}{r}\psi_{i-1}\psi_{i-2}\psi_{i-1}
\Psid{i-3}{r+1}
\Psid{i-2}{r+2}\Psid{a_{r+3}-1}{r+3}\dots
\Psid{a_m-1}{m}z_{((n-m),(1^m))}\\
&=-\Psid{a_1-1}{1}\dots\Psid{a_{r-1}-1}{r-1}\Psid{a_{r+2}-1}{i+2}
\Psid{i-3}{r}
(\psi_{i-2}\psi_{i-1}\psi_{i-2}-1)
\Psid{i-3}{r+1}
\Psid{i-2}{r+2}\Psid{a_{r+3}-1}{r+3}\dots\\
&\qquad\dots\Psid{a_m-1}{m}z_{((n-m),(1^m))},
\end{align*}
by using \cref{psicomm} and \cref{psibraid}; we again consider its two summands. The first term is
\begin{align*}
&-\Psid{a_1-1}{1}\dots\Psid{a_{r-1}-1}{r-1}\Psid{a_{r+2}-1}{i+2}
\Psid{i-3}{r}
\psi_{i-2}\psi_{i-1}\psi_{i-2}
v(1,\dots,r,i-2,i-1,a_{r+3},\dots,a_m),
\end{align*}
where $\psi_{i-2}v(1,\dots,r,i-2,i-1,a_{r+3},\dots,a_m)$ equals zero by induction since $i-2\not\equiv{2+\kappa_2-\kappa_1}\Mod{3}$.
The second term becomes
\begin{align*}
&-\Psid{a_1-1}{1}\dots\Psid{a_{r-1}-1}{r-1}\Psid{a_{r+2}-1}{i+2}\psi_{i-3}
\Psid{i-4}{r}
\Psid{i-3}{r+1}
\Psid{i-2}{r+2}\Psid{a_{r+3}-1}{r+3}\dots\Psid{a_m-1}{m}z_{((n-m),(1^m))}.
\end{align*}
If $i=r+3$, then $\Psid{i-4}{r}\Psid{i-3}{r+1}\Psid{i-2}{r+2}$ is trivial, so the expression becomes
\begin{align*}
&-\Psid{a_1-1}{1}\dots\Psid{a_{r-1}-1}{r-1}\Psid{a_{r+2}-1}{i+2}
\Psid{a_{r+3}-1}{r+3}\dots\Psid{a_m-1}{m}\psi_{i-3}z_{((n-m),(1^m))}=0,
\end{align*}
by \cref{psicomm} and \cref{cor:spechtpres}.
Now supposing that $i\geqslant{r+4}$, the expression becomes
\begin{align*}
&-\Psid{a_1-1}{1}\dots\Psid{a_{r-1}-1}{r-1}\Psid{a_{r+2}-1}{i+2}\psi_{i-3}
v(1,\dots,r-1,i-3,i-2,i-1,a_{r+3},\dots,a_m),
\end{align*}
where $\psi_{i-3}v(1,\dots,r-1,i-3,i-2,i-1,a_{r+3},\dots,a_m)$ equals zero by induction since $i-3\not\equiv{2+\kappa_2-\kappa_1}\Mod{e}$.
Thus, we have proved the first statement, as required, by assuming the inductive hypothesis for both parts of the lemma.}
\item{ We prove the second statement similarly to the first, by using the inductive hypotheses of both statements, together with the relations in the Khovanov--Lauda--Rouquier algebra and the Specht module presentations given in \cref{defKLR} and in the first part of \cref{cor:spechtpres}.}
\qedhere
\end{enumerate}
\end{proof}

\begin{rmk}
We note that these are not the only cases when $\psi_i$ kills the basis vector $v(a_1,\dots,a_m)$, for $1\leqslant{i}<n$. For example, let $e=3$, $\kappa=(0,0)$, $i=3$, and $S_{((3),(1^3))}$. Then $\psi_3v(1,2,4)=0$, where $3\not\in\{a_1,a_2,a_3\}$. In \cref{subsec:furtheraction}, we will expand on the previous lemma and give an explicit description of the complete action of the generators $\psi_1,\dots,\psi_{n-1}\in\mathscr{H}_n^{\Lambda}$ on the basis elements $v_{\ttt}\in S_{((n-m),(1^m))}$.
\end{rmk}

\begin{lem}\label{lem:psiup}
	If $i\not\equiv{1+\kappa_2-\kappa_1}\Mod{e}$ with $i\leqslant{m}$, then $y_i\Psiu{i}{m}z_{((n-m),(1^m))}=0$.
\end{lem}

\begin{proof}
	We proceed by downwards induction on $i$. If $i=m$, then $y_m\psi_mz_{((n-m),(1^m))}=\psi_my_{m+1}z_{((n-m),(1^m))}=0$ by \cref{psiyup} and the first part of \cref{cor:spechtpres}.
	We now suppose that $i<m$. Then
	\[y_i\Psiu{i}{m}z_{((n-m),(1^m))}
	=\psi_iy_{i+1}\Psiu{i+1}{m}z_{((n-m),(1^m))}\]
	by \cref{psiidemp} and \cref{psiyup}, where $y_{i+1}\Psiu{i+1}{m}z_{((n-m),(1^m))}$ equals zero by induction if $i\not\equiv{\kappa_2-\kappa_1}\Mod{e}$.
	Assuming that $i\equiv{\kappa_2-\kappa_1}\Mod{e}$, this term becomes
	\begin{align*}
	&\psi_i(\psi_{i+1}y_{i+2}-1)\Psiu{i+2}{m}z_{((n-m),(1^m))}
	=\psi_i\psi_{i+1}y_{i+2}\Psiu{i+2}{m}z_{((n-m),(1^m))}
	-\psi_i\Psiu{i+2}{m}z_{((n-m),(1^m))},
	\end{align*}
	by \cref{psiyup}. The first term equals zero since $y_{i+2}\Psiu{i+2}{m}z_{((n-m),(1^m))}$ equals zero by induction, whilst the second term becomes $-\Psiu{i+2}{m}\psi_iz_{((n-m),(1^m))}=0$ by \cref{psicomm} and \cref{cor:spechtpres}.
\end{proof}

We now show when the generators $y_1,\dots,y_n\in\mathscr{H}_n^{\Lambda}$ act trivially on basis elements $v_{\ttt}\in S_{((n-m),(1^m))}$.

\begin{lem}\label{ygeneratorlem}
	\begin{enumerate}
	\item{ Let $i\equiv{1+\kappa_2-\kappa_1}\Mod{e}$. Then $y_iv(a_1,\dots,a_m)=0$ if and only if either $i\in\{a_1,\dots,a_m\}$ or $i+1\not\in\{a_1,\dots,a_m\}$.}
	\item{ Let $i\equiv{2+\kappa_2-\kappa_1}\Mod{e}$. Then $y_iv(a_1,\dots,a_m)=0$ if and only if either $i-1\in\{a_1,\dots,a_m\}$ or $i\not\in\{a_1,\dots,a_m\}$.}
	\item{ Let $i-\kappa_2+\kappa_1\not\equiv{1,2}\Mod{e}$. Then $y_iv(a_1,\dots,a_m)=0$.}
	\end{enumerate}
\end{lem}

\begin{proof}
We first proceed by simultaneous induction on the sum $a_1+\dots+a_m$ to show that $y_iv(a_1,\dots,a_m)$ equals zero in the following six cases:
\begin{itemize}
\item{ $i\equiv{1+\kappa_2-\kappa_1}\Mod{e}$ and $i\in\{a_1,\dots,a_m\}$; }
\item{ $i\equiv{1+\kappa_2-\kappa_1}\Mod{e}$ and $i+1\not\in\{a_1,\dots,a_m\}$; }
\item{ $i\equiv{2+\kappa_2-\kappa_1}\Mod{e}$ and $i-1\in\{a_1,\dots,a_m\}$; }
\item{ $i\equiv{2+\kappa_2-\kappa_1}\Mod{e}$ and $i\not\in\{a_1,\dots,a_m\}$; }
\item{ $i-\kappa_2+\kappa_1\not\equiv{1,2}\Mod{e}$ and $i\in\{a_1,\dots,a_m\}$; }
\item{ $i-\kappa_2+\kappa_1\not\equiv{1,2}\Mod{e}$ and $i\not\in\{a_1,\dots,a_m\}$. }
\end{itemize}
We label these cases $A$, $A'$, $B$, $B'$, $C$ and $C'$, respectively, from top to bottom. We only provide full details of the statements $A$ and $A'$ since the other statements are similarly proved using \cref{rel} and the first part of \cref{cor:spechtpres}.
\begin{enumerate}
	\item{
	\begin{enumerate}
		\item{
		Suppose that $i\in\{a_1,\dots,a_m\}$ and let $a_r=i$.
		If $i=r$ then $\Psid{a_1-1}{1}\dots\Psid{a_r-1}{r}$ is trivial, so that
			\begin{align*}
			y_iv(a_1,\dots,a_m)
			&=
			y_i
			\Psid{a_{r+1}-1}{r+1}
			\dots
			\Psid{a_m-1}{m}
			z_{((n-m),(1^m))}\\
			&=
			\Psid{a_{r+1}-1}{r+1}
			\dots
			\Psid{a_m-1}{m}
			y_i
			z_{((n-m),(1^m))}
			=0,
			\end{align*}
		by \cref{ypsicomm} and \cref{cor:spechtpres}.	
		We now suppose that $i\geqslant{r+1}$. Then, by using \cref{ypsicomm} and \cref{ypsidown}, we have
			\begin{align*}
			y_iv(a_1,\dots,a_m)
			&=
			\Psid{a_1-1}{1}
			\dots
			\Psid{a_{r-1}-1}{r-1}
			y_i\psi_{i-1}
			\Psid{i-2}{r}
			\Psid{a_{r+1}-1}{r+1}
			\dots
			\Psid{a_m-1}{m}
			z_{((n-m),(1^m))}
			\\&=
			\Psid{a_1-1}{1}
			\dots
			\Psid{a_{r-1}-1}{r-1}
			\psi_{i-1}y_{i-1}
			\Psid{i-2}{r}
			\Psid{a_{r+1}-1}{r+1}
			\dots
			\Psid{a_m-1}{m}
			z_{((n-m),(1^m))}
			\\&=
			\Psid{a_1-1}{1}
			\dots
			\Psid{a_{r-1}-1}{r-1}
			\psi_{i-1}y_{i-1}
			v(1,\dots,r-1,i-1,a_{r+1},\dots,a_m),
			\end{align*}	
		where $y_{i-1}v(1,\dots,r-1,i-1,a_{r+1},\dots,a_m)$ equals zero by the inductive hypothesis of $C$.				
		}
		\item{
		Suppose that $i+1\not\in\{a_1,\dots,a_m\}$, so let $a_r\leqslant{i}$ and $a_{r+1}\geqslant{i+2}$.
		If $i=r$ then $y_iv(a_1,\dots,a_m)$ is trivial by part $(a)$.		
		So let $i\geqslant{r+1}$.
		\begin{enumerate}
			\item{
			Suppose that $a_r=i$. Then $y_iv(a_1,\dots,a_m)=0$ by part $(a)$.			
			}
			\item{
			Suppose that $a_r\leqslant{i-1}$.
			Then, by applying \cref{ypsicomm} and \cref{psiyup}, we have
			\begin{align*}
			&
			y_iv(a_1,\dots,a_m)
			\\&=
			\Psid{a_1-1}{1}
			\dots
			\Psid{a_r-1}{r}
			\Psid{a_{r+1}-1}{i+1}
			y_i\psi_i
			\Psid{i-1}{r+1}
			\Psid{a_{r+2}-1}{r+2}
			\dots
			\Psid{a_m-1}{m}z_{((n-m),(1^m))}
			\\&=
			\Psid{a_1-1}{1}
			\dots
			\Psid{a_r-1}{r}
			\Psid{a_{r+1}-1}{i+1}
			(\psi_iy_{i+1}-1)
			\Psid{i-1}{r+1}
			\Psid{a_{r+2}-1}{r+2}
			\dots
			\Psid{a_m-1}{m}z_{((n-m),(1^m))}\\
			&=
			\Psid{a_1-1}{1}
			\dots
			\Psid{a_r-1}{r}
			\Psid{a_{r+1}-1}{i+1}
			\psi_iy_{i+1}
			v(1,\dots,r,i,a_{r+2},\dots,a_m)\\
			&\ -
			\Psid{a_1-1}{1}
			\dots
			\Psid{a_r-1}{r}
			\Psid{a_{r+1}-1}{i+2}
			\psi_{i+1}
			v(1,\dots,r,i,a_{r+2},\dots,a_m),
			\end{align*}	
			
		where $y_{i+1}v(1,\dots,r,i,a_{r+2},\dots,a_m)$ equals zero by the inductive hypothesis of $B$ and $\psi_{i+1}v(1,\dots,r,i,a_{r+2},\dots,a_m)$ equals zero by part two of \cref{psigeneratorlem} since $a_{r+2}\geqslant{i+3}$ and $i+1\not\equiv{\kappa_2-\kappa_1}\Mod{e}$.	}
		\end{enumerate}}\end{enumerate}}
	\item{\begin{enumerate}
	\item{Suppose that $i-1\in\{a_1,\dots,a_m\}$ and let $a_r=i-1$. If $a_{r+1}=i$, we show that $y_iv(a_1,\dots,a_m)=0$ by using that inductive hypothesis of $A$ and the first part of \cref{psigeneratorlem}. We now suppose that $a_{r+1}>i$ and provide details of the base case as follows. If $i=r+1$, then the term $\Psid{a_1-1}{1}\dots\Psid{a_r-1}{r}$ is trivial. Thus, by \cref{ypsicomm}, we have
	\begin{align*}
	y_iv(a_1,\dots,a_m)
	&=y_i\Psid{a_{i}-1}{i}\Psid{a_{i+1}-1}{i+1}\dots\Psid{a_m-1}{m}z_{((n-m),(1^m))}\\
	&=\Psid{a_i-1}{i+1}\Psid{a_{i+1}-1}{i+2}\dots\Psid{a_m-1}{m+1}y_i\Psiu{i}{m}z_{((n-m),(1^m))},
	\end{align*}
	where $y_i\Psiu{i}{m}z_{((n-m),(1^m))}=0$ by \cref{lem:psiup}. We then let $i>r+1$ and show that $y_iv(a_1,\dots,a_m)=0$ by using the inductive hypothesis of $C'$.
	}
	\item{Suppose that $i\not\in\{a_1,\dots,a_m\}$. We show that statement $B'$ holds by the inductive hypothesis of $C'$.	}
	\end{enumerate}	}
	\item{\begin{enumerate}
		\item{Suppose that $i\in\{a_1,\dots,a_m\}$ and show that statement $C$ holds by the inductive hypotheses of $A$ and $C$, together with the second part of \cref{psigeneratorlem}.}
		\item{Suppose that $a_{r-1}\leqslant{i-1}$ and $a_r\geqslant{i+1}$.
We let $i=r$ for the base case and remark that $y_rv(a_1,\dots,a_m)$ equals zero, similarly to part 2(a), by applying \cref{lem:psiup}.
We now suppose that $i>{r}$, and prove this case using the inductive hypotheses of $B$, $C$ and $C'$, together with the second part of \cref{psigeneratorlem}.}
		\end{enumerate}}\end{enumerate}
We now suppose that $i\equiv{\kappa_2-\kappa_1+1}\Mod{e}$, $i\not\in\{a_1,\dots,a_m\}$ and $i+1\in\{a_1,\dots,a_m\}$.
Let $a_r\leqslant{i-1}$ and $a_{r+1}=i+1$. By using \cref{ypsicomm} and \cref{psiyup}, we have
	\begin{align*}
	y_iv(a_1,\dots,a_m) 
	& = \Psid{a_1-1}{1}\dots\Psid{a_r-1}{r}y_i\psi_i\Psid{i-1}{r+1}\Psid{a_{r+2}-1}{r+2}\dots\Psid{a_m-1}{m}z_{((n-m),(1^m))} \\
	& = \Psid{a_1-1}{1}\dots\Psid{a_r-1}{r}(\psi_iy_{i+1}-1)\Psid{i-1}{r+1}\Psid{a_{r+2}-1}{r+2}\dots\Psid{a_m-1}{m}z_{((n-m),(1^m))}\\
	& = \Psid{a_1-1}{1}\dots\Psid{a_r-1}{r}\psi_iy_{i+1} v(1,\dots,r,i,a_{r+2},\dots,a_m)\\
	&\ -\Psid{a_1-1}{1}\dots\Psid{a_r-1}{r}\Psid{i-1}{r+1}\Psid{a_{r+2}-1}{r+2}\dots\Psid{a_m-1}{m}z_{((n-m),(1^m))},
\end{align*}

	where $y_{i+1} v(1,\dots,r,i,a_{r+2},\dots,a_m)$ equals zero by $B'$, whilst the second term is clearly non-zero.
	
Finally, suppose that $i\equiv{\kappa_2-\kappa_1+2}\Mod{e}$, $i-1\not\in\{a_1,\dots,a_m\}$ and $i\in\{a_1,\dots,a_m\}$.
We let $a_r\leqslant{i-2}$ and $a_{r+1}=i$. Then, by \cref{ypsicomm} and \cref{ypsidown}, we have
	\begin{align*}
	 y_i v(a_1,\dots,a_m) 
	& = \Psid{a_1-1}{1}\dots \Psid{a_r-1}{r}y_i\psi_{i-1}\Psid{i-2}{r+1} \Psid{a_{r+2}-1}{r+2}\dots\Psid{a_m-1}{m}z_{((n-m),(1^m))} \\
	& = \Psid{a_1-1}{1}\dots \Psid{a_r-1}{r}(\psi_{i-1}y_{i-1}+1)\Psid{i-2}{r+1} \Psid{a_{r+2}-1}{r+2}\dots\Psid{a_m-1}{m}z_{((n-m),(1^m))}\\
	& = \Psid{a_1-1}{1}\dots \Psid{a_r-1}{r}\psi_{i-1}y_{i-1} v(1,\dots,r,i-1,a_{r+2},\dots,a_m)\\
	&\ +\Psid{a_1-1}{1}\dots \Psid{a_r-1}{r}\Psid{i-2}{r+1} \Psid{a_{r+2}-1}{r+2}\dots\Psid{a_m-1}{m}z_{((n-m),(1^m))},
	\end{align*}

where $y_{i-1} v(1,\dots,r,i-1,a_{r+2},\dots,a_m)$ equals zero by $A$, whilst the second term is clearly non-zero.
\end{proof}

\subsection{Specht module homomorphisms}

We now consider Specht module $\mathscr{H}_n^{\Lambda}$-homomorphisms $H:S_{\lambda}\rightarrow{S_{\mu}}$ such that $\lambda$, and similarly $\mu$, is either a hook bipartition or a bipartition with only one non-empty component that is a hook partition. Suppose that $\ttt\in\std(\lambda)$. Then it is apparent from \cref{lem:idemp} that the homomorphism $H$ maps $v_{\ttt}\in\mathscr{H}_n^{\Lambda}$ to either $0$ or a linear combination of standard basis elements $v_{\tts}$ for some $\tts\in\text{Std}(\mu)$ when $H$ is a non-trivial $\mathscr{H}_n^{\Lambda}$-homomorphism, in which case $\mathbf{i}_{\ttt}=\mathbf{i}_{\tts}$. 

For a standard $((n-m,1^m),\varnothing)$-tableau $\tts$, we write $b_j:=\tts(j,1,1)$ for all $j\in\{2,\dots,m+1\}$. Then $\tts$ is completely determined by $b_2,\dots,b_{m+1}$. Analogously to the homogeneous elements of $((n-m),(1^m))$, we write
\[v_{\tts}=\Psid{b_2-1}{2}\Psid{b_3-1}{3}\dots\Psid{b_{m+1}-1}{m+1}z_{((n-m,1^m),\varnothing)}
	\in S_{((n-m,1^m),\varnothing)}.\]
Thus $v_{\tts}$ is completely determined by $b_2,\dots,b_{m+1}$, and we write $v(b_2,\dots,b_{m+1}):=v_{\tts}$.
Similarly, for $\ttt\in\std(\varnothing,(n-m,1^m))$, we define $v(c_2,\dots,c_{m+1})$ to be $v_{\ttt}\in{S_{(\varnothing,(n-m,1^m))}}$. We note that it will be obvious throughout for which Specht module $v(-,\dots,-)$ belongs to.

\begin{prop}\label{prop:homs}
\begin{enumerate}
\item{
If $n\equiv{\kappa_2-\kappa_1+1}\Mod{e}$ and $0\leqslant{m}\leqslant{n-1}$, then there exists the following non-zero homomorphism of Specht modules
	\begin{align*}
	\gamma_m:
	S_{((n-m),(1^m))}
	&\longrightarrow
	S_{((n-m-1),(1^{m+1}))},\ 
	\gamma_m\left(z_{((n-m),(1^m))}\right)=v(1,\dots,m,n).
	\end{align*}
}
\item{
If $\kappa_2\equiv{\kappa_1-1}\Mod{e}$, then there exist the following two non-zero homomorphisms of Specht modules.
	\begin{enumerate}
	\item{For $1\leqslant{m}\leqslant{n-1}$, we have
	\begin{align*}
	\chi_m:
	S_{((n-m,1^m),\varnothing)}
	&\longrightarrow
	S_{((n-m),(1^m))},\ 
	\chi_m\left(z_{((n-m,1^m),\varnothing)}\right)=v(2,3,\dots,m+1).
	\end{align*}
	}
	\item{For $1\leqslant{m}\leqslant{n}$, we have
	\begin{align*}	
	\tau_m:
	S_{((n-m),(1^m))}
	&\longrightarrow
	S_{(\varnothing,(n-m+1,1^{m-1}))},\ 
	\tau_m\left(z_{((n-m),(1^m))}\right)=z_{(\varnothing,(n-m+1,1^{m-1}))}.
	\end{align*}
	}
	\end{enumerate}
}
\item{
If $\kappa_2\equiv{\kappa_1-1}\Mod{e}$ and $n\equiv{0}\Mod{e}$, then there exist the following three non-zero homomorphisms of Specht modules.
	\begin{enumerate}
	\item{ For $0\leqslant{m}\leqslant{n-2}$, we have
	\begin{align*}	
	\alpha_m:
	S_{((n-m,1^m),\varnothing)}
	&\longrightarrow
	S_{((n-m-1,1^{m+1}),\varnothing)},\ 
	\alpha_m\left(z_{((n-m,1^m),\varnothing)}\right)=v(2,\dots,m+1,n).
	\end{align*}
	}
	\item{ For $0\leqslant{m}\leqslant{n-2}$, we have
	\begin{align*}
	\beta_m:
	S_{(\varnothing,(n-m,1^m))}
	&\longrightarrow
	S_{(\varnothing,(n-m-1,1^{m+1}))},\ 
	\beta_m\left(z_{(\varnothing,(n-m,1^m))}\right)=v(2,\dots,m+1,n).
	\end{align*}
	}
	\item{ For $1\leqslant{m}\leqslant{n-1}$, we have
	\begin{align*}
	\phi_m:
	S_{((n-m+1,1^{m-1}),\varnothing)}
	&\longrightarrow
	S_{((n-m),(1^m))},\ 
	\phi_m\left(z_{((n-m+1,1^{m-1}),\varnothing)}\right)=v(2,\dots,m,n).
	\end{align*}
	}
	\end{enumerate}
}
\end{enumerate}
\end{prop}

\begin{proof}
Residues are taken modulo $e$ throughout. 
\begin{enumerate}
\item{ 
Firstly, let $m<n-1$. We know from \cref{lem:std} that the $((n-m-1),(1^{m+1}))$-tableau
\begin{align*}
\ttt=\s{n-1}{m+1}\ttt_{((n-m-1),(1^{m+1}))}
=\gyoungxy(2,1,;\one;\hdts;\enminusone,:,;1,|\sesqui\vdts,;m,;n)
\end{align*}
is standard, and hence $v(1,\dots,m,n)=\Psid{n-1}{m+1}z_{((n-m-1),(1^{m+1}))}\neq{0}$.

Recall the presentation of $S_{((n-m),(1^m))}$ as given in the first part of \cref{cor:spechtpres}. We show that $\Psid{n-1}{m+1}z_{((n-m-1),(1^{m+1}))}$ satisfies the defining relations that $z_{((n-m),(1^m))}$ satisfies.
We first observe that $\ttt_{((n-m),(1^m))}$ and $\s{n-1}{m+1}\ttt_{((n-m-1),(1^{m+1}))}$ share the same $e$-residue sequence, that is, $\mathbf{i}_{((n-m),(1^m))}=\mathbf{i}_{\ttt}=\s{n-1}{m+1}\mathbf{i}_{((n-m-1),(1^{m+1}))}$. By part $1$ of \cref{cor:spechtpres} and \cref{psiidemp}, we thus have
\begin{align*}
e(\mathbf{i}_{((n-m),(1^m))})\Psid{n-1}{m+1}z_{((n-m-1),(1^{m+1}))}
&=e_{\ttt}\Psid{n-1}{m+1}z_{((n-m-1),(1^{m+1}))}\\
&=\Psid{n-1}{m+1}e(\mathbf{i}_{((n-m-1),(1^{m+1}))})z_{((n-m-1),(1^{m+1}))}\\
&=\Psid{n-1}{m+1}z_{((n-m-1),(1^{m+1}))}.
\end{align*}

Applying \cref{ypsicomm} and \cref{psicomm}, it is clear that $\Psid{n-1}{m+1}z_{((n-m-1),(1^{m+1}))}$ is killed by $y_1,\dots,y_m$ and $\psi_1,\dots,\psi_{m-1}$.
	\begin{enumerate}[(i)]
	\item{
	Let $i\in\{m+1,\dots,n-1\}$. Then, by \cref{ypsicomm}, we have
	\[
	y_i\Psid{n-1}{m+1}z_{((n-m-1),(1^{m+1}))}
	=\Psid{n-1}{i+1}y_i\Psid{i}{m+1}z_{((n-m-1),(1^{m+1}))},
	\]	
	which clearly equals zero by part 3 of \cref{ygeneratorlem} if $i-\kappa_2+\kappa_1\not\equiv{1,2}\Mod{e}$. First suppose that $i\equiv{1+\kappa_2-\kappa_1}\Mod{e}$.
	It follows that $i\leqslant{n-e}\leqslant{n-3}$ since $i\equiv{n}\Mod{e}$ and $i<n$. We thus have
		\begin{align*}
		&
		\Psid{n-1}{i+1}
		y_i\psi_i\Psid{i-1}{m+1}z_{((n-m-1),(1^{m+1}))}
		=
		\Psid{n-1}{i+1}
		(\psi_iy_{i+1}-1)\Psid{i-1}{m+1}z_{((n-m-1),(1^{m+1}))}
		=0,
		\end{align*}
	by \cref{psiyup} and \cref{cor:spechtpres}.
	Now suppose that $i\equiv{2+\kappa_2-\kappa_1}\Mod{e}$. Then
		\begin{align*}
		\Psid{n-1}{i+1}
		y_i\psi_i\Psid{i-1}{m+1}
		z_{((n-m-1),(1^{m+1}))}
		&=
		\Psid{n-1}{i+1}
		\Psid{i}{m+1}
		y_{i+1}
		z_{((n-m-1),(1^{m+1}))}
		=0,
		\end{align*}
	by \cref{psiyup} and \cref{cor:spechtpres}.
	Finally,
		$
		y_n\Psid{n-1}{m+1}z_{((n-m-1),(1^{m+1}))}
		=y_nv(1,\dots,m,n)$,
	which is zero by part 1 of \cref{ygeneratorlem} since $n\equiv{1+\kappa_2-\kappa_1}\Mod{e}$.
	}
	\item{For $i\in\{m+1,\dots,n-2\}$, we have
	\begin{align*}
	&\psi_i\Psid{n-1}{m+1}z_{((n-m-1),(1^{m+1}))}\\
	&=
	\Psid{n-1}{i+2}
	\psi_i\psi_{i+1}\psi_i
	\Psid{i-1}{m+1}
	z_{((n-m-1),(1^{m+1}))}\\
	&=
	\begin{cases}
	\Psid{n-1}{i+2}
	\psi_{i+1}\psi_i\psi_{i+1}
	\Psid{i-1}{m+1}
	z_{((n-m-1),(1^{m+1}))}
	&\text{if $i\not\equiv\kappa_2-\kappa_1\Mod{e}$;}\\
	\Psid{n-1}{i+2}
	(\psi_{i+1}\psi_i\psi_{i+1}-1)
	\Psid{i-1}{m+1}
	z_{((n-m-1),(1^{m+1}))}
	&\text{if $i\equiv\kappa_2-\kappa_1\Mod{e}$}
	\end{cases}\\
	&=0,
	\end{align*}
	by applying \cref{psicomm} and \cref{psibraid}.

	We now let $i=n-1$ and observe that
	\begin{align*}
	\psi_{n-1}\Psid{n-1}{m+1}z_{((n-m-1),(1^{m+1}))}
	&=\psi_{n-1}^2\Psid{n-2}{m+1}z_{((n-m-1),(1^{m+1}))}\\
	&=(y_{n-1}-y_n)
	\Psid{n-2}{m+1}z_{((n-m-1),(1^{m+1}))},
	\end{align*}
	by \cref{psipsi}.
	The second term of this expression is clearly zero, and the first term is zero by part 3 of \cref{ygeneratorlem} since $n\equiv{\kappa_2-\kappa_1+1}\Mod{e}$.}
	\end{enumerate}
	
Finally, let $m=n-1$. Clearly, $z_{(\varnothing,(1^n))}$ is non-zero. We see that $\ttt_{((1),(1^{n-1}))}$ and $\ttt_{(\varnothing,(1^n))}$ share the $e$-residue sequence $\mathbf{i}_{((1),(1^{n-1}))}$ since
	$\res(1,1,1)
	=\kappa_1
	=\kappa_2+1-n
	=\res(2,n,1)$.
We thus have $e(\mathbf{i}_{((1),(1^{n-1}))})z_{(\varnothing,(1^n))}=z_{(\varnothing,(1^n))}$.
We note that the remaining relations are trivial.}

\item{ For the first part, we show that $\psi_1\dots\psi_mz_{((n-m),(1^m))}$ satisfies the defining relations that $z_{((n-m,1^m),\varnothing)}$ satisfies in the second part of \cref{cor:spechtpres} by using the Khovanov--Lauda--Rouquier algebra and Specht module presentations, and for the second part it suffices to check that $\ttt_{((n-m),(1^m))}$ and $\ttt_{(\varnothing,(n-m+1,1^{m-1}))}$ share the same $e$-residue sequence.}

\item{ For the first part, we show that $\Psid{n-1}{m+2}z_{((n-m-1,1^{m+1}),\varnothing)}$ satisfies the defining relations that $z_{((n-m,1^m),\varnothing)}$ satisfies in the second part of \cref{cor:spechtpres} by using the Khovanov--Lauda--Rouquier algebra and Specht module presentations, which the second part follows from since $\ttt_{(\varnothing,(n-m,1^m))}$ and $\s{n-1}{m+2}\ttt_{(\varnothing,(n-m-1,1^{m+1}))}$ share the same $e$-residue sequence.	For the third part, we show that $\psi_1\dots\psi_{m-1}\Psid{n-1}{m}z_{((n-m),(1^m))}$ satisfies the defining relations that $z_{((n-m+1,1^{m-1}),\varnothing)}$ satisfies, which can be deduced from \cref{cor:spechtpres}.	}
\qedhere
\end{enumerate}
\end{proof}

We can compose the above homomorphisms of Specht modules as follows.

\begin{lem}\label{lem:homcomps}
	If $\kappa_2\equiv{\kappa_1-1}\Mod{e}$ and $n\equiv{0}\Mod{e}$, then
	\begin{enumerate}
		\item{$\beta_{m-1}\circ{\tau_m}
			=\tau_{m+1}\circ{\gamma_m}$, and}
		\item{$\gamma_m\circ{\chi_m}
			=\chi_{m+1}\circ{\alpha_m}=\phi_{m+1}$.}
	\end{enumerate}
\end{lem}

\begin{proof} To check the above equalities, we show that the generator $z_{((n-m),(1^m))}$ satisfies them. We apply \cref{prop:homs} throughout.
	\begin{enumerate}
		\item{We have
			\begin{align*}
			\beta_{m-1}\circ{\tau_m}\left(z_{((n-m),(1^m))}\right)
			=
			\beta_{m-1}\left(
			z_{(\varnothing,(n-m+1,1^{m-1}))}
			\right)
			&=
			\Psid{n-1}{m+1}z_{(\varnothing,(n-m,1^m))}
			\\&=
			\tau_{m+1}\left(
			\Psid{n-1}{m+1}z_{((n-m-1),(1^{m+1}))}
			\right)
			\\&=
			\tau_{m+1}\circ{\gamma_m}\left(z_{((n-m),(1^m))}\right).
			\end{align*}}
		\item{We first observe that $\phi_{m+1}\left(z_{((n-m,1^m),\varnothing)}\right)=\Psiu{1}{m}\Psid{n-1}{m+1}z_{((n-m-1),(1^{m+1}))}$. By applying \cref{psicomm}, we see that
			\begin{align*}
			\gamma_m\circ{\chi_m}
			\left(z_{((n-m,1^m),\varnothing)}\right)
			&=
			\gamma_m
			\left(\Psiu{1}{m}
			z_{((n-m),(1^m))}\right)\\
			&=
			\Psiu{1}{m}
			\Psid{n-1}{m+1}
			z_{((n-m-1),(1^{m+1}))}
			\\&=
			\chi_{m+1}\left(
			\Psid{n-1}{m+2}
			z_{((n-m-1,1^{m+1}))}
			\right)
			\\&=
			\chi_{m+1}\circ{\alpha_m}
			\left(z_{((n-m,1^m),\varnothing)}\right).
			\end{align*}}
		\qedhere
	\end{enumerate}
\end{proof}

We now determine when the aforementioned Specht module homomorphisms act non-trivially.

\begin{lem}\label{image}
Let $\tts\in\std((n-m),(1^m))$, $\ttt\in\std((n-m,1^m),\varnothing)$ and $\ttu\in\std(\varnothing,(n-m,1^m))$, where $\tts$, $\ttt$ and $\ttu$ are determined by $\{a_1,\dots,a_m\}$, $\{b_2,\dots,b_{m+1}\}$ and $\{c_2,\dots,c_{m+1}\}$, respectively.
\begin{enumerate}
\item{Let $n\equiv{\kappa_2-\kappa_1+1}\Mod{e}$. Then $\gamma_m(v_{\tts})\neq{0}$ if and only if $a_m<n$, in which case $\gamma_m(v_{\tts})=v(a_1,\dots,a_m,n)\in{S_{((n-m-1),(1^{m+1}))}}$.
}
\item{
Let $\kappa_2\equiv{\kappa_1-1}\Mod{e}$.
	\begin{enumerate}
	\item{Then $0\neq\chi_m(v_{\ttt})=v(b_2,\dots,b_{m+1})\in{S_{((n-m),(1^m))}}$.
	}
	\item{
	Then $\tau_m(v_{\ttt})\neq{0}$ if and only if $a_1=1$, in which case $\tau_m(v_{\ttt})=v(1,a_2,\dots,a_m)\in{S_{(\varnothing,(n-m+1,1^{m-1}))}}$.	
	}
	\end{enumerate}
}
\item{
Let $\kappa_2\equiv{\kappa_1-1}\Mod{e}$ and $n\equiv{0}\Mod{e}$.
	\begin{enumerate}
	\item{
	Then $\alpha_m(v_{\ttt})\neq{0}$ if and only if $b_{m+1}<n$, in which case $\alpha_m(v_{\ttt})=v(b_2,\dots,b_{m+1},n)\in{S_{((n-m-1,1^{m+1}),\varnothing)}}$.
	}
	\item{
	Then $\beta_m(v_{\ttu})\neq{0}$ if and only if $c_{m+1}<n$, in which case $\beta_m(v_{\ttu})=v(c_2,\dots,c_{m+1},n)\in{S_{(\varnothing,(n-m-1,1^{m+1}))}}$.
	}
	\item{
	Then $\phi_{m+1}(v_{\ttt})\neq{0}$ if and only if $b_{m+1}<n$, in which case $\phi_{m+1}(v_{\ttt})=v(b_2,\dots,b_{m+1},n)\in{S_{((n-m-1),(1^{m+1}))}}$.
	}
	\end{enumerate}
}
\end{enumerate}
\end{lem}

\begin{proof}
We provide details only for the third part since the other parts are proved similarly. We write $\psi_{w_{\ttt}}=\Psid{b_2-1}{2}\dots\Psid{b_{m+1}-1}{m+1}$ and apply the third part of \cref{prop:homs} throughout.

	\begin{enumerate}[(a)]
	\item{
	Let $b_{m+1}<n$. Then, for $m<n-1$,
	\begin{align*}
	\alpha_m\left(\psi_{w_{\ttt}}z_{((n-m,1^m),\varnothing)}\right)
	& = \psi_{w_{\ttt}}\Psid{n-1}{m+2}z_{((n-m-1,1^{m+1}),\varnothing)} 
	= v(b_2,\dots,b_{m+1},n) \neq {0}.
	\end{align*}
	Instead, suppose that $b_{m+1}=n$. Then
	\begin{align*}
	\alpha_m\left(\psi_{w_{\ttt}}z_{((n-m,1^m),\varnothing)}\right)
	& = \Psid{b_2-1}{2}\dots\Psid{b_m-1}{m}\Psid{n-1}{m+1}\Psid{n-1}{m+2}
	z_{((n-m-1,1^{m+1}),\varnothing)}, \\
	& = \Psid{b_2-1}{2}\dots\Psid{b_m-1}{m}\psi_{n-1}
	v(b_2,\dots,b_m,n-1,n),
\end{align*}
which equals zero since $\psi_{n-1}v(b_2,\dots,b_m,n-1,n)$ is zero by part one of \cref{psigeneratorlem}.		
	}
	\item{Similar to the previous part.}
	\item{By applying parts 1 and 2(a) of this result together with the second part of \cref{lem:homcomps}, we have
	\begin{align*}
	\phi_{m+1}\left( \psi_{w_{\ttt}}z_{((n-m,1^m),\varnothing)} \right)
	& = \gamma_m\circ\chi_m\left(\psi_{w_{\ttt}}z_{((n-m,1^m),\varnothing)} \right)\\
	& = \gamma_m\left( \Psid{b_2-1}{1}\Psid{b_3-1}{2}\dots\Psid{b_{m+1}-1}{m}z_{((n-m),(1^m))} \right)\\
	& = \begin{cases}
			v(b_2,\dots,b_{m+1},n)\neq 0
			&\text{if $b_{m+1}<n$;}\\
			0&\text{if $b_{m+1}=n$.}
	\end{cases}
	\end{align*}	
	}
\qedhere
\end{enumerate}
\end{proof}

We thus have basis elements $v(a_1,\dots,a_m)\in{S_{((n-m),(1^m))}}$ and $v(b_2,\dots,b_{m+1})\in{S_{((n-m,1^m),\varnothing)}}$, where $v(a_1,\dots,a_m)$ corresponds to the standard $((n-m),(1^m))$-tableau with $a_1,\dots,a_m$ lying in its leg, and $v(b_2,\dots,b_{m+1})$ corresponds to the standard $((n-m,1^m),\varnothing)$-tableau with $b_2,\dots,b_{m+1}$ lying in its leg.

We can informally think of the action of $\gamma_m$ on $v(a_1,\dots,a_m)$ by its corresponding action on the standard $((n-m),(1^m))$-tableau determined by $a_1,\dots,a_m$, which moves the node $(1,n-m,1)$ containing entry $n$ to the addable node at the end of its leg as follows
\[
\gyoungxy(1.2,1.2,;;|\sesqui\hdts;;n,,\aone,|\sesqui\vdts,\aem)
\xmapsto{\:\gamma_m\:}
\gyoungxy(1.2,1.2,;;|\sesqui\hdts;,,\aone,|\sesqui\vdts,\aem,!\gr{n})\:.
\]
Homomorphisms $\alpha_m$ and $\beta_m$ act similarly on standard $((n-m,1^m),\varnothing)$- and $(\varnothing,(n-m,1^m))$-tableaux, respectively.

We now observe the action of $\chi_m$ on $v(b_2,\dots,b_{m+1})$ by its corresponding action on the standard $((n-m,1^m),\varnothing)$-tableau determined by $b_2,\dots,b_{m+1}$, which essentially splits its first row and its remaining rows into two separate components as follows
\[
\gyoungxy(1.9,1,;1;;|\sesqui\hdts;,\btwo,|\sesqui\vdts,\bemplusone,,:\emset)
\xmapsto{\:\chi_m\:}
\gyoungxy(1.9,1,;1;;|\sesqui\hdts;,,\btwo,|\sesqui\vdts,\bemplusone)\:.
\]

\subsection{Exact sequences of Specht modules}

We obtain exact sequences of Specht modules in this section, similar in nature to the exact sequence of $\mathscr{H}_n^{\Lambda}$-homomorphisms given in \cite[Corollary 5.17]{dels}. Following part 1 of \cref{image}, we introduce a useful bijection between sets of basis elements of Specht modules, which is a restriction of the Specht module homomorphisms $\gamma_m$ given above.

\begin{lem}\label{gamiso}
Let $n\equiv{\kappa_2-\kappa_1+1}\Mod{e}$.
Define
	\[M:=
	\{
	v_{\ttt}\mid
	\ttt\in\std\left((n-m),(1^m)\right),
	\ttt(1,n-m,1)=n
	\}\]
and
	\[
	N:=\{
	v_{\ttt}\mid
	\ttt\in\std\left((n-m-1),(1^{m+1})\right),
	\ttt(m+1,1,2)=n
	\}
	.\]
Then $\gamma_m$ restricts to a bijection from $M$ to $N$.
\end{lem}

We now determine standard basis elements of the kernels and the images of the Specht modules homomorphisms given in \cref{prop:homs}. It follows from \cref{gradedmod} that their bases are a subset of the bases of Specht modules labelled by hook bipartitions, whose basis elements are labelled by standard $((n-m),(1^m))$-tableaux.

\begin{lem}\label{imker}
\begin{enumerate}
\item{
If $n\equiv{\kappa_2-\kappa_1+1}\Mod{e}$, then
	\begin{enumerate}
	\item{
	$\im(\gamma_m)
	=\spn\left\{
	v_{\ttt}\:|\:\ttt\in\std\left((n-m-1),(1^{m+1})\right),
	\ttt(m+1,1,2)=n
	\right\}$;
	}
	\item{
	$\ker(\gamma_m)
	=\spn\left\{
	v_{\ttt}\:|\:\ttt\in\std\left((n-m),(1^m)\right),
	\ttt(m,1,2)=n
	\right\}$.}
	\end{enumerate}}
\item{If $\kappa_2\equiv\kappa_1-1\Mod{e}$, then
	\begin{enumerate}
	\item{\begin{enumerate}
		\item{$\im(\chi_m)=\spn\left\{v_{\ttt}\:|\:\ttt\in\std\left((n-m),(1^m)\right),\ttt(1,1,1)=1\right\}$; }
			\item{ $\ker(\chi_m)=0$; }
		\end{enumerate}}
	\item{\begin{enumerate}
		\item{ $\im(\tau_m)=S_{(\varnothing,(n-m+1,1^{m-1}))}$; }
		\item{ $\ker(\tau_m)=\spn\{v_{\ttt}\:|\:\ttt\in\std((n-m),(1^m)),\ttt(1,1,1)=1\}$. }
		\end{enumerate}}
	\end{enumerate}}
\item{If $\kappa_2\equiv\kappa_1-1\Mod{e}$ and $n\equiv{0}\Mod{e}$, then
	\begin{enumerate}
	\item{$\im(\alpha_m)=\spn\left\{
	v_{\ttt}\:|\:\ttt\in\std((n-m-1,1^{m+1}),\varnothing),\:\ttt(m+2,1,1)=n
	\right\}$;}
	\item{$\im(\beta_m)=\spn\left\{
	v_{\ttt}\:|\:\ttt\in\std(\varnothing,(n-m-1,1^{m+1}),\:\ttt(m+2,1,2)=n
	\right\}$;}
	\item{\begin{enumerate}
		\item If $m<n-1$, then
		\[\im(\phi_m)=\spn\left\{v_{\ttt}\mid
		\ttt\in\std\left((n-m),(1^m)\right),\ttt(1,1,1)=1,\ttt(m,1,2)=n\right\}.\]
		\item If $m=n-1$, then
		\[\im(\phi_m)=\spn\{
		v_{\ttt}\mid
		\ttt\in\std\left((1),(1^{n-1})\right),\ttt(1,1,1)=1\}.\]
		\end{enumerate}}
	\end{enumerate}}
\end{enumerate}
\end{lem}

\begin{proof}
The images of the Specht module homomorphisms $\gamma_m$, $\chi_m$, $\alpha_m$, $\beta_m$ and $\phi_m$ are immediate from \cref{image}. We subsequently determine the spanning sets of the respective kernels.
\end{proof}

An immediate consequence is the following result, which aids us in finding the composition factors of Specht modules labelled by hook bipartitions.

\begin{lem}\label{exact}
\begin{enumerate}
\item{
If $n\equiv{\kappa_2-\kappa_1+1}\Mod{e}$, then we have the following exact sequence
	\begin{align*}
	0
	\longrightarrow
	S_{((n),\varnothing)}
	\xrightarrow{\:\gamma_0\:}
	S_{((n-1),(1))}
	\xrightarrow{\:\gamma_1\:}
	S_{((n-2),(1^2))}
	\xrightarrow{\:\gamma_2\:}
	\cdots
	\xrightarrow{\gamma_{n-1}}
	S_{(\varnothing,(1^n))}
	\longrightarrow
	0.
	\end{align*}}
\item{ If $\kappa_2\equiv{\kappa_1-1}\Mod{e}$, then the following sequence is exact
\[
	0
	\longrightarrow
	S_{((n-m,1^m),\varnothing)}
	\xrightarrow{\:\chi_m\:}
	S_{((n-m),(1^m))}
	\xrightarrow{\:\tau_m\:}
	S_{(\varnothing,(n-m+1,1^{m-1}))}
	\longrightarrow
	0.
	\]}
\item{If $\kappa_2\equiv{\kappa_1-1}\Mod{e}$ and $n\equiv{0}\Mod{e}$, then the following sequences are exact:
	\begin{enumerate}
	\item{
	$0
	\longrightarrow
	S_{((n),\varnothing)}
	\xrightarrow{\:\alpha_0\:}
	S_{((n-1,1),\varnothing)}
	\xrightarrow{\:\alpha_1\:}
	S_{((n-2,1^2),\varnothing)}
	\xrightarrow{\:\alpha_2\:}
	\cdots
	\xrightarrow{\alpha_{n-2}}
	S_{((1^n),\varnothing)}
	\longrightarrow
	0$;
	}
	\item{$
	0
	\longrightarrow
	S_{(\varnothing,(n))}
	\xrightarrow{\:\beta_0\:}
	S_{(\varnothing,(n-1,1))}
	\xrightarrow{\:\beta_1\:}
	S_{(\varnothing,(n-2,1^2))}
	\xrightarrow{\:\beta_2\:}
	\cdots
	\xrightarrow{\beta_{n-2}}
	S_{(\varnothing,(1^n))}
	\longrightarrow
	0.$}
	\end{enumerate}}
\end{enumerate}
\end{lem}

We thus obtain a commutative diagram of exact sequences of Specht module homomorphisms by applying \cref{lem:homcomps}.

\begin{lem}\label{lem:pic}
If $\kappa_2\equiv{\kappa_1-1}\Mod{e}$ and $n\equiv{0}\Mod{e}$, then the following diagram consists entirely of exact sequences where every square and every triangle commutes:
\[
\begin{tikzcd}
& 0 \arrow[d]
& 0 \arrow[d] \\
0 \arrow[r]
& S_{((n),\varnothing)} \arrow[r,"\chi_0"]
\arrow[d,"\alpha_0"]
\arrow[rd,"\phi_1"]
& S_{((n),\varnothing)} \arrow[r]
\arrow[d,"\gamma_0"]
& 0 \arrow[d] \\
0 \arrow[r]
& S_{((n-1,1),\varnothing)} \arrow[r,"\chi_1"]
\arrow[d,"\alpha_1"]
\arrow[rd,"\phi_2"]
& S_{((n-1),(1))} \arrow[r,"\tau_1"]
\arrow[d,"\gamma_1"]
& S_{(\varnothing,(n))} \arrow[r]
\arrow[d,"\beta_0"]
& 0 \\
0 \arrow[r]
& S_{((n-2,1^2),\varnothing)} \arrow[r,"\chi_2"]
\arrow[d,"\alpha_2"]
\arrow[rd,"\phi_3"]
& S_{((n-2),(1^2))} \arrow[r,"\tau_2"]
\arrow[d,"\gamma_2"]
& S_{(\varnothing,(n-1,1))} \arrow[r]
\arrow[d,"\beta_1"]
& 0 \\
& \vdots \arrow[d,"\alpha_{n-1}"]
& \quad\ \vdots\ \quad \arrow[d,"\gamma_{n-1}"]
& \vdots \arrow[d,"\beta_{n-2}"] \\
0 \arrow[r]
& S_{((2,1^{n-2}),\varnothing)} \arrow[r,"\chi_{n-2}"]
\arrow[d,"\alpha_{n-2}"]
\arrow[rd,"\phi_{n-1}"]
& S_{((2),(1^{n-2}))} \arrow[r,"\tau_{n-2}"]
\arrow[d,"\gamma_{n-2}"]
& S_{(\varnothing,(3,1^{n-3}))} \arrow[r]
\arrow[d,"\beta_{n-3}"]
& 0 \\
0 \arrow[r]
& S_{((1^n),\varnothing)} \arrow[r,"\chi_{n-1}"]
\arrow[d]
& S_{((1),(1^{n-1}))} \arrow[r,"\tau_{n-1}"]
\arrow[d,"\gamma_{n-1}"]
& S_{(\varnothing,(2,1^{n-2}))} \arrow[r]
\arrow[d,"\beta_{n-2}"]
& 0 \\
& 0 \arrow[r]
& S_{(\varnothing,(1^n))} \arrow[r,"\tau_n"]
\arrow[d]
& S_{(\varnothing,(1^n))} \arrow[r]
\arrow[d]
& 0 \\
&
& 0
& 0
\end{tikzcd}
\]
\end{lem}

\section{Composition series of $S_{((n-m),(1^m))}$}\label{sec:compseries}

We now completely determine the composition factors of $S_{((n-m),(1^m))}$ for $\mathscr{H}_n^{\Lambda}$, up to isomorphism, with quantum characteristic at least three. In order to do so, we first provide a complete, explicit action of the $\mathscr{H}_n^{\Lambda}$-generators $\psi_1,\dots,\psi_{n-1}$ on the standard basis vectors of $S_{((n-m),(1^m))}$. We then use this $\mathscr{H}_n^{\Lambda}$-action to systematically show that we can map a standard basis vector $v_{\ttt}\in S_{((n-m),(1^m))}$, corresponding to an element $\psi_{w_{\ttt}}\in\mathscr{H}_n^{\Lambda}$ for a reduced expression of $w_{\ttt}\in\mathfrak{S}_n$, to another basis vector $v_{\tts}\in S_{((n-m),(1^m))}$ such that $\ttt\triangleright\tts$. This enables us to show that the quotients of the kernels and the images of the Specht module homomorphisms given in \cref{prop:homs} are, in fact, irreducible $\mathscr{H}_n^{\Lambda}$-modules, and hence arise as composition factors of $S_{((n-m),(1^m))}$.
We remark that the composition series of $S_{((n-m),(1^m))}$ split into four distinct cases, depending on whether $\kappa_2\equiv{\kappa_1-1}\Mod{e}$ or not and on whether $n\equiv{\kappa_2-\kappa_1+1}\Mod{e}$ or not.

\subsection{Further action of $\mathscr{H}_n^{\Lambda}$ on $S_{((n-m),(1^m))}$}\label{subsec:furtheraction}

In order to determine the irreducibility of $\mathscr{H}_n^{\Lambda}$-submodules of Specht modules labelled by hook bipartitions, we now establish results towards this end.

Each basis vector $v_{\ttt}$ of $S_{((n-m),(1^m))}$ equals $\psi_{w_{\ttt}}z_{\lambda}$ for a $\psi_{w_{\ttt}}\in\mathscr{H}_n^{\Lambda}$ and a reduced expression for $w_{\ttt}\in\mathfrak{S}_n$. We wish to determine the non-trivial mappings between these basis vectors by the generators $\psi_1,\dots,\psi_{n-1}\in\mathscr{H}_n^{\Lambda}$. Appealing to \cref{psigeneratorlem} and \cref{ygeneratorlem}, we explicitly describe the action of these generators on the basis vectors of $S_{((n-m),(1^m))}$, which act non-trivially only in a small number of cases.

\begin{thm}\label{rel}
Let $1\leqslant{l}\leqslant{n-1}$, $\ttt\in\std\left((n-m),(1^m)\right)$, and for $1\leqslant{r}\leqslant{m}$, set $a_r:=\ttt(r,1,2)$. Then $\psi_lv(a_1,\dots,a_m)=0$ except in the following cases.
\begin{enumerate}[(i)]
\item { Suppose that $a_r=l$ for some $1\leqslant{r}\leqslant{m}$, and that either $r=m$ or $a_{r+1}\geqslant{l+2}$. Then
	\begin{align} \label{rel1}	
	\psi_l v(a_1,\dots,a_m)
	= v(a_1,\dots,a_{r-1},l+1,a_{r+1},\dots,a_m).
	\end{align} }
\item { Suppose that $l\equiv{\kappa_2-\kappa_1}\Mod{e}$ and $l<n-1$.
	\begin{itemize}
	\item { Suppose $a_r=l+1$ and $a_{r+1}=l+2$ for some $1\leqslant{r}\leqslant{n-1}$, and that either $r=1$ or $a_{r-1}\leqslant{l-1}$. Then
		\begin{align}\label{rel2}
		\psi_lv(a_1,\dots,a_m)
		= v(a_1,\dots,a_{r-1},l,l+1,a_{r+2},\dots,a_m).
		\end{align}	}
	\item { Suppose $a_r=l+2$ for some $1\leqslant{r}\leqslant{m}$, and that either $r=1$ or $a_{r-1}\leqslant{l-1}$. Then
		\begin{align}\label{rel3}
		 \psi_l v(a_1,\dots,a_m)
		= -v(a_1,\dots,a_{r-1},l,a_{r+1},\dots,a_m).
		\end{align}	 }
	\end{itemize}	}
\item{ Suppose that $l\equiv{2+\kappa_2-\kappa_1}\Mod{e}$.
	\begin{itemize}
	\item { Suppose $a_r=l$ and $a_{r+1}=l+1$ for some $1\leqslant{r}\leqslant{m-1}$, and that either $r=1$ or $a_{r-1}\leqslant{l-2}$. Then
		\begin{align}\label{rel4}
		\psi_lv(a_1,\dots,a_m)
		= v(a_1,\dots,a_{r-1},l-1,l,a_{r+2},\dots,a_m).
		\end{align} }
	\item{ Suppose $a_r=l+1$ for some $1\leqslant{r}\leqslant{m}$, and that either $r=1$ or $a_{r-1}\leqslant{l-2}$. Then
		\begin{align}\label{rel5}
		\psi_l
		 v(a_1,\dots,a_m)
		= -v(a_1,\dots,a_{r-1},l-1,a_{r+1},\dots,a_m).	
		\end{align} }
	\end{itemize} }
\item { Suppose that $l+\kappa_1-\kappa_2\not\equiv{0,1,2}\Mod{e}$, $a_r=l+1$ for some $1\leqslant{r}\leqslant{m}$, and either $r=1$ or $a_{r-1}\leqslant{l-1}$. Then
	\begin{align}\label{rel6}
	 \psi_l v(a_1,\dots,a_m)
	= v(a_1,\dots,a_{r-1},l,a_{r+1},\dots,a_m).
	\end{align} }
\end{enumerate}
\end{thm}

\begin{proof}
We consider $\psi_lv(a_1,\dots,a_m)$ for all $a_r\geqslant{l}$.
\begin{enumerate}
\item {We let $a_r=l$ and suppose $a_{r+1}\geqslant{l+2}$. Then, by using \cref{psicomm}, we have
\begin{align*}
	\psi_lv(a_1,\dots,a_m)
	= & \Psid{a_1-1}{1}\dots\Psid{a_{r-1}-1}{r-1}
	\Psid{l}{r}\Psid{a_{r+1}-1}{r+1}
	\dots\Psid{a_m-1}{m} z_{((n-m),(1^m))} \\
	= & v(a_1,\dots,a_{r-1},l+1,a_{r+1},\dots,a_m),
\end{align*}
which satisfies \cref{rel1}. }

\item {Suppose $a_{r-1}+1\leqslant{l}\leqslant{a_r-3}$. Then, by using \cref{psicomm}, we have
\begin{align*}
	\psi_lv(a_1,\dots,a_m)
	& = \Psid{a_1-1}{1}\dots\Psid{a_{r-1}-1}{r-1}
	\Psid{a_r-1}{l+2}\psi_l\Psid{l+1}{r}
	\Psid{a_{r+1}-1}{r+1}\dots\Psid{a_m-1}{m} z_{((n-m),(1^m))} \\
	& = \Psid{a_1-1}{1}\dots\Psid{a_{r-1}-1}{r-1}\Psid{a_r-1}{l+2}\psi_lv(1,\dots,r-1,l+2,a_{r+1},\dots,a_m).
\end{align*}
By part two of \cref{psigeneratorlem}, $\psi_lv(1,\dots,r-1,l+2,a_{r+1},\dots,a_m)$ equals zero if $l\not\equiv{\kappa_2-\kappa_1}\Mod{e}$. Suppose instead $l\equiv{\kappa_2-\kappa_1}\Mod{e}$. Then, by using \cref{psicomm} and \cref{psibraid}, $\psi_lv(a_1,\dots,a_m)$ becomes
\begin{align*}
	&\Psid{a_1-1}{1}\dots\Psid{a_{r-1}-1}{r-1} \Psid{a_r-1}{l+2}
	\psi_{l}\psi_{l+1}\psi_{l}\Psid{l-1}{r}
	\Psid{a_{r+1}-1}{r+1}\dots\Psid{a_m-1}{m} z_{((n-m),(1^m))} \\
	&= \Psid{a_1-1}{1}\dots\Psid{a_{r-1}-1}{r-1} 
	\Psid{a_r-1}{l+2}(\psi_{l+1}\psi_{l}\psi_{l+1}-1)\Psid{l-1}{r}
	\Psid{a_{r+1}-1}{r+1}\dots\Psid{a_m-1}{m} z_{((n-m),(1^m))} \\
	&= \Psid{a_1-1}{1}\dots\Psid{a_{r-1}-1}{r-1} 
	\Psid{a_r-1}{r}\Psid{a_{r+1}-1}{l+3}\psi_{l+1}
	v(1,\dots,r,l+3,a_{r+2},\dots,a_m)\\
	& \quad - \Psid{a_1-1}{1}\dots\Psid{a_{r-1}-1}{r-1}
	 \Psid{a_r-1}{l+3}\Psid{l-1}{r}\Psid{a_{r+1}-1}{l+4}
	\psi_{l+2}v(1,\dots,r,l+4,a_{r+2},\dots,a_m).
\end{align*}
By part two of \cref{psigeneratorlem}, both $\psi_{l+1}v(1,\dots,r,l+3,a_{r+2},\dots,a_m)$ and 
$\psi_{l+2}v(1,\dots,r,l+4,a_{r+2},\dots,a_m)$ equal zero. }
	
\item {Let $a_r=l+2$ and suppose that $a_{r-1}\leqslant{l-1}$.
\begin{enumerate}[(i)]
	\item {Suppose $l\not\equiv{\kappa_2-\kappa_1}\Mod{e}$. By repeatedly applying \cref{psigeneratorlem} whilst employing the Khovanov--Lauda--Rouquier algebra and Specht module presentations, we find that $\psi_lv(a_1,\dots,a_m)=0$. }

	\item {Suppose $l\equiv{\kappa_2-\kappa_1}\Mod{e}$.	
	Then, by using \cref{psicomm} and \cref{psibraid}, we have
	\begin{align*}
		\psi_{l}v(a_1,\dots,a_m) 
		& = \Psid{a_1-1}{1}\dots\Psid{a_{r-1}-1}{r-1} 
		(\psi_{l}\psi_{l+1}\psi_{l})\Psid{l-1}{r}
		\Psid{a_{r+1}-1}{r+1}\dots\Psid{a_m-1}{m} z_{((n-m),(1^m))} \\
		& = \Psid{a_1-1}{1}\dots\Psid{a_{r-1}-1}{r-1} 
		(\psi_{l+1}\psi_{l}\psi_{l+1}-1)\Psid{l-1}{r}
		\Psid{a_{r+1}-1}{r+1}\dots\Psid{a_m-1}{m} z_{((n-m),(1^m))}\\
		& = \Psid{a_1-1}{1}\dots\Psid{a_{r-1}-1}{r-1} 
		\psi_{l+1}\psi_{l}\psi_{l+1}v(1,\dots,r-1,l,a_{r+1},\dots,a_m)\\
		&\ - v(a_1,\dots,a_{r-1},l,a_{r+1},\dots,a_m),
	\end{align*}
	
	where $\psi_{l+1}v(1,\dots,r-1,l,a_{r+1},\dots,a_m)$ equals zero by part two of 
	\cref{psigeneratorlem}, whilst the second term is clearly non-zero, and thus satisfies \cref{rel3}. }

\end{enumerate} }

	\item {\begin{enumerate}
			\item {Let $a_r=l+1$ and suppose that $a_{r-1}\leqslant{l-1}$.
	\begin{enumerate}[(i)]
		\item {Suppose $l\equiv{\kappa_2-\kappa_1}\Mod{e}$. If $l<n-1$ and $a_{r+1}\geqslant{l+3}$ or $l=n-1$, then by applying \cref{psigeneratorlem} and \cref{ygeneratorlem} whilst employing the Khovanov--Lauda--Rouquier algebra and Specht module presentations, we find that $\psi_lv(a_1,\dots,a_m)=0$. 

		If we now let $a_{r+1}=l+2$, then by applying \cref{ygeneratorlem} and \cref{psicomm}, we have that

		\begin{align*}
		\psi_lv(a_1,\dots,a_m)
			&=\Psid{a_1-1}{1}\dots\Psid{a_{r-1}-1}{r-1}
			 \Psid{l-1}{r}\Psid{l}{r+1}
			\Psid{a_{r+2}-1}{r+2}\dots\Psid{a_m-1}{m} z_{((n-m),(1^m))} \\
			&= v(a_1,\dots,a_{r-1},l,l+1,a_{r+2},\dots,a_m),
		\end{align*}
		which satisfies \cref{rel2}. }

		\item {Suppose $l\equiv{1+\kappa_2-\kappa_1}\Mod{e}$.
		Then, by using \cref{psipsi}, we have
		\begin{align*}
			\psi_{l}v(a_1,\dots,a_m)
			& = \Psid{a_1-1}{1}\dots 
			\psi_{l}^2\Psid{l-1}{r}\dots
			\Psid{a_m-1}{m}	z_{((n-m),(1^m))}=0.
		\end{align*} }

		\item {Suppose $l\equiv{2+\kappa_2-\kappa_1}\Mod{e}$. Then by applying \cref{ygeneratorlem}, we have
		
		\begin{align*}
			\psi_lv(a_1,\dots,a_m)
			&= - \Psid{a_1-1}{1}\dots\Psid{a_{r-1}-1}{r-1}
			 \Psid{l-2}{r}\Psid{a_{r+1}-1}{r+1}
			\dots\Psid{a_m-1}{m} z_{((n-m),(1^m))}\\
			&= - v(a_1,\dots,a_{r-1},l-1,a_{r+1},\dots,a_m),
		\end{align*}
		
		which is clearly non-zero and satisfies \cref{rel5} if $a_{r-1}\leqslant{l-2}$. However, if $r>1$
		and $a_{r-1}=l-1$, then the term becomes zero by applying \cref{psigeneratorlem}. }

		\item {Suppose $l+\kappa_1-\kappa_2\not\equiv{0,1,2}\Mod{e}$.	
		Then, by using \cref{psicomm} and \cref{psipsi}, we have
		\begin{align*}
			\psi_{l}v(a_1,\dots,a_m) 
			& = 
			\Psid{a_1-1}{1}\dots\Psid{a_{r-1}-1}{r-1} 
			\psi_{l}^2\Psid{l-1}{r}\Psid{a_{r+1}-1}{r+1}
			\dots\Psid{a_m-1}{m} z_{((n-m),(1^m))} \\
			& = \Psid{a_1-1}{1}\dots\Psid{a_{r-1}-1}{r-1}
			 \Psid{l-1}{r}\Psid{a_{r+1}-1}{r+1}
			\dots\Psid{a_m-1}{m} z_{((n-m),(1^m))} \\
			& = v(a_1,\dots,a_{r-1},l,a_{r+1},\dots,a_m),
		\end{align*}
		which satisfies \cref{rel6}. }
	\end{enumerate}} 
	
	\item{Suppose $a_r=a_{r-1}+1$.
	Firstly, suppose $l\not\equiv{2+\kappa_2-\kappa_1}\Mod{e}$. Then, by using \cref{psicomm}, we have
	\begin{align*}
		\psi_{l}v(a_1,\dots,a_m) 
		& = \Psid{a_1-1}{1}\dots\Psid{a_{r-1}-1}{r-1}
		\psi_{l}v(1,\dots,r-1,l,l+1,a_{r+2},\dots,a_m),
	\end{align*}
	where $\psi_{l}v(1,\dots,r-1,l,l+1,a_{r+2},\dots,a_m)$ equals zero by part one of 
	\cref{psigeneratorlem}.
	
	Now suppose that $l\equiv{2+\kappa_2-\kappa_1}\Mod{e}$. Then by applying \cref{psigeneratorlem}, we have that

	\begin{align*}
		 \psi_l v(a_1,\dots,a_m)
		& = \Psid{a_1-1}{1}\dots\Psid{a_{r-1}-1}{r-1}
		\Psid{l-2}{r}\Psid{l-1}{r+1}\Psid{a_{r+2}-1}{r+2}
		\dots\Psid{a_m-1}{m} z_{((n-m),(1^m))}\\
		&= v(a_1,\dots,a_{r-1},l-1,l,a_{r+2},\dots,a_m),
	\end{align*}
	
	which is clearly non-zero and hence satisfying \cref{rel4} if $r=1$ or $r>1$ and $a_{r-1}\leqslant{l-2}$. However, if $r>1$ and $a_{r-1}=l-1$, then this term becomes zero by applying \cref{psigeneratorlem}. }

\end{enumerate}}

\item{ Suppose $a_r=l$ and $a_{r+1}\geqslant{l+3}$. Then $\psi_lv(a_1,\dots,a_m)$ clearly satisfies \cref{rel1}.
	
	Now suppose $a_{r+1}={l+2}$. Then, by using \cref{psicomm}, we have
	\begin{align*}
		\psi_l v(a_1,\dots,a_m)
		& = \Psid{a_1-1}{1}\dots\Psid{a_{r-1}-1}{r-1}
		 \psi_l
		v(1,\dots,r-1,l,l+1,a_{r+2},\dots,a_m).
	\end{align*}
	By part one of \cref{psigeneratorlem}, $\psi_lv(1,\dots,r-1,l,l+1,a_{r+2},\dots,a_m)$ equals zero if 
	$l\not\equiv{2+\kappa_2-\kappa_1}\Mod{e}$. Suppose instead that $l\equiv{2+\kappa_2-\kappa_1}\Mod{e}$. Then by applying \cref{psigeneratorlem}, $\psi_lv(a_1,\dots,a_m)$ becomes
	
	\begin{align*}
		& \Psid{a_1-1}{1}\dots\Psid{a_{r-1}-1}{r-1}
		 \Psid{l-2}{r}\Psid{l-1}{r+1}
		\Psid{a_{r+2}-1}{r+2}\dots\Psid{a_m-1}{m} z_{((n-m),(1^m))} \\
		& = v(a_1,\dots,a_{r-1},l-1,l,a_{r+2},\dots,a_m),
	\end{align*}
	
	which is clearly non-zero if $a_{r-1}\leqslant{l-2}$, and thus satisfies \cref{rel4}. However, this term becomes zero by applying \cref{psigeneratorlem} if $a_{r-1}=l-1$.	}
\qedhere
\end{enumerate}
\end{proof}

\begin{cor}\label{cor:matrix}
For all $l\in\{1,\dots,n-1\}$, the matrix of the action of $\psi_l$ on $S_{((n-m),(1^m))}$ with respect to our chosen standard basis has at most one non-zero entry in each row and in each column.
\end{cor}

Ultimately, when $S_{((n-m),(1^m))}$ is irreducible, we will show that we can map an arbitrary element of $S_{((n-m),(1^m))}$ under the action of $\mathscr{H}_n^{\Lambda}$ to the standard generator $z_{((n-m),(1^m))}$. If $S_{((n-m),(1^m))}$ is not irreducible, then our following results aid us to map an arbitrary element of a composition factor $M$ to a single basis vector $v_{\ttt}$, where $\ttt$ is the least dominant $((n-m),(1^m))$-tableau labelling any basis vector of $M$.

We now define the set of all products of the Khovanov--Lauda--Rouquier algebra generators $\psi_1,\dots,\psi_{n-1}$, up to scalar, to be
\[
\scaleobj{1.2}{\Psi}:=\left\{\alpha\psi_{r_1}\psi_{r_2}\dots\psi_{r_k}\ |\ 1\leqslant r_i <n,1\leqslant i\leqslant k\in\mathbb{N},\alpha\in\mathbb{R} \right\}.
\]
We note that for an arbitrary element $\alpha\psi_{r_1}\dots\psi_{r_k}\in\scaleobj{1.2}{\Psi}$, we do not assume that the associated expression $s_{r_1}\dots s_{r_k}\in\mathfrak{S}_n$ is reduced. For $\tts,\ttt\in\std((n-m),(1^m))$, we explicitly map each standard basis vector $v_{\ttt}$ of $S_{((n-m),(1^m))}$ to another standard basis vector $v_{\tts}$ by an element $x\in\scaleobj{1.2}{\Psi}$, where $\tts$ is less dominant than $\ttt$, that is, $\tts\triangleleft\ttt$.

\begin{prop}\label{irr}
Suppose that $a_i>i$ for some $i\in\{1,\dots,m\}$.
Then there exists an element $x\in\scaleobj{1.2}{\Psi}$ such that $xv(1,\dots,i-1,a_i,a_{i+1},\dots,a_m)=v(1,\dots,i-1,i,a_{i+1},\dots,a_m)$, where $x$ is given as follows.
\begin{enumerate}
\item{Let $i\equiv{1+\kappa_2-\kappa_1}\Mod{e}$.
	\begin{enumerate}
	\item{If $a_i\equiv{1+\kappa_2-\kappa_1}\Mod{e}$,
	 then \[x=
		\begin{cases}
		-\psi_{a_i}\Psiu{i+1}{a_i-1}
			&\text{if }a_{i+1}=a_i+1,\\
		\Psiu{i+1}{a_i}
			&\text{if $a_{i+1}\geqslant{a_i+2}$ or ($i=m$ and $a_m<n$)}.
		\end{cases}\]}
	\item{If $a_i\equiv{2+\kappa_2-\kappa_1}\Mod{e}$ and
		\begin{enumerate}
		\item{$a_i=i+1$,
	 	then $x=\begin{cases}
	 	-\psi_{i+1}^2
	 	&\text{if $a_{i+1}=i+2$;}\\
	 	\psi_{i+1}^2
	 	&\text{if $a_{i+1}> i+2$ or ($i=m$ and $a_m<n$),}
	 	\end{cases}$}
		\item{$a_i>{i+1}$,
	 	then $x=\Psiu{i+1}{a_i-2}$.}
		\end{enumerate}	}
	\item{If $a_i+\kappa_1-\kappa_2\not\equiv{1,2}\Mod{e}$,
	then $x=\Psiu{i+1}{a_i-1}$.}
	\end{enumerate}}
\item{Let $i\equiv{2+\kappa_2-\kappa_1}\Mod{e}$ and suppose that $i\neq{1}$ when $\kappa_2\equiv{\kappa_1-1}\Mod{e}$.
	\begin{enumerate}
	\item{If $a_i\equiv{1+\kappa_2-\kappa_1}\Mod{e}$,
	 then 
	 \[x=\begin{cases}
	 	\psi_{a_i}\Psid{i}{i-1}\Psiu{i+1}{a_i-1}
	 		&\text{if }a_{i+1}=a_i+1,\\
	 	-\Psid{i}{i-1}\Psiu{i+1}{a_i}
	 		&\text{if $a_{i+1}\geqslant{a_i+2}$ or ($i=m$ and $a_m<n$)}.
	 	\end{cases}\]}
	\item{If $a_i\equiv{2+\kappa_2-\kappa_1}\Mod{e}$,
	 then $x=-\Psid{i}{i-1}\Psiu{i+1}{a_i-2}$;}
	\item{If $a_i+\kappa_1-\kappa_2\not\equiv{1,2}\Mod{e}$,
	 then $x=-\Psid{i}{i-1}\Psiu{i+1}{a_i-1}$.}
	\end{enumerate}}
\item{Let $i+\kappa_1-\kappa_2\not\equiv{1,2}\Mod{e}$.
	\begin{enumerate}
	\item{If $a_i\equiv{1+\kappa_2-\kappa_1}\Mod{e}$,
	 then \[x=
		\begin{cases}	
		-\psi_{a_i}\Psiu{i}{a_i-1}
			&\text{if }a_{i+1}=a_i+1,\\ 
	 	\Psiu{i}{a_i}
	 		&\text{if $a_{i+1}\geqslant{a_i+2}$ or ($i=m$ and $a_m<n$)}.
		\end{cases}\]}
	\item{If $a_i\equiv{2+\kappa_2-\kappa_1}\Mod{e}$,
	 then $x=\Psiu{i}{a_i-2}$;}
	\item{If $a_i+\kappa_1-\kappa_2\not\equiv{1,2}\Mod{e}$,
	 then $x=\Psiu{i}{a_i-1}$.}
	\end{enumerate}}
\end{enumerate}
\end{prop}

\begin{proof}
We provide the details only for part of 1(a); the other parts are similarly proved by induction. We suppose that $i\equiv{\kappa_2-\kappa_1+1}\Mod{e}$, $a_i\equiv{1+\kappa_2-\kappa_1}\Mod{e}$ and $a_{i+1}\geqslant{a_i+2}$.
	Then $a_i=i+ke$ for some $k>0$.
	We proceed by induction on $k$.
	Let $l$ be the residue of $\kappa_2-\kappa_1$ modulo $e$ and set $i=1+l$ and $a_i=1+l+e$ for the base case, so that
		\begin{align*}
		&\Psiu{2+\kappa_2-\kappa_1}{1+\kappa_2-\kappa_1+e}
		v(1,\dots,i-1,1+\kappa_2-\kappa_1+e,a_{i+1},\dots,a_m)
		\\
		=&
		\Psiu{2+\kappa_2-\kappa_1}{\kappa_2-\kappa_1+e}
		v(1,\dots,i-1,2+\kappa_2-\kappa_1+e,a_{i+1},\dots,a_m)
		&\text{(\cref{rel1})}
		\\
		=&
		-
		\Psiu{2+\kappa_2-\kappa_1}{\kappa_2-\kappa_1+e-1}
		v(1,\dots,i-1,\kappa_2-\kappa_1+e,a_{i+1},\dots,a_m)
		&\text{(\cref{rel3})}
		\\
		=&
		-
		\psi_{2+\kappa_2-\kappa_1}
		v(1,\dots,i-1,3+\kappa_2-\kappa_1,a_{i+1},\dots,a_m)
		&\text{(\cref{rel6})}
		\\
		=&
		v(1,\dots,i-1,1+\kappa_2-\kappa_1,a_{i+1},\dots,a_m),
		&\text{(\cref{rel5}) }
		\end{align*}
	as required.
	Now assume that $\Psiu{i+1}{i+ke}v(1,\dots,i-1,i+ke,a_{i+1},\dots,a_m)=v(1,\dots,i-1,i,a_{i+1},\dots,a_m)$ for some $k>0$.
	Observe
		\begin{align*}
		&
		\Psiu{i+1}{i+(k+1)e}
		v(1,\dots,i-1,i+(k+1)e,a_{i+1},\dots,a_m)
		\\
		=&
		\Psiu{i+1}{i+(k+1)e-1}
		v(1,\dots,i-1,i+(k+1)e+1,a_{i+1},\dots,a_m)
		&\text{(\cref{rel1})}
		\\
		=&
		-
		\Psiu{i+1}{i+(k+1)e-2}
		v(1,\dots,i-1,i+(k+1)e-1,a_{i+1},\dots,a_m)
		&\text{(\cref{rel3})}
		\\
		=&
		-
		\Psiu{i+1}{i+ke+1}
		v(1,\dots,i-1,i+ke+2,a_{i+1},\dots,a_m)
		&\text{(\cref{rel6})}
		\\
		=&
		\Psiu{i+1}{i+ke}
		v(1,\dots,i-1,i+ke,a_{i+1},\dots,a_m)
		&\text{(\cref{rel5})}
		\\
		=&
		v(1,\dots,i-1,i,a_{i+1},\dots,a_m),
		\end{align*}					
	by the inductive hypothesis as required.
\end{proof}

If $\kappa_2\equiv{\kappa_1-1}\Mod{e}$, then there exists no $x\in\scaleobj{1.2}{\Psi}$ for which $xv(2,a_2,\dots,a_m)=v(1,a_2,\dots,a_m)$. Instead, we map each basis vector $v(a_1,\dots,a_m)$ of $S_{((n-m),(1^m))}$ such that $a_1>2$ to $v_{\ttt}=\psi_1\psi_2\dots\psi_mz_{((n-m),(1^m))}$, where $2,3,\dots,m+1$ lie in the leg of $\ttt$.

\begin{lem}
Let $\kappa_2\equiv{\kappa_1-1}\Mod{e}$ and suppose that $a_1>2$. Then there exists an element $x\in\scaleobj{1.2}{\Psi}$ such that $xv(a_1,a_2,\dots,a_m)=v(2,a_2,\dots,a_m)$ as given in the following cases.
\begin{enumerate}
\item{If $a_1\equiv{0}\Mod{e}$, then $x=
	\begin{cases}
	-\psi_{a_1}
	\Psiu{2}{a_1-1}
	&\text{if $a_2=a_1+1$;}\\
	\Psiu{2}{a_1}
	&\text{if $a_2>a_1+1$ or $m=1$ and $a_1<n$}.
	\end{cases}$}
\item{If $a_1\equiv{1}\Mod{e}$, then 
$x=\Psiu{2}{a_1-2}$.}
\item{If $a_1\not\equiv{0,1}\Mod{e}$, then
$x=\Psiu{2}{a_1-1}$.}
\end{enumerate}
\end{lem}

The next result will be a useful addition for determining irreducibility of $\mathscr{H}_n^{\Lambda}$-modules in the following section.

\begin{cor}\label{cor:irr}
\begin{enumerate}
\item{Let $\kappa_2\not\equiv{\kappa_1-1}\Mod{e}$ and suppose that $a_i>i$ for some $i\in\{1,\dots,m\}$. Then there exists an element $x\in\scaleobj{1.2}{\Psi}$ such that
$
x
	v(1,\dots,i-1,a_i,\dots,a_m)
	= v(1,\dots,i,a_{i+1},\dots,a_m).
$}
\item{Let $\kappa_2\equiv{\kappa_1-1}\Mod{e}$.
\begin{enumerate}
\item{Suppose that $a_i>i+1$ for some $i\in\{1,\dots,m-1\}$. Then there exists an element $x\in\scaleobj{1.2}{\Psi}$ such that
$
x
	v(2,\dots,i,a_i,\dots,a_{m-1},n)
	= v(2,\dots,i+1,a_{i+1},\dots,a_{m-1},n).
$}
\item{Suppose that $a_i>i$ for some $i\in\{2,\dots,m\}$. Then there exists an element $x\in\scaleobj{1.2}{\Psi}$ such that
$
x
	v(1,\dots,i-1,a_i,\dots,a_m)
	= v(1,\dots,i,a_{i+1},\dots,a_m).
$}
\end{enumerate}}
\end{enumerate}
\end{cor}

\subsection{Linear combinations of basis vectors of $S_{((n-m),(1^m))}$}

To ascertain irreducibility of a non-zero submodule $M\in S_{((n-m),(1^m))}$, we need to show that this submodule is generated by any element in the basis of $M$. However, it is non-trivial that an arbitrary, non-zero submodule of $S_{((n-m),(1^m))}$ even contains a single basis element. To this end, we first introduce a result necessary for understanding the action of $\mathscr{H}_n^{\Lambda}$ on non-zero linear combinations of standard basis elements of $S_{((n-m),(1^m))}$.

\begin{prop}\label{lincomb}
Let $\tts,\ttt\in\std((n-m),(1^m))$ be distinct. Then there exists an $x\in\mathscr{H}_n^{\Lambda}$, which either lies in $\scaleobj{1.2}{\Psi}$ or is of the form $x=e(\mathbf{i})$ for some $\mathbf{i}\in I^n$, such that exactly one of $xv_{\tts}$ and $xv_{\ttt}$ is zero.
\end{prop}

\begin{proof}
Set $a_r:=\ttt(r,1,2)$ and $b_r:=\tts(r,1,2)$ for $1\leqslant{r}\leqslant{m}$ and first suppose that $a_m=b_m$.

If $a_r>b_r$ then observe that $\ttt(r,1,2)=\tts(1,a_r-r,1)=a_r$ where $\res(r,1,2)={\kappa_2+1-r}$, $\res(1,a_r-r,1)={\kappa_1+a_r-r-1}$. Hence, $e_{\tts}\neq{e_{\ttt}}$ if $a_r\not\equiv{2+\kappa_2-\kappa_1}\Mod{e}$. We thus have $e_{\ttt}(v(a_1,\dots,a_m)+v(b_1,\dots,b_m))=v(a_1,\dots,a_m)$ and $e_{\tts}(v(a_1,\dots,a_m)+v(b_1,\dots,b_m))=v(b_1,\dots,b_m)$, by \cref{psiidemp} and the first part of \cref{cor:spechtpres}.

If $m=1$ then $a_1=i>b_1$. We have $e_{\ttt}\neq{e_{\tts}}$ if $i\not\equiv{2+\kappa_2-\kappa_1}\Mod{e}$. So suppose $i\equiv{2+\kappa_2-\kappa_1}\Mod{e}$. Then $\psi_iv(a_1)=v(i+1)\neq 0$ by \cref{rel1}, whereas $\psi_iv(b_1)=\psi_i\Psid{b_1-1}{1}z_{\lambda}=\Psid{b_1-1}{1}\psi_iz_{((n-m),(1^m))}=0$, by \cref{psicomm} and \cref{cor:spechtpres}.
		
Now suppose $m>1$ and let $r$ be maximal such that $a_r\neq{b_r}$ without loss of generality. Set $a_r=i>b_r$ and $a_{r+1}=b_{r+1}=j$.
\begin{enumerate}
	\item{Suppose $j\geqslant{i+3}$.
	Then $\psi_iv(b_1,\dots,b_m)=0$ by \cref{rel}, whereas, by \cref{rel1},
	$\psi_iv(a_1,\dots,a_m)
	=v(a_1,\dots,a_{r-1},i+1,j,a_{r+2})\neq 0$.
}
	\item{Suppose $j=i+2$.
		\begin{enumerate}
		\item{Suppose $j\not\equiv{2+\kappa_2-\kappa_1}\Mod{e}$.
		Then $\psi_iv(b_1,\dots,b_m)=0$ by \cref{rel}, whereas	$\psi_iv(a_1,\dots,a_m)=v(a_1,\dots,a_{r-1},i+1,i+2,a_{r+2},\dots,a_m)\neq 0$, by \cref{rel1}.}
		\item{Suppose $j\equiv{2+\kappa_2-\kappa_1}\Mod{e}$.
		Then $e_{\tts}\neq{e_{\ttt}}$ since $i\not\equiv{2+\kappa_2-\kappa_2}\Mod{e}$.}
		\end{enumerate}}
	\item{Suppose that $j=i+1$.	
		\begin{enumerate}
		\item{Suppose $j\not\equiv{3+\kappa_2-\kappa_1}\Mod{e}$. Then $e_{\tts}\neq{e_{\ttt}}$ since $i\not\equiv{2+\kappa_2-\kappa_1}\Mod{e}$.}
		\item{Suppose $j\equiv{3+\kappa_2-\kappa_1}\Mod{e}$.	
		Firstly, let $r\neq{1}$.
			\begin{enumerate}
			\item{Suppose $a_{r-1}=i-1$. We know from \cref{rel} that $\psi_{i-1}v(a_1,\dots,a_m)=0$ and $\psi_{i}v(a_1,\dots,a_m)=0$. However, if $b_r=i-1$ then $\psi_{i-1}v(b_1,\dots,b_m)=v(b_1,\dots,b_{r-1},i,i+1,b_{r+2},\dots,b_m)\neq 0$, by \cref{rel1}; if $b_r\leqslant{i-2}$ then $\psi_iv(b_1,\dots,b_m)
				=-v(b_1,\dots,b_r,i-1,b_{r+2},\dots,b_m)\neq 0$, by \cref{rel5}.}
			\item{Suppose $a_{r-1}\leqslant{i-2}$.
				\begin{enumerate}
				\item{Suppose $b_r=i-1$. 
				Then we know from \cref{rel4} that $\psi_iv(a_1,\dots,a_m)=v(a_1,\dots,a_{r-1},i-1,i,a_{r+2},\dots,a_m)\neq 0$, whereas $\psi_iv(b_1,\dots,b_m)=0$ by \cref{rel}.}
				\item{Suppose $b_r\leqslant{i-2}$. Then we know from \cref{rel1} and \cref{rel4} that
					\begin{align*}
					\psi_i^2v(a_1,\dots,a_m)
					=&\psi_i
					v(a_1,\dots,a_r
					,i-1,i,a_{r+2},\dots,a_m)\\
					=&
					v(a_1,\dots,a_r
					,i-1,i+1,a_{r+2}
					,\dots,a_m)\neq 0,	
					\end{align*}	
				whereas	we know from \cref{rel}, and in particular \cref{rel5}, that $\psi_i^2v(b_1,\dots,b_m)=-\psi_iv(b_1,\dots,b_r,i-1,b_{r+2},\dots,b_m)=0$.	
					}
				\end{enumerate}}
			\end{enumerate}
	Now let $r=1$.
	If $b_1=i-1$, then $\psi_iv(a_1,\dots,a_m)=v(i-1,i,a_3,\dots,a_m)\neq 0$, by \cref{rel4}, whereas $\psi_iv(b_1,\dots,b_m)=0$ by \cref{rel}.
	If $b_1\leqslant{i-2}$ then it follows from \cref{rel}, in particular \cref{rel4}, that $\psi_{i-1}\psi_iv(a_1,\dots,a_m)=\psi_{i-1}v(i-1,i,a_3,\dots,a_m)=0$, whereas we know from \cref{rel1} and \cref{rel5} that
			\begin{align*}
			\psi_{i-1}\psi_iv(b_1,\dots,b_m)=&
			-\psi_{i-1}v(b_1,i-1,a_3,\dots,a_m)
			= -v(b_1,i,a_3,\dots,a_m)\neq 0.
			\end{align*}
		}
		\end{enumerate}}
\end{enumerate}

Now suppose that $a_m\neq{b_m}$. It is sufficient to consider the following three cases:
	\begin{itemize}
	\item{$a_{m-1}=b_{m-1}=i$, $a_{m}=j>b_{m}$,}
	\item{$a_{m-1}=i>b_{m-1}$, $a_{m}=j>b_{m}$,}
	\item{$a_{m-1}=i<b_{m-1}$, $a_{m}=j>b_{m}$.}
	\end{itemize}
Observe that $\ttt(m,1,2)=\tts(1,j-m,1)=j$ where $\res(m,1,2)={\kappa_1+1-m}$ and $\res(1,j-m,1)={\kappa_1+j-m-1}$. Thus, if $j\not\equiv{2+\kappa_2-\kappa_1}\Mod{e}$ then $e_{\ttt}\neq{e_{\tts}}$. We now suppose that $j\equiv{2+\kappa_2-\kappa_1}\Mod{e}$. If $j<n$ then
	\[
	\psi_j(v(a_1,\dots,a_m)+v(b_1,\dots,b_m))
	=\psi_jv(a_1,\dots,a_m)=
	v(a_1,\dots,a_{m-1},j+1)\neq 0,
	\]
by \cref{rel1}. Now suppose $j=n$ and let $b_{m}\leqslant{n-2}$. We have $\ttt(1,n-m,1)=\tts(1,n-m-1,1)$ where $\res(1,n-m,1)={\kappa_2+1-m}\neq{\kappa_2-m}=\res(1,n-m-1,1)$, and thus $e_{\ttt}\neq{e_{\tts}}$. Whereas, if $b_{m}=n-1$, then it follows from \cref{rel1} that
	\[
	\psi_{n-1}(v(a_1,\dots,a_m)+v(b_1,\dots,b_m))
	=
	\psi_{n-1}v(b_1,\dots,b_m)
	=
	(b_1,\dots,b_{m-1},n)\neq 0.
	\] 
\end{proof}

\begin{lem}\label{lem:lincomb}
Any non-zero submodule of $S_{((n-m),(1^m))}$ contains a standard basis vector $v_{\ttt}$ for some $\ttt\in\std((n-m),(1^m))$.
\end{lem}

\begin{proof}
Let $0\neq{M}\subseteq{S_{((n-m),(1^m))}}$ and consider an arbitrary non-zero element $v$ of $M$. Then
\[
v=\sum_{i=1}^r c_iv_{{\ttt}_i}, \quad {c_i}\in\mathbb{F}\backslash\{0\},
\]
where $v_{{\ttt}_i}\neq{v_{{\ttt}_j}}$ for all $1\leqslant{i}\neq{j}\leqslant{r}$.
We can instead replace $v$ with $e(\mathbf{i})v$ such that $e(\mathbf{i})v=v$ for some $\mathbf{i}\in I^n$, and thus assume that $\mathbf{i}_{\ttt_i}=\mathbf{i}_{\ttt_j}$ for all $i,j\in\{1,\dots,r\}$.

We choose $v$ with $r\geqslant{1}$ minimal. If $r=1$, we are done, so we now suppose that $r>1$.
It thus follows from \cref{lincomb} that we can find an $x\in\scaleobj{1.2}{\Psi}$ such that exactly one of $xv_{\ttt_1}$ and $xv_{\ttt_2}$ is zero. Without loss of generality, we let $xv_{\ttt_1}=0$. Then we have
\[
M\ni{xv}=\sum_{i=2}^r\pm{c_iv_{\ttt_i}},
\]
where, for all $i\in\{2,\dots,r\}$, $x(c_iv_{\ttt_i})$ equals zero or $\pm c_iv_{\tts_i}$ for some standard $((n-m),(1^m))$-tableau $\tts_i$.
Moreover, by \cref{cor:matrix}, we know that $i=j$ whenever $\tts_i=\tts_j$, for all $i,j\in\{2,\dots,r\}$.
We have thus contradicted the minimality of $r$, and hence there must exist an $x\in\scaleobj{1.2}{\Psi}$ such that $xv=v_{\ttt}$ for some $v_{\ttt}\in{S_{((n-m),(1^m))}}$.
\end{proof}

\subsection{Case I: $S_{((n-m),(1^m))}$ with $\kappa_2\not\equiv{\kappa_1-1}\Mod{e}$ and $n\not\equiv{\kappa_2-\kappa_1+1}\Mod{e}$}

For this case, we claim that Specht modules labelled by hook bipartitions are irreducible, and thus are generated by any standard basis element of $S_{((n-m),(1^m))}$.

\begin{thm}\label{thm:case1comp}
Suppose that $\kappa_2\not\equiv{\kappa_1-1}\Mod{e}$ and $n\not\equiv{\kappa_2-\kappa_1+1}\Mod{e}$. Then $S_{((n-m),(1^m))}$ is an irreducible $\mathscr{H}_n^{\Lambda}$-module for all $m\in\{0,\dots,n\}$.
\end{thm}

\begin{proof}
Since we know from \cref{lem:lincomb} that any non-sero submodule of $S_{((n-m),(1^m))}$ contains a standard basis element $v_{\ttt}$ for some $\ttt\in\std((n-m),(1^m))$, it suffices to show that any standard basis element $v(a_1,\dots,a_m)$ generates $S_{((n-m),(1^m))}$.
We know that $v(1,\dots,m)=z_{((n-m),(1^m))}$ generates $S_{((n-m),(1^m))}$, so we now let $i$ be minimal such that $a_i>i$, for some $i\in\{1,\dots,m\}$, and proceed by downwards induction on $i$. 
By \cref{cor:irr}, there exists an element $x\in\scaleobj{1.2}{\Psi}$ such that
$xv(1,\dots,i-1,a_i,\dots,a_m)=v(1,\dots,i,a_{i+1},\dots,a_m)$.
By induction, we know that $v(1,\dots,i,a_{i+1},\dots,a_m)$ generates $S_{((n-m),(1^m))}$, and thus $v(a_1,\dots,a_m)$ also generates $S_{((n-m),(1^m))}$.
\end{proof}

\begin{ex}
	Set $e=3$ and $\kappa=(0,0)$. We know from \cref{thm:case1comp} that $S_{((2),(1^3))}$ is an irreducible $\mathscr{H}_5^{\Lambda}$-module, so for each $\tts,\ttt\in\std((2),(1^3))$ there exists an element $x\in\scaleobj{1.2}{\Psi}$ for which $x{v_{\ttt}}=v_{\tts}$. Recall that a standard $((2),(1^3))$-tableau is completely determined by the three entries in its leg. We represent the basis elements of $S_{((2),(1^3))}$ by the legs of the corresponding $((2),(1^3))$-tableaux, together with the only non-trivial relations between these elements. Observe that we can find a directed path from any standard $((2),(1^3))$-tableau to any other standard $((2),(1^3))$-tableau, as expected.
	
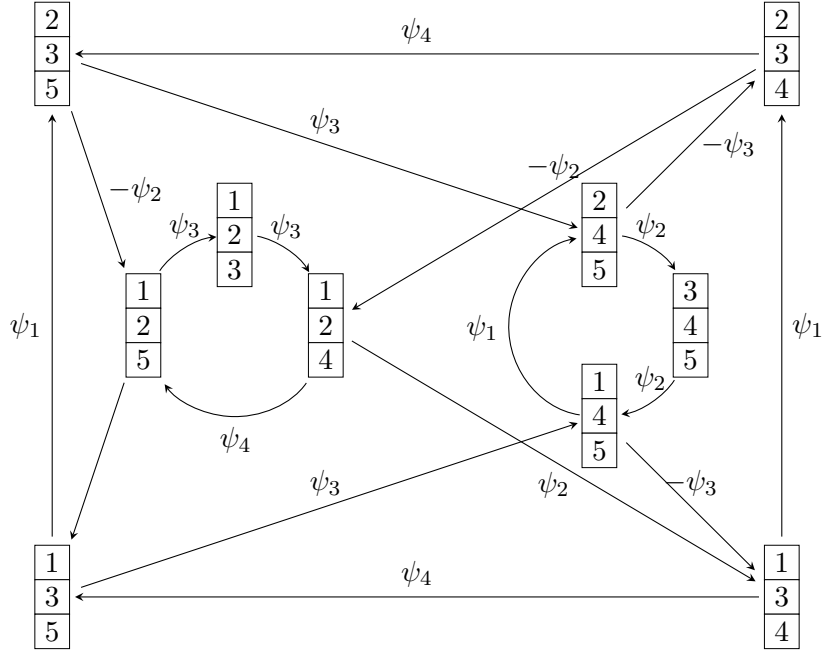
\[
	\begin{tikzpicture}[scale=0.4,>=stealth]
	\fill [] (-12,9) node {$\young(2,3,5)$};
	\fill [] (-12,-9) node {$\young(1,3,5)$};
	\fill [] (12,9) node {$\young(2,3,4)$};
	\fill [] (12,-9) node {$\young(1,3,4)$};
	\fill [] (-9,0) node {$\young(1,2,5)$};
	\fill [] (-3,0) node {$\young(1,2,4)$};
	\fill [] (-6,3) node {$\young(1,2,3)$};
	\fill [] (6,3) node {$\young(2,4,5)$};
	\fill [] (9,0) node {$\young(3,4,5)$};
	\fill [] (6,-3) node {$\young(1,4,5)$};
	\draw [shorten >=8mm,shorten <=8mm,->] (-12,-9) -- (-12,9)node [pos=0.5, left] {$\psi_{1}$};
	\draw [shorten >=3mm,shorten <=4mm,->] (-12,-9)--(6,-3)node [pos=0.5, above] {$\psi_{3}$};
	\draw [shorten >=8mm,shorten <=8mm,->] (-9,0)--(-12,-9);
	\draw [shorten >=4mm,shorten <=4mm,->] (-3,0) -- (12,-9)node [pos=0.5, below] {$\psi_{2}$};
	\draw [shorten >=3mm,shorten <=3mm,->] (12,-9) -- (-12,-9)node [pos=0.5, above] {$\psi_{4}$};
	\draw [shorten >=8mm,shorten <=8mm,->] (12,-9) -- (12,9)node [pos=0.5, right] {$\psi_{1}$};
	\draw [shorten >=5mm,shorten <=5mm,->] (6,-3) -- (12,-9)node [pos=0.5, above] {$-\psi_{3}$};
	\draw [shorten >=3mm,shorten <=3mm,->] (12,9) -- (-12,9) node [pos=0.5, above] {$\psi_{4}$};
	\draw [shorten >=4mm,shorten <=4mm,->] (12,9) -- (-3,0)node [pos=0.5, above] {$-\psi_{2}$};
	\draw [shorten >=3mm,shorten <=4mm,->] (-12,9) -- (6,3)node [pos=0.5, above] {$\psi_{3}$};
	\draw [shorten >=8mm,shorten <=8mm,->] (-12,9) -- (-9,0)node [pos=0.5, right] {$-\psi_{2}$};
	\draw [shorten >=5mm,shorten <=5mm,->] (6,3) -- (12,9)node [pos=0.5, right] {$-\psi_{3}$};
	\draw [shorten >=1mm,shorten <=1mm,->] (-3.5,-1.7) arc (-34:-145:3) node [pos=0.5, below] {$\psi_{4}$};
	\draw [shorten >=1mm,shorten <=1mm,->] (-8.6,1.6) arc (-213:-265:3) node [pos=0.5, above] {$\psi_{3}$};
	\draw [shorten >=1mm,shorten <=1mm,->] (-5.5,3) arc (-280:-330:3) node [pos=0.5, above] {$\psi_{3}$};
	\draw [shorten >=1mm,shorten <=1mm,->] (5.6,-3) arc (-98:-260:3)node [pos=0.5,left] {$\psi_{1}$};
	\draw [shorten >=1mm,shorten <=1mm,->] (6.5,3) arc (-278:-328:3)node [pos=0.5,above] {$\psi_{2}$};
	\draw [shorten >=1mm,shorten <=1mm,->] (8.6,-1.6) arc (-32:-82:3)node [pos=0.5,above] {$\psi_{2}$};
	\end{tikzpicture}
	\]
\end{ex}

\subsection{Case II: $S_{((n-m),(1^m))}$ with $\kappa_2\not\equiv{\kappa_1-1}\Mod{e}$ and $n\equiv{\kappa_2-\kappa_1+1}\Mod{e}$}

We claim that the images of the homomorphisms $\gamma_m$, which are generated by $v(1,\dots,m,n)\in{S_{((n-m-1),(1^{m+1}))}}$, appear as composition factors of Specht modules labelled by hook bipartitions. We first determine the irreducibility of these $\mathscr{H}_n^{\Lambda}$-modules.

\begin{thm}\label{imgam}
Suppose that $\kappa_2\not\equiv{\kappa_1-1}\Mod{e}$ and $n\equiv{\kappa_2-\kappa_1+1}\Mod{e}$. Then $\im(\gamma_m)$ is an irreducible $\mathscr{H}_n^{\Lambda}$-module for all $m\in\{0,\dots,n-1\}$.
\end{thm}

\begin{proof}
	It follows from \cref{imker} part $1$(a) that an arbitrary element of $\im(\gamma_m)$ is of the form $v(a_1,\dots,a_m,n)$. Similarly to the proof of \cref{thm:case1comp}, we use \cref{lem:lincomb} and part one of \cref{cor:irr} to show that $v(a_1,\dots,a_m,n)$ generates $\im(\gamma_m)$.
\end{proof}

An immediate consequence of this theorem, together with part one of \cref{exact}, is the following result.

\begin{cor}\label{thm:case2comp}
Suppose that $\kappa_2\not\equiv{\kappa_1-1}\Mod{e}$ and $n\equiv{\kappa_2-\kappa_1+1}\Mod{e}$. Then $S_{((n-m),(1^m))}$ has the composition series \[0\subset\im(\gamma_{m-1})\subset S_{((n-m),(1^m))}\]
for all $m\in\{1,\dots,n-1\}$, where $S_{((n-m),(1^m))}/\im(\gamma_{m-1})\cong\im(\gamma_m)$.
\end{cor}

\begin{ex}
	Let $e=3$ and $\kappa=(0,1)$. By \cref{thm:case2comp}, $S_{((2),(1^3))}$ has the composition series $0\subset\im(\gamma_2)\subset{S_{((2),(1^3))}}$, where $S_{((2),(1^3))}/\im(\gamma_2)\cong\im(\gamma_3)$.
	We know from \cref{imker} that the basis elements of $\im(\gamma_2)=\left\langle{z_{((2),(1^3))}}\right\rangle$ correspond to the $((2),(1^3))$-tableaux
	\[
	\gyoung(4!\gr5,,!\wh1,2,3),\ \gyoung(3!\gr5,,!\wh1,2,4),\ \gyoung(2!\gr5,,!\wh1,3,4),\ \gyoung(1!\gr5,,!\wh2,3,4),
	\]
	and the basis elements of $\im(\gamma_3)=\left\langle{\psi_{4}\psi_{3}z_{((2),(1^3))}}\right\rangle$ correspond to the $((2),(1^3))$-tableaux
	\[
	\gyoung(34,,1,2,!\gr5),\ \gyoung(24,,1,3,!\gr5),\ \gyoung(23,,1,4,!\gr5),\ \gyoung(14,,2,3,!\gr5),\ \gyoung(13,,2,4,!\gr5),\ \gyoung(12,,3,4,!\gr5).
	\]
	
	Observe that for any $v_\ttr,v_\tts\in\im(\gamma_2)$ and $v_{\ttt},v_{\ttu}\in\im(\gamma_3)$ we can find a directed path from $\ttr$ to $\tts$ and from $\ttt$ to $\ttu$, respectively, as follows, where the basis elements of $S_{((2),(1^3))}$ are represented by the legs of the corresponding $((2),(1^3))$-tableaux.
	
	
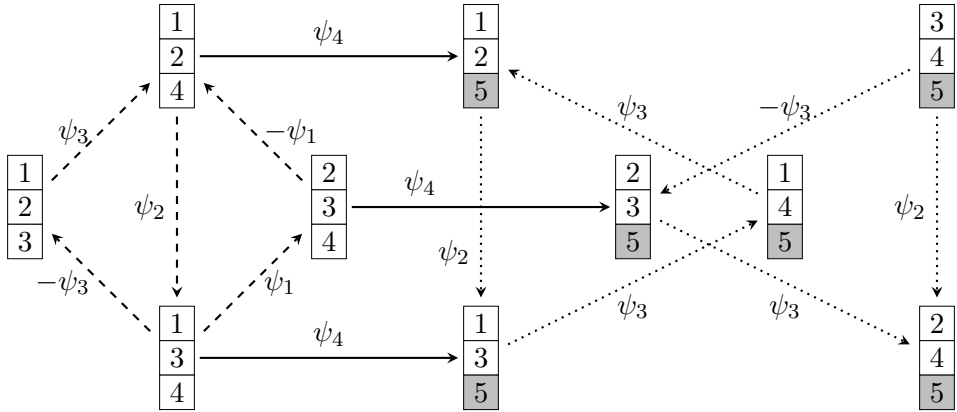
\[
	\begin{tikzpicture}[scale=0.5,>=stealth]
	\fill [] (-8,-4) node {$\young(1,3,4)$};
	\fill [] (-8,4) node {$\young(1,2,4)$};
	\fill [] (-12,0) node {$\young(1,2,3)$};
	\fill [] (-4,0) node {$\young(2,3,4)$};
	\fill [] (0,4) node {$\gyoung(1,2,!\gr5)$};
	\fill [] (0,-4) node {$\gyoung(1,3,!\gr5)$};
	\fill [] (4,0) node {$\gyoung(2,3,!\gr5)$};
	\fill [] (8,0) node {$\gyoung(1,4,!\gr5)$};
	\fill [] (12,4) node {$\gyoung(3,4,!\gr5)$};
	\fill [] (12,-4) node {$\gyoung(2,4,!\gr5)$};
	\draw [dashed,shorten >=3mm,shorten <=3mm,thick,->] (-8,3) -- (-8,-3)node [pos=0.5, left] {$\psi_{2}$};
	\draw [dashed,shorten >=5mm,shorten <=5mm,thick,->] (-12,0) -- (-8,4)node [pos=0.5, left] {$\psi_{3}$};
	\draw [dashed,shorten >=5mm,shorten <=5mm,thick,->] (-8,-4) -- (-12,0)node [pos=0.5, left] {$-\psi_{3}$};
	\draw [dashed,shorten >=5mm,shorten <=5mm,thick,->] (-4,0) -- (-8,4)node [pos=0.5, right] {$-\psi_{1}$};
	\draw [dashed,shorten >=5mm,shorten <=5mm,thick,->] (-8,-4) -- (-4,0)node [pos=0.5, right] {$\psi_{1}$};
	\draw [shorten >=3mm,shorten <=3mm,thick,->] (-8,4)--(0,4)node [pos=0.5, above] {$\psi_{4}$};
	\draw [shorten >=3mm,shorten <=3mm,thick,->] (-4,0)--(4,0)node [pos=0.3, above] {$\psi_{4}$};
	\draw [shorten >=3mm,shorten <=3mm,thick,->] (-8,-4)--(0,-4)node [pos=0.5, above] {$\psi_{4}$};
	\draw [dotted,shorten >=3mm,shorten <=3mm,thick,->] (0,3) -- (0,-3)node [pos=0.7, left] {$\psi_{2}$};
	\draw [dotted,shorten >=3mm,shorten <=3mm,thick,->] (12,3) -- (12,-3)node [pos=0.5, left] {$\psi_{2}$};
	\draw [dotted,shorten >=4mm,shorten <=4mm,thick,->] (4,0) -- (12,-4)node [pos=0.5, below] {$\psi_{3}$};
	\draw [dotted,shorten >=4mm,shorten <=4mm,thick,->] (12,4) -- (4,0)node [pos=0.5, above] {$-\psi_{3}$};
	\draw [dotted,shorten >=4mm,shorten <=4mm,thick,->] (8,0) -- (0,4)node [pos=0.5, above] {$\psi_{3}$};
	\draw [dotted,shorten >=4mm,shorten <=4mm,thick,->] (0,-4) -- (8,0)node [pos=0.5, below] {$\psi_{3}$};
	\end{tikzpicture}
	\]
\end{ex}

\subsection{Case III: $S_{((n-m),(1^m))}$ with $\kappa_2\equiv{\kappa_1-1}\Mod{e}$ and $n\not\equiv{0}\Mod{e}$}

We need only understand the homomorphism $\chi_m$, whose image is generated by $v(2,\dots,m+1)\in S_{((n-m),(1^m))}$, in order to determine the composition factors of $S_{((n-m),(1^m))}$ in this case. Let us first confirm their irreducibility.

\begin{prop}\label{imchi}\label{thm:case3comp}
Suppose that $\kappa_2\equiv{\kappa_1-1}\Mod{e}$ and $n\not\equiv{0}\Mod{e}$. Then $\im(\chi_m)$ and $S_{((n-m),(1^m))}/\im(\chi_m)$ are irreducible $\mathscr{H}_n^{\Lambda}$-modules for all $m\in\{1,\dots,n-1\}$.
\end{prop}

\begin{proof}
We know that $\chi_m$ is injective and $\tau_m$ is surjective by the exact sequence given in part two of \cref{exact}, so that $\im(\chi_m)\cong{S_{((n-m,1^m),\varnothing)}}$ and $\im(\tau_m)\cong{S_{(\varnothing,(1^{n-m+1,1^{m-1}}))}}$. By appealing to the $v$-analogue of Peel's Theorem \cite[Theorem $1(1)$]{CMT}, $S_{(n-m,1^m)}$ and $S_{(n-m+1,1^{m-1})}$ are both irreducible, and hence, so are $S_{((n-m,1^m),\varnothing)}$ and $S_{(\varnothing,(1^{n-m+1,1^{m-1}}))}$.
Thus, $\im(\chi_m)$ and $S_{((n-m),(1^m))}/\im(\chi_m)$ are irreducible, as required.
\end{proof}

Hence, for $1\leqslant{m}\leqslant{n-1}$, it is immediately obvious that $0\subset\im(\chi_m)\subset S_{((n-m),(1^m))}$ is a composition series for $S_{((n-m),(1^m))}$ when $\kappa_2\equiv{\kappa_1-1}\Mod{e}$ and $n\not\equiv{0}\Mod{e}$.

\begin{ex}
	Let $e=3$ and $\kappa=(0,2)$. It follows from \cref{thm:case3comp} that $S_{((3),(1^2))}$ has the composition series
	\[0\subset\im(\chi_2)\subset{S_{((3),(1^2))}}.\]
	By \cref{imker}, the basis elements of $\im(\chi_2)$ correspond to the $((3),(1^2))$-tableaux
	\[\gyoung(345,,!\gr1,!\wh2),\ \gyoung(245,,!\gr1,!\wh3),\ \gyoung(235,,!\gr1,!\wh4),\ \gyoung(234,,!\gr1,!\wh5),\]
	and the basis elements of $S_{((3),(1^2))}/\im(\chi_2)$ correspond to the $((3),(1^2))$-tableaux
	\[\gyoung(!\gr1!\wh45,,2,3),\ \gyoung(!\gr1!\wh35,,2,4),\ \gyoung(!\gr1!\wh34,,2,5),\ \gyoung(!\gr1!\wh25,,3,4),\ \gyoung(!\gr1!\wh24,,3,5),\ \gyoung(!\gr1!\wh23,,4,5)\:.\]
	
	Observe that for any $v_\ttr,v_\tts\in\im(\chi_2)$ and $v_{\ttt},v_{\ttu}\in S_{((3),(1^2))}/\im(\chi_3)$ we can find a directed path from $\ttr$ to $\tts$ and from $\ttt$ to $\ttu$, respectively, as follows, where the basis elements of $S_{((3),(1^2))}$ are represented by the legs of the corresponding $((3),(1^2))$-tableaux.
	
	
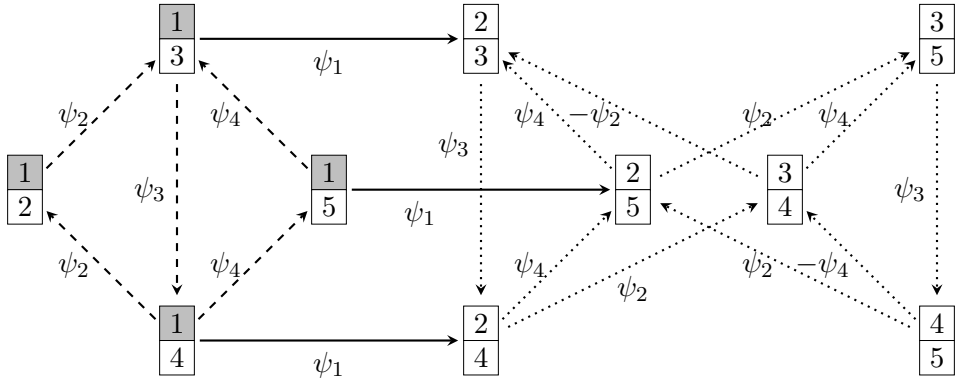
\[
	\begin{tikzpicture}[scale=0.5,>=stealth]
	\fill [] (-12,0) node {$\gyoung(!\gr1,!\wh2)$};
	\fill [] (-8,4) node {$\gyoung(!\gr1,!\wh3)$};
	\fill [] (-8,-4) node {$\gyoung(!\gr1,!\wh4)$};
	\fill [] (-4,0) node {$\gyoung(!\gr1,!\wh5)$};
	\fill [] (0,4) node {$\young(2,3)$};
	\fill [] (0,-4) node {$\young(2,4)$};
	\fill [] (4,0) node {$\young(2,5)$};
	\fill [] (8,0) node {$\young(3,4)$};
	\fill [] (12,4) node {$\young(3,5)$};
	\fill [] (12,-4) node {$\young(4,5)$};
	\draw [dashed,shorten >=1mm,shorten <=1mm,thick,->] (-8,3) -- (-8,-3)node [pos=0.5, left] {$\psi_{3}$};
	\draw [dashed,shorten >=4mm,shorten <=4mm,thick,->] (-12,0) -- (-8,4)node [pos=0.5, left] {$\psi_{2}$};
	\draw [dashed,shorten >=4mm,shorten <=4mm,thick,->] (-8,-4) -- (-12,0)node [pos=0.5, left] {$\psi_{2}$};
	\draw [dashed,shorten >=4mm,shorten <=4mm,thick,->] (-4,0) -- (-8,4)node [pos=0.5, left] {$\psi_{4}$};
	\draw [dashed,shorten >=4mm,shorten <=4mm,thick,->] (-8,-4) -- (-4,0)node [pos=0.5, left] {$\psi_{4}$};
	\draw [dotted,shorten >=1mm,shorten <=1mm,thick,->] (0,3) -- (0,-3)node [pos=0.3, left] {$\psi_{3}$};
	\draw [dotted,shorten >=4mm,shorten <=4mm,thick,->] (0,-4) -- (4,0)node [pos=0.5, left] {$\psi_{4}$};
	\draw [dotted,shorten >=4mm,shorten <=4mm,thick,->] (4,0) -- (0,4)node [pos=0.5, left] {$\psi_{4}$};
	\draw [dotted,shorten >=4mm,shorten <=4mm,thick,->] (8,0) -- (12,4)node [pos=0.5, left] {$\psi_{4}$};
	\draw [dotted,shorten >=4mm,shorten <=4mm,thick,->] (12,-4) -- (8,0)node [pos=0.5, left] {$-\psi_{4}$};
	\draw [dotted,shorten >=1mm,shorten <=1mm,thick,->] (12,3) -- (12,-3)node [pos=0.5, left] {$\psi_{3}$};
	\draw [dotted,shorten >=4mm,shorten <=4mm,thick,->] (0,-4) -- (8,0)node [pos=0.5, below] {$\psi_{2}$};
	\draw [dotted,shorten >=4mm,shorten <=4mm,thick,->] (12,-4) -- (4,0)node [pos=0.5, left] {$\psi_{2}$};
	\draw [dotted,shorten >=4mm,shorten <=4mm,thick,->] (8,0) -- (0,4)node [pos=0.5, left] {$-\psi_{2}$};
	\draw [dotted,shorten >=4mm,shorten <=4mm,thick,->] (4,0) -- (12,4)node [pos=0.5, left] {$\psi_{2}$};
	\draw [shorten >=3mm,shorten <=3mm,thick,->] (-8,4) -- (0,4)node [pos=0.5, below] {$\psi_{1}$};
	\draw [shorten >=3mm,shorten <=3mm,thick,->] (-4,0) -- (4,0)node [pos=0.3, below] {$\psi_{1}$};
	\draw [shorten >=3mm,shorten <=3mm,thick,->] (-8,-4) -- (0,-4)node [pos=0.5, below] {$\psi_{1}$};
	\end{tikzpicture}
	\]	
\end{ex}

\subsection{Case IV: $S_{((n-m),(1^m))}$ with $\kappa_2\equiv{\kappa_1-1}\Mod{e}$ and $n\equiv{0}\Mod{e}$}

The structure of $S_{((n-m),(1^m))}$ is more complicated in this case than the other three cases; each Specht module has either three or four composition factors (except for the irreducible Specht modules $S_{((n),\varnothing)}$ and $S_{(\varnothing,(1^n))}$). We determine the irreducibility of the $\mathscr{H}_n^{\Lambda}$-modules $\im (\phi_m)$	
and $\ker (\gamma_m)/\im (\phi_m)$, which are generated by $v(2,\dots,m,n)\in S_{((n-m),(1^m))}$ and $v(1,\dots,m-1,n)\in S_{((n-m),(1^m))}$, respectively.

\begin{prop}\label{prop:case4}
Suppose that $\kappa_2\equiv{\kappa_1-1}\Mod{e}$ and $n\equiv{0}\Mod{e}$. Then
\begin{enumerate}
	\item $\im (\phi_m)$ is an irreducible $\mathscr{H}_n^{\Lambda}$-module for all $m\in\{1,\dots,n-1\}$, and
	\item $\ker(\gamma_m)/\im(\phi_m)$	
	is an irreducible $\mathscr{H}_n^{\Lambda}$-module for all $m\in\{2,\dots,n-1\}$.
\end{enumerate}
\end{prop}

\begin{proof}
	\begin{enumerate}
		\item{
			It follows from \cref{imker} part $3(c)$ that an arbitrary element of $\im(\phi_m)$ is of the form $v(a_1,\dots,a_{m-1},n)$ with $a_1>1$. Similarly to the proof of \cref{thm:case1comp}, we use \cref{lem:lincomb} and part 2(a) of \cref{cor:irr} to show that $v(a_1,\dots,a_m,n)$ generates $\im(\phi_m)$.}
		\item{
			Suppose that $0\neq v\in\ker(\gamma_m)\backslash\im(\phi_m)$. Then $v\equiv\alpha_1v_{{\ttt}_1}+\dots+\alpha_rv_{{\ttt}_r}\Mod{\im(\phi_m)}$ for some $r\geqslant{1}$, $\alpha_1,\dots,\alpha_r\in\mathbb{F}\backslash\{0\}$ and $v_{{\ttt}_1},\dots,v_{{\ttt}_r}\in\ker(\gamma_m)\backslash\im(\phi_m)$. We proceed by induction on $r$ to show that $v+\im(\phi_m)$ generates $\ker(\gamma_m)/\im(\phi_m)$.
			
			Let $r=1$. We know from \cref{imker} that an arbitrary element of $\ker(\gamma_m)$ is of the form $v(a_1,\dots,a_{m-1},n)$. We thus write $v\equiv \alpha v(a_1,\dots,a_{m-1},n)\Mod{\im(\phi_m)}$ for some ${\alpha}\in\mathbb{F}\backslash\{0\}$ and observe from \cref{imker} part $3(c)$ that $v\not\in\im(\phi_m)$ whenever $a_1=1$. We know that if $a_i=i$ for all $i\in\{1,\dots,m-1\}$, then $v(1,\dots,m,n)$ generates $\ker(\gamma_m)$. Now let $i$ be minimal such that $a_i>i$, for some $i\in\{2,\dots,m-1\}$, and proceed by downwards induction on $i$. By part 2(b) of \cref{cor:irr}, there exists $x\in\scaleobj{1.2}{\Psi}$ such that
			\[x v(1,\dots,i-1,a_i,\dots,a_{m-1},n)\equiv v(1,\dots,i,a_{i+1},\dots,a_{m-1},n)\Mod{\im(\phi_m)}.\]
			By induction, $v(1,\dots,i,a_{i+1},\dots,a_{m-1},n)+\im(\phi_m)$ generates $\ker(\gamma_m)/\im(\phi_m)$, and thus, so does $v+\im(\phi_m)$.
			
			We now let $r=2$, so that $v\equiv \alpha_1 v_{\ttt_1}+\alpha_2 v_{\ttt_2}\Mod{\im(\phi_m)}$. We now define the vector space homomorphism
			\begin{align*}
			\eta:S_{((n-m),(1^{m-1}))}&\longrightarrow S_{((n-m),(1^m))},\\
			v(a_1,\dots,a_{m-1})&\longmapsto v(1,a_1+1,a_2+1,\dots,a_{m-1}+1).
			\end{align*}
			We claim that $\eta$ satisfies $\eta(xv(a_1,\dots,a_{m-1}))=\operatorname{shift}(x)\eta(v(a_1,\dots,a_{m-1}))$ for $x=e(\mathbf{i})$ or $x\in\scaleobj{1.2}{\Psi}$, when $S_{\left((n-m),(1^{m-1})\right)}$ is defined over $\mathscr{H}_{n-1}^{\Lambda}$ with $e$-multicharge $\kappa=(\kappa_1,\kappa_1-2)$.
			For $\mathbf{i}=(i_1,\dots,i_{n-1})\in I^{n-1}$, observe that
			\[
			\operatorname{shift}(e(\mathbf{i}))
			=e(\mathbf{i})^{+1}
			=\sum_{j_1\in I}
			e(j_1,i_1,\dots,i_{n-1}),
			\]
			where $e(j_1,i_1,\dots,i_{n-1})v(1,a_1+1,a_2+1,\dots,a_{m-1}+1)\neq 0$ if and only if $j_1\equiv\kappa_1-1\Mod{e}$. Hence
			\begin{align*}
			\eta(e(\mathbf{i})v(a_1,a_2,\dots,a_{m-1}))
			&=\eta (v(a_1,a_2,\dots,a_{m-1}))\\
			&= v(1,a_1+1,a_2+1,\dots,a_{m-1}+1)\\
			&= e(\kappa_1-1,i_1,\dots,i_{n-1})
			v(1,a_1+1,a_2+1,\dots,a_{m-1}+1)\\
			&=\sum_{j_1\in I}
			e(j_1,i_1,\dots,i_{n-1})
			v(1,a_1+1,a_2+1,\dots,a_{m-1}+1)\\
			&=\operatorname{shift}(e(\mathbf{i}))\eta(v(a_1,a_2,\dots,a_{m-1})).
			\end{align*}
			
			Now observe that $\mathbf{i}_{\ttt_{\left((n-m),(1^m)\right)}}=(\kappa_1-1,\mathbf{i}_{\ttt_{\left((n-m),(1^{m-1})\right)}})$. For all $x=\alpha\psi_{r_1}\psi_{r_2}\dots\psi_{r_k}\in\scaleobj{1.2}{\Psi}$, it thus follows that
			\begin{align*}
			\eta(xv(a_1,a_2,\dots,a_{m-1}))
			&= \eta\left(\alpha\psi_{r_1}\psi_{r_2}\dots\psi_{r_k}\Psid{a_1-1}{1}\Psid{a_2-1}{2}\dots\Psid{a_{m-1}-1}{m-1}z_{((n-m),(1^{m-1}))}\right)\\
			&= \eta\left(\alpha\Psid{b_1-1}{1}\Psid{b_2-1}{2}\dots\Psid{b_{m-1}-1}{m-1}z_{((n-m),(1^{m-1}))}\right)\quad\left(\text{for some $b_i$}\right)\\
			&= \alpha\Psid{b_1}{2}\Psid{b_2}{3}\dots\Psid{b_{m-1}}{m}z_{((n-m),(1^m))}
			\\
			&= \alpha\psi_{r_1+1}\psi_{r_2+1}\dots\psi_{r_k+1}
			\Psid{a_1}{2}\Psid{a_2}{3}\dots\Psid{a_{m-1}}{m}z_{((n-m),(1^m))}\\
			&= \operatorname{shift}(\alpha\psi_{r_1}\psi_{r_2}\dots\psi_{r_k})
			v(1,a_1+1,a_2+1,\dots,a_{m-1}+1)\\
			&= \operatorname{shift}(x)\eta(v(a_1,a_2,\dots,a_{m-1})),
			\end{align*}
			which proves the claim.
			
			Now, we know that $\eta(v_{{\tts}_1})=v_{\ttt_1}$ and $\eta(v_{{\tts}_2})=v_{\ttt_2}$ for some $v_{{\tts}_1},v_{{\tts}_2}\in S_{((n-m),(1^{m-1}))}$.	We know from \cref{lincomb} that there exists $x\in\mathscr{H}_{n-1}^{\Lambda}$ either lying in $\scaleobj{1.2}{\Psi}$ or equal to $e(\mathbf{i})$ for some $\mathbf{i}=(i_1,\dots,i_{n-1})\in{I^{n-1}}$ such that exactly one of $xv_{{\tts}_1}$ and $xv_{{\tts}_2}$ is zero.
			We know from above that $\eta(xv_{{\tts}_1})=\operatorname{shift}(x)v_{{\ttt_1}}$ and $\eta(xv_{{\tts}_2})=\operatorname{shift}(x)v_{{\ttt_2}}$, and moreover, by the injectivity of $\eta$, exactly one of these is non-zero.
			Without loss of generality, we assume that $\operatorname{shift}(x)v_{\ttt_1}\neq 0$ and observe that $\operatorname{shift}(x)v\equiv\alpha_1\operatorname{shift}(x)v_{\ttt_1}\Mod {\im(\phi_m)}$. We know from \cref{rel} that  $\operatorname{shift}(x)$ either lies in $\scaleobj{1.2}{\Psi}$ and is of the form $\alpha\psi_{r_1}\psi_{r_2}\dots\psi_{r_k}$ for some $\alpha\in\mathbb{F}\backslash\{0\}$, $1<r_i<n$ and $1\leqslant i\leqslant k$, or it is equal to an idempotent $e(\kappa_1-1,\mathbf{i})$ for some $\mathbf{i}\in I^{n-1}$, so that $\operatorname{shift}(x)v_{\ttt_1}=v(1,b_{1},\dots,b_{m-1})$ for some $1<b_1<\dots<b_{m-1}\leqslant n$.
			Hence $\operatorname{shift}(x)v_{\ttt_1}\in
			\ker(\gamma_m)\backslash\im(\phi_m)$. We now observe that $\operatorname{shift}(x)v\equiv\alpha_1v(1,b_{1},\dots,b_{m-1})\Mod{\im(\phi_m)}$, and recall from above that $v(1,b_{1},\dots,b_{m-1})+\im(\phi_m)$ generates $\ker(\gamma_m)/\im(\phi_m)$, so that $v+\im(\phi_m)$ generates $\ker(\gamma_m)/\im(\phi_m)$ too.
			
			We now let $r>2$ and suppose that there exists $x\in\mathscr{H}_n^{\Lambda}$, which either lies in $\scaleobj{1.2}{\Psi}$ or is equal to $e(\mathbf{i})$ for some $\mathbf{i}\in I^n$, such that $x(\alpha_1v_{{\ttt}_1}+\dots+\alpha_{r-1}v_{{\ttt}_{r-1}})\equiv \alpha v_{\ttt}\Mod{\im(\phi_m)}$ for some $\alpha\in\mathbb{F}\backslash \{0\}$ and for some $v_{\ttt}\in\ker(\gamma_m)\backslash\im(\phi_m)$. It follows by induction that $xv\equiv \alpha v_{\ttt}+x(\alpha_rv_{\ttt_r})\Mod{\im(\phi_m)}$.
			If we first suppose that $x(\alpha_rv_{\ttt_r})\equiv 0\Mod{\im(\phi_m)}$, then $\alpha v_{\ttt}+\im(\phi_m)$ generates $\ker(\gamma_m)/\im(\phi_m)$ and so must $v+\im(\phi_m)$. Instead suppose that $x(\alpha_rv_{\ttt_r})\equiv\beta v_{\tts}\Mod{\im(\phi_m)}$ for some $\beta\in\mathbb{F}\backslash\{0\}$ and $v_{\tts}\in\ker(\gamma_m)\backslash\im(\phi_m)$, so that $xv\equiv \alpha v_{\ttt}+\beta v_{\tts}\Mod{\im(\phi_m)}$. We know from above with $r=2$ that there exists $x'\in\mathscr{H}_n^{\Lambda}$ such that $x'(xv)\equiv \gamma v_{\ttr}\Mod{\im(\phi_m)}$ for some $\gamma\in\mathbb{F}\backslash\{0\}$ and $v_{\ttr}\in\ker(\gamma_m)\backslash\im(\phi_m)$. Thus $v+\im(\phi_m)$ generates $\ker(\gamma_m)/\im(\phi_m)$ for all $r>2$.	}
		\qedhere
	\end{enumerate}
\end{proof}

\begin{thm}\label{thm:case4comp}
Suppose that $\kappa_2\equiv{\kappa_1-1}\Mod{e}$ and $n\equiv{0}\Mod{e}$.
\begin{enumerate}
	\item{Then $S_{((n-1),(1))}$ has the composition series
		\[
		0
		\subset
		\im (\phi_1)
		\subset
		\im (\chi_1)
		\subset
		S_{((n-1),(1))},
		\]
	which has composition factors
	$S_{((n),\varnothing)}$,
	$\im (\phi_2)$ and
	$\ker (\gamma_2)/\im (\phi_2)$ from bottom to top.}
	\item{Then, for all $m\in\{2,\dots,n-2\}$, $S_{((n-m),(1^m))}$ has the composition series
		\[
		0
		\subset
		\im (\phi_m)
		\subset
		\im (\chi_m)
		\subset
		\ker (\gamma_m)
		+
		\im (\chi_m)
		\subset
		S_{((n-m),(1^m))}
		\]	
	which has composition factors
		$\im (\phi_m)$,
		$\im (\phi_{m+1})$,
		$\ker (\gamma_m)/
		\im (\phi_m)$ and
		$\ker (\gamma_{m+1})/
		\im (\phi_{m+1})$
		from bottom to top.}
	\item{Then $S_{((1),(1^{n-1}))}$ has the composition series
		\[
		0
		\subset
		\im (\phi_{n-1})
		\subset
		\im (\gamma_{n-2})
		\subset
		S_{((1),(1^{n-1}))},
		\]
	which has composition factors 
	$\im (\phi_{n-1})$, 
	$\im (\gamma_{n-2})/\im (\phi_{n-1})$
	and $S_{(\varnothing,(1^n))}$ from bottom to top.}
\end{enumerate}
\end{thm}

\begin{proof}
We use \cref{prop:case4} throughout.
	\begin{enumerate}
		\item{From \cref{lem:pic} we know that $S_{((n-1),(1))}$ has the filtration 
			$\im(\phi_1)\subset\im(\chi_1)\subset S_{((n-1),(1))}$.
			Also, we know from \cref{prop:homs} that $\gamma_1\circ\phi_1=0$, so that by \cref{lem:pic} the middle factor in the filtration of $S_{((n-1),(1))}$ is
			\begin{align*}
			\im(\chi_1)/\im(\phi_1)
			\cong\im(\gamma_1\circ\chi_1)/\im(\gamma_1\circ\phi_1)
			=\im(\gamma_1\circ\chi_1)
			=\im(\phi_2).
			\end{align*}
			By \cref{imker}, we have $\ker(\tau_1)=\spn\{v_{\ttt}\ |\ \ttt\in\std((n-1),(1)),\ttt(1,1,1)=1\}$. Now, by using \cref{gamiso}, the top factor in the filtration of $S_{((n-1),(1))}$ is given by
			\begin{align*}
			 S_{((n-1),(1))}/\im(\chi_1) 
			=	\left\{v(1)\in S_{((n-1),(1))}\right\}
			&\cong	\left\{\gamma_1\left(v(1)\right)\ \mid v(1)\in S_{((n-1),(1))}\right\}\\
			&=	\left\{v(1,n)\in S_{((n-2),(1^2))}\right\}\\
			&=\ker(\gamma_2)/\im(\phi_2).
			\end{align*}}
		\item{Let $2\leqslant{m}\leqslant{n-2}$.
			By \cref{lem:pic}, we know that $S_{((n-m),(1^m))}$ has the filtration $\im(\phi_m)\subset\im(\chi_m)\subset\ker(\gamma_m)+\im(\chi_m)\subset	S_{((n-m),(1^m))}$.	We know from \cref{prop:homs} that $\gamma_m\circ\phi_m=0$. Thus, together with \cref{lem:pic}, the second from bottom factor in the filtration of $S_{((n-m,1^m))}$ is
			\begin{align*}
			\im(\chi_m)/\im(\phi_m)
			\cong\im(\gamma_m\circ\chi_m)/\im(\gamma_m\circ\phi_m)
			=\im(\gamma_m\circ\chi_m)
			&=\im(\phi_{m+1}).
			\end{align*}
			Using \cref{lem:pic}, the second from top factor in the filtration of $S_{((n-m,1^m))}$ is
			\begin{align*}
			\left(\ker(\gamma_m)+\im(\chi_m)/\im(\chi_m)\right)
			&=\ker(\gamma_m)/\left(\ker(\gamma_m)\cap\im(\chi_m)\right)\\
			&=\ker(\gamma_m)/\left(\im(\gamma_{m-1})\cap\im(\chi_m)\right)\\
			&=\ker(\gamma_m)/\im(\phi_{m}).
			\end{align*}
			Finally, using \cref{lem:pic}, the top factor in the filtration of $S_{((n-m,1^m))}$ is
			\begin{align*}
			S_{((n-m,1^m))}/\left(\ker(\gamma_m)+\im(\chi_m)\right)
			&\cong\left(\im(\gamma_m)+\im(\chi_{m+1})\right)/\im(\chi_{m+1})\\
			&=\left(\ker(\gamma_{m+1})+\im(\chi_{m+1})\right)/\im(\chi_{m+1})\\
			&=\ker(\gamma_{m+1})/\left(\ker(\gamma_{m+1})\cap\im(\chi_{m+1})\right)\\
			&=\ker(\gamma_{m+1})/\im(\phi_{m+1}).
			\end{align*}}
		\item{It is clear from \cref{lem:pic} that $\im(\phi_{n-1})\subset\im(\gamma_{n-2})\subset S_{((1),(1^{n-1}))}$ is a filtration of $S_{((1),(1^{n-1}))}$.
			Also by \cref{lem:pic}, the top factor of $S_{((1)(1^{n-1})}$ is
			\begin{align*}
			S_{((1),(1^{n-1})}/\im(\gamma_{n-2})
			&=\left(\ker(\gamma_{n-1})+\im(\gamma_{n-1})\right)/\im(\gamma_{n-2})\\
			&=\left(\ker(\gamma_{n-1})+\im(\gamma_{n-1})\right)/\ker(\gamma_{n-1})\\
			&=\im(\gamma_{n-1})/\ker(\gamma_{n-1})\cap\im(\gamma_{n-1})\\
			&=\im(\gamma_{n-1})\\
			&\cong S_{(\varnothing,(1^n))}.
			\end{align*}}
		\qedhere
	\end{enumerate}
\end{proof}

\begin{ex}\label{ex:comp4}
	Let $e=3$ and $\kappa=(0,2)$.
	Then $S_{((3),(1^3))}$ has the composition series
	\[0\subset\im(\phi_3)\subset\im(\chi_3)\subset\ker(\gamma_3)+\im(\chi_3)\subset S_{((3),(1^3))},\]
	where
	\begin{itemize}
		\item $\im(\chi_3)/\im(\phi_3)\cong\im(\phi_4)$,
		\item $\left(\ker(\gamma_3)+\im(\chi_3)\right)/\im(\chi_3)\cong\ker(\gamma_3)/\im(\phi_3)$, 
		\item $S_{((3),(1^3))}/\left(\ker(\gamma_3)+\im(\chi_3)\right)\cong\ker(\gamma_{4})/\im(\phi_{4})$.
	\end{itemize}
	By \cref{imker}, the basis elements of $\im(\phi_3)=\left\langle \psi_1\psi_2\psi_5\psi_4\psi_3z_{((3),(1^3))}\right\rangle$ correspond to the $((3),(1^3))$-tableaux
	\[\gyoung(!\gr1!\wh45,,2,3,!\gr6),\ \gyoung(!\gr1!\wh35,,2,4,!\gr6),\ \gyoung(!\gr1!\wh25,,3,4,!\gr6),\ \gyoung(!\gr1!\wh34,,2,5,!\gr6),\ \gyoung(!\gr1!\wh24,,3,5,!\gr6),\ \young(!\gr1!\wh23,,4,5,!\gr6)\:.\]
	We have $\im(\chi_3)=\left\langle \psi_1\psi_2\psi_3z_{((3),(1^3))}\right\rangle$, so that the basis elements of $\im(\phi_4)$ correspond to
	\[\gyoung(!\gr1!\wh5!\gr6,,!\wh2,3,4),\ \gyoung(!\gr1!\wh4!\gr6,,!\wh2,3,5),\ \gyoung(!\gr1!\wh3!\gr6,,!\wh2,4,5),\ \gyoung(!\gr1!\wh2!\gr6,,!\wh3,4,5)\:.\]
	We have $\ker(\gamma_3)+\im(\chi_3)=\left\langle \psi_5\psi_4\psi_3z_{((3),(1^3))}\right\rangle$, so that the basis elements of $\ker(\gamma_3)/\im(\phi_3)$ correspond to
	\[\gyoung(345,,!\gr1,!\wh2,!\gr6),\ \gyoung(245,,!\gr1,!\wh3,!\gr6),\ \gyoung(235,,!\gr1,!\wh4,!\gr6),\ \gyoung(234,,!\gr1,!\wh5,!\gr6)\:.\]
	We have $S_{((3),(1^3))}=\left\langle z_{((3),(1^3))}\right\rangle$, so that the basis elements of $\ker(\gamma_{4})/\im(\phi_{4})$ correspond to
	\[\gyoung(45!\gr6,,1,!\wh2,3),\ \gyoung(35!\gr6,,1,!\wh2,4),\ \gyoung(34!\gr6,,1,!\wh2,5),\ \gyoung(25!\gr6,,1,!\wh3,4), \gyoung(24!\gr6,,1,!\wh3,5), \gyoung(23!\gr6,,1,!\wh4,5)\:.\]
	
	Observe that for any $v_\ttr,v_\tts\in\im(\chi_3)$, $v_{\ttt},v_{\ttu}\in\im(\phi_4)$, $v_{\ttw},v_{\ttx}\in\ker(\gamma_3)/\im(\gamma_3)$ and $v_{\tty},v_{\ttz}\in\ker(\gamma_4)/\im(\gamma_4)$ we can find a directed path from $\ttr$ to $\tts$, from $\ttt$ to $\ttu$, from $\ttw$ to $\ttx$ and from $\tty$ to $\ttz$, respectively, as follows, where the basis elements of $S_{((3),(1^3))}$ are represented by the legs of the corresponding $((3),(1^3))$-tableaux.
	
	
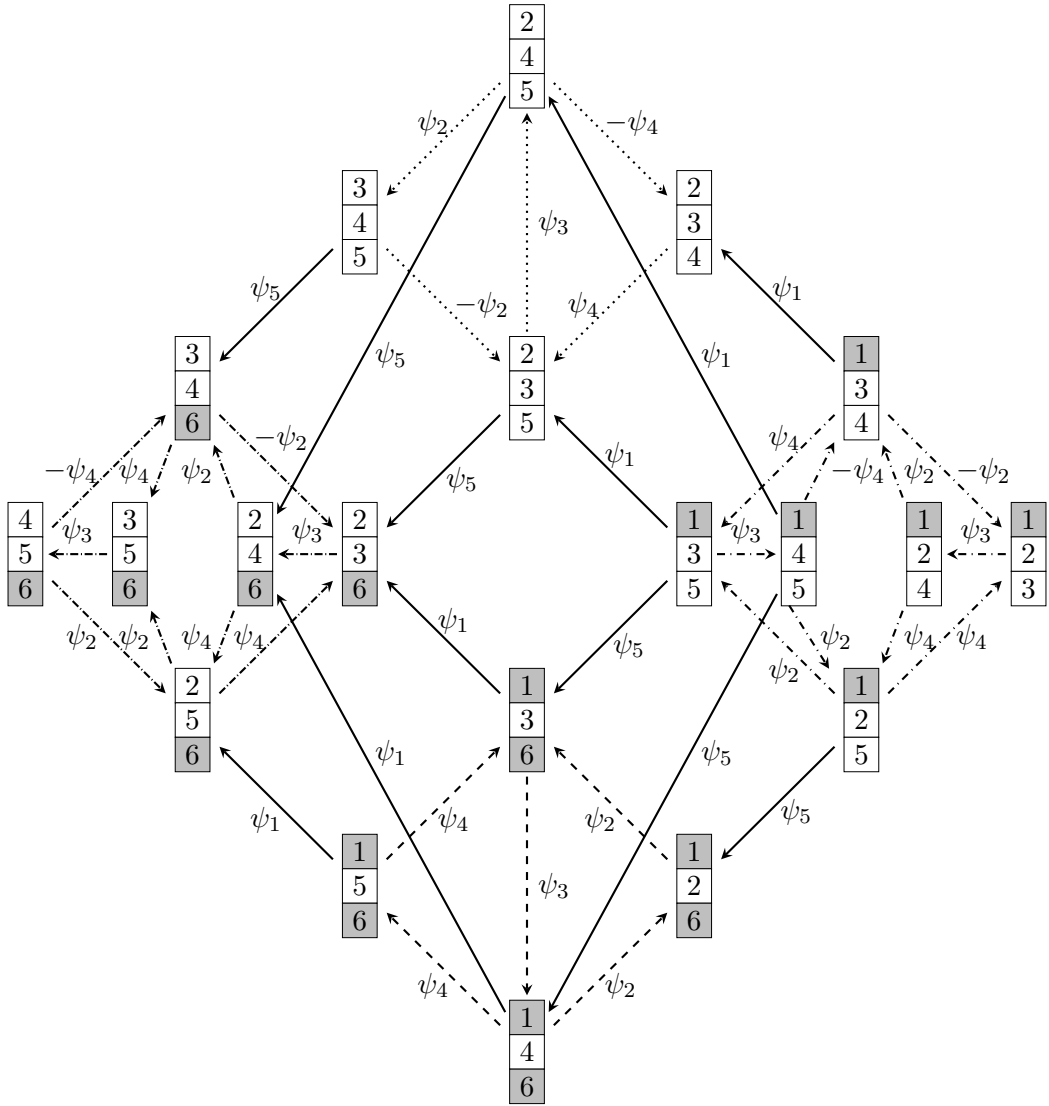
\[
	\begin{tikzpicture}[scale=0.55,>=stealth]
	\fill [] (-12,0) node {$\gyoung(4,5,!\gr6)$};
	\fill [] (-9.5,0) node {$\gyoung(3,5,!\gr6)$};
	\fill [] (-6.5,0) node {$\gyoung(2,4,!\gr6)$};
	\fill [] (-4,0) node {$\gyoung(2,3,!\gr6)$};
	\fill [] (4,0) node {$\gyoung(!\gr1,!\wh3,5)$};
	\fill [] (6.5,0) node {$\gyoung(!\gr1,!\wh4,5)$};
	\fill [] (9.5,0) node {$\gyoung(!\gr1,!\wh2,4)$};
	\fill [] (12,0) node {$\gyoung(!\gr1,!\wh2,3)$};
	\fill [] (0,4) node {$\young(2,3,5)$};
	\fill [] (0,12) node {$\young(2,4,5)$};
	\fill [] (4,8) node {$\young(2,3,4)$};
	\fill [] (-4,8) node {$\young(3,4,5)$};
	\fill [] (-8,4) node {$\gyoung(3,4,!\gr6)$};
	\fill [] (8,4) node {$\gyoung(!\gr1,!\wh3,4)$};
	\fill [] (-8,-4) node {$\gyoung(2,5,!\gr6)$};
	\fill [] (8,-4) node {$\gyoung(!\gr1,!\wh2,5)$};
	\fill [] (0,-4) node {$\gyoung(!\gr1,!\wh3,!\gr6)$};
	\fill [] (4,-8) node {$\gyoung(!\gr1,!\wh2,!\gr6)$};
	\fill [] (-4,-8) node {$\gyoung(!\gr1,!\wh5,!\gr6)$};
	\fill [] (0,-12) node {$\gyoung(!\gr1,!\wh4,!\gr6)$};
	\draw [densely dashdotted,shorten >=5mm,shorten <=5mm,thick,->] (-12,0) -- (-8,4)node [pos=0.5, left] {$-\psi_{4}$};
	\draw [densely dashdotted,shorten >=3mm,shorten <=3mm,thick,->] (-9.5,0) -- (-12,0)node [pos=0.5, above] {$\psi_{3}$};
	\draw [densely dashdotted,shorten >=3mm,shorten <=3mm,thick,->] (-4,0) -- (-6.5,0)node [pos=0.5, above] {$\psi_{3}$};
	\draw [densely dashdotted,shorten >=8mm,shorten <=8mm,thick,->] (-8,4) -- (-9.5,0)node [pos=0.5, left] {$\psi_{4}$};
	\draw [densely dashdotted,shorten >=8mm,shorten <=8mm,thick,->] (-6.5,0) -- (-8,4)node [pos=0.5, left] {$\psi_{2}$};
	\draw [densely dashdotted,shorten >=5mm,shorten <=5mm,thick,->] (-8,4) -- (-4,0)node [pos=0.3, right] {$-\psi_{2}$};
	\draw [densely dashdotted,shorten >=8mm,shorten <=8mm,thick,->] (-6.5,0) -- (-8,-4)node [pos=0.5, left] {$\psi_{4}$};
	\draw [densely dashdotted,shorten >=5mm,shorten <=5mm,thick,->] (-8,-4) -- (-4,0)node [pos=0.5, left] {$\psi_{4}$};
	\draw [densely dashdotted,shorten >=8mm,shorten <=8mm,thick,->] (-8,-4) -- (-9.5,0)node [pos=0.5, left] {$\psi_{2}$};
	\draw [densely dashdotted,shorten >=5mm,shorten <=5mm,thick,->] (-12,0) -- (-8,-4)node [pos=0.5, left] {$\psi_{2}$};
	\draw [dashdotted,shorten >=8mm,shorten <=8mm,thick,->] (9.5,0) -- (8,4)node [pos=0.5, right] {$\psi_{2}$};
	\draw [dashdotted,shorten >=5mm,shorten <=5mm,thick,->] (8,4) -- (12,0)node [pos=0.5, right] {$-\psi_{2}$};
	\draw [dashdotted,shorten >=3mm,shorten <=3mm,thick,->] (12,0) -- (9.5,0)node [pos=0.5, above] {$\psi_{3}$};
	\draw [dashdotted,shorten >=3mm,shorten <=3mm,thick,->] (4,0) -- (6.5,0)node [pos=0.5, above] {$\psi_{3}$};
	\draw [dashdotted,shorten >=5mm,shorten <=5mm,thick,->] (8,4) -- (4,0)node [pos=0.3, left] {$\psi_{4}$};
	\draw [dashdotted,shorten >=8mm,shorten <=8mm,thick,->] (6,0) -- (8,4)node [pos=0.5, right] {$-\psi_{4}$};
	\draw [dashdotted,shorten >=5mm,shorten <=5mm,thick,->] (8,-4) -- (4,0)node [pos=0.3, left] {$\psi_{2}$};
	\draw [dashdotted,shorten >=5mm,shorten <=5mm,thick,->] (8,-4) -- (12,0)node [pos=0.5, right] {$\psi_{4}$};
	\draw [dashdotted,shorten >=8mm,shorten <=8mm,thick,->] (9.5,0) -- (8,-4)node [pos=0.5, right] {$\psi_{4}$};
	\draw [dashdotted,shorten >=8mm,shorten <=8mm,thick,->] (5.5,0) -- (8,-4)node [pos=0.5, right] {$\psi_{2}$};
	\draw [dotted,shorten >=2mm,shorten <=2mm,thick,->] (0,5) -- (0,11)node [pos=0.5, right] {$\psi_{3}$};
	\draw [dotted,shorten >=5mm,shorten <=5mm,thick,->] (0,12) -- (4,8)node [pos=0.4, right] {$-\psi_{4}$};
	\draw [dotted,shorten >=5mm,shorten <=5mm,thick,->] (4,8) -- (0,4)node [pos=0.5, left] {$\psi_{4}$};
	\draw [dotted,shorten >=5mm,shorten <=5mm,thick,->] (-4,8) -- (0,4)node [pos=0.5, right] {$-\psi_{2}$};
	\draw [dotted,shorten >=5mm,shorten <=5mm,thick,->] (0,12) -- (-4,8)node [pos=0.4, left] {$\psi_{2}$};
	\draw [dashed,shorten >=5mm,shorten <=5mm,thick,->] (4,-8) -- (0,-4)node [pos=0.4, left] {$\psi_{2}$};
	\draw [dashed,shorten >=5mm,shorten <=5mm,thick,->] (-4,-8) -- (0,-4)node [pos=0.4, right] {$\psi_{4}$};
	\draw [dashed,shorten >=5mm,shorten <=5mm,thick,->] (0,-12) -- (-4,-8)node [pos=0.4, left] {$\psi_{4}$};
	\draw [dashed,shorten >=5mm,shorten <=5mm,thick,->] (0,-12) -- (4,-8)node [pos=0.4, right] {$\psi_{2}$};
	\draw [dashed,shorten >=2mm,shorten <=2mm,thick,->] (0,-5) -- (0,-11)node [pos=0.5, right] {$\psi_{3}$};
	\draw [shorten >=6mm,shorten <=6mm,thick,->] (6.5,0) -- (0,12)node [pos=0.4, right] {$\psi_{1}$};
	\draw [shorten >=6mm,shorten <=6mm,thick,->] (0,12) -- (-6.5,0)node [pos=0.6, right] {$\psi_{5}$};
	\draw [shorten >=6mm,shorten <=6mm,thick,->] (6.5,0) -- (0,-12)node [pos=0.4, right] {$\psi_{5}$};
	\draw [shorten >=6mm,shorten <=6mm,thick,->] (0,-12) -- (-6.5,0)node [pos=0.6, right] {$\psi_{1}$};
	\draw [shorten >=5mm,shorten <=5mm,thick,->] (8,4) -- (4,8)node [pos=0.6, right] {$\psi_{1}$};
	\draw [shorten >=5mm,shorten <=5mm,thick,->] (4,0) -- (0,4)node [pos=0.6, right] {$\psi_{1}$};
	\draw [shorten >=5mm,shorten <=5mm,thick,->] (0,-4) -- (-4,0)node [pos=0.6, right] {$\psi_{1}$};
	\draw [shorten >=5mm,shorten <=5mm,thick,->] (-4,-8) -- (-8,-4)node [pos=0.4, left] {$\psi_{1}$};
	\draw [shorten >=5mm,shorten <=5mm,thick,->] (-4,8) -- (-8,4)node [pos=0.4, left] {$\psi_{5}$};
	\draw [shorten >=5mm,shorten <=5mm,thick,->] (0,4) -- (-4,0)node [pos=0.55, right] {$\psi_{5}$};
	\draw [shorten >=5mm,shorten <=5mm,thick,->] (4,0) -- (0,-4)node [pos=0.55, right] {$\psi_{5}$};
	\draw [shorten >=5mm,shorten <=5mm,thick,->] (8,-4) -- (4,-8)node [pos=0.55, right] {$\psi_{5}$};
	\end{tikzpicture}
	\]
\end{ex}

We have thus established the composition series, up to isomorphism, of Specht modules labelled by hook bipartitions with quantum characteristic $e\in\{3,4,\dots\}$. To completely determine the rows of the decomposition matrix for $\mathscr{H}_n^{\Lambda}$ labelled by hook bipartitions, we need to deduce the non-isomorphic composition factors. In subsequent work, we compute the irreducible labels of these factors in order to obtain the corresponding decomposition numbers, together with their graded analogues.

\section*{Acknowledgements}
This paper was written under the guidance of the author's PhD supervisor, Matthew Fayers, at Queen Mary University of London. The author would like to thank Dr Fayers for his utmost patience and invaluable counsel, without which, this paper would not have come to fruition. The author is also grateful to the referee for their many helpful comments and corrections. This research was supported by the Engineering Physical Sciences Research Council.

\bibliographystyle{plain}
%\bibliography{biblio}

\end{document}